\documentclass[11pt]{amsart}
\usepackage{amsmath,amsxtra,amssymb,amsthm,amsfonts}
\usepackage{color}

\newtheorem{main}{Theorem}

\newtheorem{mcor}[main]{Corollary}

\numberwithin{equation}{section}
\newtheorem{lemma}{Lemma}[section]
\newtheorem{prop}[lemma]{Proposition}
\newtheorem{theorem}[lemma]{Theorem}

\newtheorem{cor}[lemma]{Corollary}

\newtheorem{rem}[lemma]{Remark}

\newcommand{\re}{\begin{rem}\rm}
  \newcommand{\mar}{\end{rem}}

\newtheorem{defi}[lemma]{Definition}

\renewcommand{\for}{\begin{eqnarray*}}

\newcommand{\mel}{\end{eqnarray*}}

\newcommand{\kl}{\pl \le \pl}
\newcommand{\gl}{\pl \ge \pl}

\newcommand{\lel}{\pl = \pl}

\newcommand{\ez}{{\mathbb E}}

\newcommand{\nz}{{\mathbb N}}
\newcommand{\nen}{n \in \nz}

\newcommand{\rz}{{\mathbb R}}
\newcommand{\zz}{{\mathbb Z}}

\newcommand{\cz}{{\mathbb C}}

\newcommand{\ten}{\otimes}

\newcommand{\p}{\hspace{.05cm}}
\newcommand{\pl}{\hspace{.1cm}}

\newcommand{\qd}{\end{proof}\vspace{0.5ex}}

\newcommand{\Om}{\Omega}
\newcommand{\om}{\omega}

\newcommand{\al}{\alpha}

\newcommand{\si}{\sigma}

\renewcommand{\H}{{\mathcal H}}
\newcommand{\Si}{\Sigma}

\newcommand{\la}{\lambda}
\newcommand{\eps}{\varepsilon}

\newcommand{\F}{{\mathcal F}}
\newcommand{\E}{{\mathcal E}}
\newcommand{\A}{{\mathcal A}}

\newcommand{\K}{{\mathcal K}}

\newcommand{\N}{{\mathcal N}}

\newcommand{\pf}{\begin{proof}}

\newcommand{\xspace}{\hbox{\kern-2.5pt}}
\newcommand{\xyspace}{\hbox{\kern-1.1pt}}

\newcommand{\lan}{\langle}
\newcommand{\ran}{\rangle}

\allowdisplaybreaks

\definecolor{LightGray}{rgb}{0.94,0.94,0.94}
\definecolor{VeryLightBlue}{rgb}{0.9,0.9,1}
\definecolor{LightBlue}{rgb}{0.8,0.8,1}
\definecolor{DarkBlue}{rgb}{0,0,0.6}
\definecolor{LightGreen}{rgb}{0.88,1,0.88}
\definecolor{MidGreen}{rgb}{0.6,1,0.6}
\definecolor{DarkGreen}{rgb}{0,0.6,0}
\definecolor{DarkGrreen}{rgb}{0,0.8,0}

\definecolor{VeryLightYellow}{rgb}{1,1,0.9}
\definecolor{LightYellow}{rgb}{1,1,0.6}
\definecolor{MidYellow}{rgb}{1,1,0.5}
\definecolor{DarkYellow}{rgb}{0.8,1,0.3}
\definecolor{VeryLightRed}{rgb}{1,0.9,0.9}
\definecolor{LightRed}{rgb}{1,0.8,0.8}
\definecolor{DarkRed}{rgb}{0.8,0.2,0}
\definecolor{DarkRedb}{rgb}{0.6,0.2,0}
\definecolor{DarkLila}{rgb}{0.8,0,1}
\definecolor{Beige}{rgb}{0.96,0.96,0.86}
\definecolor{Gold}{rgb}{1.,0.84,0.}
\definecolor{Goldb}{rgb}{0.7,0.3,0.5}
\definecolor{MyYellow}{rgb}{1.,0.84,0.8}

\begin{document}

\title[]{Generalized $q$-gaussian von Neumann algebras with coefficients, I. Relative  strong solidity}

\author[Marius Junge]{Marius Junge}
\address{Department of Mathematics\\
University of Illinois, Urbana, IL 61801, USA} 
\email[Marius Junge]{junge@math.uiuc.edu}

\author[Bogdan Udrea]{Bogdan Udrea}
\address{Department of Mathematics\\
University of Iowa, Iowa City, IA 52242, USA and Institute of Mathematics ``Simion Stoilow'' of the Romanian Academy, P.O. Box 1-764, Bucharest, Romania} 
\email[Bogdan Udrea]{bogdanteodor-udrea@uiowa.edu}

\begin{abstract} We define $\Gamma_q(B,S \ten H)$, the generalized $q$-gaussian von Neumann algebras associated to a sequence of symmetric independent copies $(\pi_j,B,A,D)$ and to a subset $1 \in S = S^* \subset A$ and, under certain assumptions, prove their strong solidity relative to $B$. We provide many examples of strongly solid generalized $q$-gaussian von Neumann algebras. We also obtain non-isomorphism and non-embedability results about some of these von Neumann algebras.
\end{abstract}

\maketitle
\section{Background and statement of results}

\subsection{Normalizers in von Neumann algebras}
The study of normalizers in von Neumann algebras is nowadays an intensely active area of research within the field of von Neuman algebras. For a von Neumann algebra $M$, we denote by $\mathcal U(M)$ the group of unitaries in $M$. Recall that for an inclusion $A\subset M$ of von Neumann algebras, \emph{the normalizing group of} $A$ in $M$ is defined as $\mathcal N_M(A)=\{u\in \mathcal U(M)\mid uAu^*=A\}$ and \emph{the normalizer of} $A$ is the von Neumann algebra generated by the normalizing group, i.e. $\mathcal N_M(A)''\subset M$. When $A$ is a maximal abelian von Neumann subalgebra of a type II$_1$ factor $M$, $A$ is called a \emph{Cartan subalgebra} if its normalizer is the whole of $M$. While some results have been obtained in the early 80's (see e.g. \cite{CFW, CJ, JP}, notably \cite{CFW}), the first truly significant achievement in this area is Voiculescu's ground-breaking result about the absence of Cartan subalgebras in the free-group factors $L(\mathbb F_n)$ \cite{Voi96}. After this, in their seminal work \cite{OPCartanI}, Ozawa and Popa established the following result:
\begin{theorem}
Let $\mathbb F_n \curvearrowright B$ be a \emph{profinite} trace-preserving action of a free group on an amenable von Neumann algebra $B$. Then for every amenable von Neumann subalgebra $A\subset M=B\rtimes \mathbb F_n$, either $A \prec_M B$, or the normalizer of $A$ is amenable.
\end{theorem}
The notation $A \prec_M B$ reads ''a corner of $A$ embeds into $B$ inside $M$", in the sense of  Popa (see \cite{PoI}, Thm. 2.1), and it roughly means that $A$ can be conjugated into $B$ by a partial isometry in $M$. When $B$ is the scalars, this shows that the normalizer of any amenable diffuse von Neumann subalgebra of $L(\mathbb F_n)$ is itself amenable, not only reproving and strengthening Voiculescu's result, but also entailing a surprizingly far-reaching classification of normalizer algebras in the free group factors. More than merely proving the above theorem, \cite{OPCartanI} introduced an array of innovative ideas and techniques which remain all-pervasive and highly influential in the field to this day. The results in \cite{OPCartanI} were then extended to profinite actions of weakly amenable groups having proper 1-cocycles into (a multiple of) their left regular representations in \cite{OPCartanII}. Subsequent generalizations to the case of profinite actions of groups having quasi-cocycles or direct products of such have been obtained in \cite{CS} and \cite{CSU}. Recently, Popa and Vaes obtained the definitive form of these results, by completely removing any assumption on the action of the group. Specifically, they proved the following results (see Thm. 1.6 in \cite{PoVaI} and Thm. 1.4 in \cite{PoVaII}):
\begin{theorem} Let $\Gamma$ be a weakly amenable group having either a {\bf proper} 1-cocycle or a {\bf proper} 1-quasi-cocycle into a (representation which is weakly contained into) a multiple of its left regular representation. Let $\Gamma \curvearrowright B$ be any trace-preserving action of $\Gamma$ on the finite von Neumann algebra $B$, and let $A \subset M=B \rtimes \Gamma$ be a von Neumann subalgebra which is amenable relative to $B$ inside $M$. Then either $A \prec_M B$, or the normalizer of $A$ is amenable relative to $B$ inside $M$.
\end{theorem}
\begin{theorem}
Let $\Gamma \curvearrowright B$ be a p.m.p. free ergodic action, where $B$ is abelian diffuse and $\Gamma$ is weakly amenable and admits an {\bf unbounded} (rather than proper) 1-cocycle into a mixing representation which is weakly contained into a multiple of the left regular representation of $\Gamma$. Then $M=B\rtimes \Gamma$ has a unique Cartan subalgebra, up to unitary conjugacy.
\end{theorem}
Popa and Vaes coined the phrase "relative strong solidity" to describe the situation in which the dichotomy in Thm. 1.2 holds. Namely, a von Neumann algebra $M$ is \emph{strongly solid relative to $B$}, for $B \subset M$ a subalgebra, if for every von Neumann subalgebra $A \subset M$ which is amenable relative to $B$ inside $M$ (see \cite{OPCartanI}), it is either the case that $A \prec_M B$ or that the normalizer of $A$ is amenable relative to $B$ inside $M$. In the case of $B$ abelian diffuse and of p.m.p. free ergodic actions $\Gamma \curvearrowright B$, the strong solidity of the von Neumann algebra $M=B \rtimes \Gamma$ relative to $B$ implies its uniqueness of Cartan subalgebra, up to unitary conjugacy. Strong solidity relative to the scalars is simply termed strong solidity. Strong solidity is in turn an enhancement of Ozawa's concept of \emph{solidity} (see \cite{OzawaSolid}). Ozawa called a von Neumann algebra $M$ solid if for every diffuse von Neumann subalgebra $\A \subset M$ one has that $\A' \cap M$ is amenable. It's easy to see that a non-amenable solid factor $M$ is automatically \emph{prime}, i.e. cannot be written as $M=M_1 \bar{\ten} M_2$, with $M_i$ an infinite dimensional factor for $i=1,2$. 
\par Further results pertaining to the classification of normalizers and relative strong solidity have been obtained by Sinclair (\cite{Si11}), Ioana (\cite{Ioana15}), Isono (\cite{Isono1,Isono2}), Avsec (\cite{Avsec}), Boutonnet, Houdayer and Vaes (\cite{BoHouVa, HouVa}), Caspers (\cite{Ca18}).

\subsection{Non-commutative probability}
Voiculescu introduced his highly influential free probability theory in the early 80's (see \cite{Voi83}), in order to tackle some problems related to the free group factors. Since then, the free probability theory has grown into an immense industry with far reaching ramifications. Very roughly speaking, in the realm of free probability classical probability spaces are replaced by C$^*$ or W$^*$-algebras endowed with distinguished states (normal in the W$^*$ case), classical random variables by operators in those algebras, classical independence by Voiculescu's free independence, and the classical distribution function by Voiculescu non-commutative distribution of a non-commutative random variable, or joint distribution in the case of a system of random variables. In particular, the normal (gaussian) distribution is replaced by Wigner's semicircular law.\\
1.2.1. {\bf Classical gaussian random variables.}
We briefly recall the construction in Section 1.1 of \cite{PetersonSinclair}. Let $H$ a real Hilbert space and, for $\xi\in H$, let $l_\xi$ be the creation operator on the symmetric Fock space of $H_\cz=H\oplus iH$. Then $s_1(\xi)=\frac{1}{2}(l_\xi+l_\xi^*)$ is an unbounded self-adjoint operator in the symmetric Fock space. The operators $s_1(\xi)$ and $s_1(\eta)$ commute for all $\xi$ and $\eta$ and are independent with respect to the vacuum state whenever $\lan \xi,\eta\ran=0$. Define $\Gamma_1(H)$ to be the abelian von Neumann algebra generated by the spectral projections of all the $s_1(\xi), \xi\in H$ (or equivalently by all the unitaries $\om(\xi_1,\ldots,\xi_k)=\exp(i\pi s(\xi_1)\cdots s(\xi_k))$), equipped with the trace given by the restriction of the vacuum state. For $\|\xi\|=1$, we have the following moment formula
\[\tau(s(\xi)^m)=\delta_{m\in 2\nz}\frac{m!}{2^{\frac{m}{2}}(\frac{m}{2})!}=|P_2(m)|=\sum_{\si \in P_2(m)} 1^{\rm cr(\si)},\]
where $P_2(m)$ is the collection of pair partitions on the set $\{1,\ldots,m\}$, and for $\si\in P_2(m)$, ${\rm cr(\si)}$ denotes the number of crossings of $\si$. These are exactly the moments of a classical gaussian random variable. By commutativity, independence and multi-linearity, the moment formula can be extended to 
\[\tau(s_1(\xi_1)\cdots s_1(\xi_m))=\sum_{\si\in P_2(m)} 1^{\rm cr(\si)}\prod_{\{l,r\}\in\si}\lan \xi_l,\xi_r\ran.\]
One also recalls
\begin{theorem}[Classical central limit theorem]
Let $\{X_n\}_{n\geq 1}$ be a sequence of independent, identically distributed random variables on a probability space $(\Om,\Si,P)$, all having mean equal to zero and variance equal to 1. Then the averages $S_n=n^{-\frac{1}{2}}\sum_{j=1}^n X_j$ converge in distribution to a normal (gaussian) random variable with mean zero and variance 1.
\end{theorem}
If $X_n$ are chosen such that $\sup_{n\geq 1}\|S_n\|_{\infty}<\infty$, then one can restate the central limit theorem by saying that the element $S=(S_n)_n\in (L^{\infty}(\Om,P)^{\om},\tau_{\om})$ has a gaussian (normal) distribution, where $\om$ is a free ultrafilter on $\nz$ and $(L^{\infty}(\Om)^{\om},\tau_{\om})$ is the ultraproduct von Neumann algebra. In other words, one could ''simulate" gaussian elements using an ultraproduct model.\\
1.2.2. {\bf Voiculescu's free semicircular random variables.} \\
In \cite{Voi83}, Voiculescu constructed a functor $\Phi$ from the category of real Hilbert spaces with contractions to the category of finite von Neumann algebras with completely positive maps. For $h\in H$, the element $\Phi(h)=s_0(h)$ is concretely realized as the real part of the creation operator on the full Fock space of $H_\cz$. Moreover, he proved that for an orthonormal set $\{h_1,\ldots,h_m\}\subset H$, the elements $s_0(h_1),\ldots,s_0(h_m)$ are freely independent, have semi-circular distributions given by $d\mu(t)=\frac{1}{\pi}\chi_{(-1,1)}(t)\sqrt{1-t^2}dt$, and generate a copy of the free group factor $L(\mathbb F_m)$. In particular, for a finite dimensional Hilbert space $H$, $\Phi(H)$ is *-isomorphic to the free group factor $L(\mathbb F_{{\rm dim}(H)})$. It is well-known that the moments of a semi-circular variable are given by
\[\tau(s_0(h)^m)=\delta_{m\in 2\nz}\frac{m!}{(\frac{m}{2}+1)((\frac{m}{2})!)^2}=\sum_{\si \in P_2(m)}0^{\rm cr(\si)}=\sum_{\si \in NCP_2(m)} 0^{\rm cr(\si)},\]
where we denote by $NCP_2(m)$ the collection of non-crossing pair partitions on the set $\{1,\ldots,m\}$ and with the convention $0^0=1$. By direct computation, the above formula can be extended to
\[\tau(s_0(h_1)\cdots s_0(h_m))=\sum_{\si\in P_2(m)} 0^{\rm cr(\si)}\prod_{\{l,r\}\in\si} \lan h_l,h_r\ran.\]
Let's recall the Voiculescu's central limit theorem in \cite{Voi83}:
\begin{theorem}[Voiculescu's central limit theorem]
Let $\{a_n\}_{n\geq 1}$ be a sequence of freely independent self-adjoint random variables in a C$^*$-probability space $(A,\varphi)$. Assume that $\varphi(a_n)=0$ for all $n$, $\sup_{n\geq 1}\|a_n\|<\infty$ and $\lim_{n\to\infty}n^{-1}\sum_{j=1}^n\varphi(a_n^2)=\frac{1}{4}$. Then the elements $S_n=n^{-\frac{1}{2}}\sum_{j=1}^n a_j$ converge in distribution to a semicircular element, i.e. their limit distribution is the semicircular law.
\end{theorem}
In particular, this says that, beside the Fock space construction, one could create semicircular random variables by taking elements of the form $S=(S_n)_n=(n^{-\frac{1}{2}}\sum_{j=1}^n a_j)_n \in 
(M^{\om},\tau_{\om})$, with $a_n$ as above in a finite W$^*$-probability space $(A,\varphi)=(M,\tau)$ and satisfying the additional condition $\sup_{n\geq 1}\|n^{-\frac{1}{2}}\sum_{j=1}^n a_j\|_{\infty}<\infty$. Let us also recall an informal statement of Voiculescu's matrix limit theorem (Theorem 2.2 in \cite{Voi91}):
\begin{theorem}[Voiculescu's matrix limit theorem]
Any family of random matrices with size going to infinity having independent normalized gaussian entries converges in distribution to a free semicircular family.
\end{theorem}
Again, we could interpret this as saying that one can create free semicircular families using elements of some suitable ultraproduct of matrix algebras over abelian von Neumann algebras.\\
1.2.3. {\bf q-gaussian von Neumann algebras.}\\
The $q$-gaussian von Neumann algebras $\Gamma_q(H)$, for $H$ a real Hilbert space, were introduced by Bo\.{z}ejko and Speicher and further studied, among others, by Bo\.{z}ejko, Speicher, Kummerer, Ricard, Sniady, Krolak, Nou, Shlyakhtenko, Nica, Dykema, Avsec, Dabrowski \cite{BoSpe1, BoSpe2, BoSpe3, BoSpe4, BoKuSpe, Ri, SniadyI, SniadyII, Kro1, Kro2, NouI, NouII, ShlyakhtenkoSEN, ShlyakhtenkoLEM, DykemaNica, Avsec, Dabrowski}. For $-1<q<1$, Bo\.{z}ejko and Speicher constructed a functor $\Gamma_q$ from the category of real Hilbert spaces with contractions to the category of finite von Neumann algebras with completely positive maps. $\Gamma_q(H)$ is called the $q$-gaussian von Neumann algebra associated to $H$. The generators $\Gamma_q(h)=s_q(h)$, for $h\in H$, admit a concrete representation as the real part of the creation operator by $h$ on the $q$-Fock space of $H$ (for details, see e.g. Section 2 of \cite{BoSpe1}). When $q=0$, the functor $\Gamma_q$ coincides with Voiculescu's functor $\Phi$, so $\Gamma_0(H)=L(\mathbb F_{{\rm dim}H})$. A direct computation using the concrete realization of the $s_q(h)$'s gives the moment formula
\[\tau(s_q(h_1)\cdots s_q(h_m))=\sum_{\si\in P_2(m)}q^{\rm cr(\si)}\prod_{\{l,r\}\in\si}\lan h_l,h_r\ran,\]
which is why, in view of the above, the $s_q(h)$'s can be called $q$-semicircular elements. The central limit theorem holds in the $q$-gaussian context as well, see e.g. Theorem 1 in \cite{SpeNCL}, Theorem 1 and 2 in \cite{SpeGSM}, \cite{Bo91} or Appendix A in \cite{JungeZeng}, but its statement is very technical and we omit it. Also, the $q$-gaussian von Neumann algebras admit random matrix models, see Thm. 3 in \cite{SniadyI}. We mention that, originally, the $q$-gaussian von Neumann algebras have been studied as concrete implementations of the canonical $q$-commutation relations, or as examples of non-classical Brownian motions (see e.g. \cite{BoSpe1, BoSpe2, BoSpe3}), but we choose to downplay these aspects in the present work. The central limit theorem suggests that the $q$-gaussians can be introduced via an ultraproduct model. In fact, a concrete ultraproduct embedding which holds a great heuristic value for us is given by
\[\Gamma_q(H)\ni s_q(h)\mapsto (n^{-\frac{1}{2}}\sum_{j=1}^n s_q(e_j\ten h))_n\in (\Gamma_q(\ell^2\ten H))^\om,\]
where $\{e_j\}_{j\in\nz}$ is the standard orthonormal basis of $\ell^2=\ell^2(\nz)$.
\subsection{Generalized q-gaussian von Neumann algebras with coefficients}
In this article we introduce a new class of von Neumann algebras and prove some structural results about them. Specifically, we introduce the \emph{generalized $q$-gaussian von Neumann algebras with coefficients} associated to a sequence of \emph{symmetric copies} $(\pi_j,B,A,D)$. A 4-tuple 
$(\pi_j,B,A,D)$ is called a sequence of symmetric copies (of $A$) if $B,A,D$ are finite tracial von Neumann algebras such that $B \subset A \cap D$ and $\pi_j:A \to D, j \in \nz$ are unital trace-preserving normal *-homomorphisms satisfying
\begin{enumerate}
\item $\pi_j |_B =id_B$, for all $j$;
\item $E_B(\pi_{j_1}(a_1)\ldots \pi_{j_m}(a_m))=E_B(\pi_{\si(j_1)}(a_1)\ldots \pi_{\si(j_m)}(a_m))$, for all finite permutations $\si$ on $\nz$, all indices $j_1,\ldots, j_m$ in $\nz$ and all $a_1,\ldots, a_m$ in $A$, where $E_B:D \to B$ is the canonical trace-preserving conditional expectation.
\end{enumerate}
We mention that our copies satisfy some additional independence conditions (see Definition 3.2). Let $-1<q<1$ be fixed. For $H$ an infinite dimensional (real) Hilbert space and $S$ a self-adjoint subset of $A$ containing $1$, the generalized $q$-gaussian von Neumann algebra 
\[\Gamma_q(B,S \ten H)\subset (\Gamma_q(\ell^2 \ten H) \bar{\ten} D)^{\om}\]
with coefficients in $B$ and associated to the symmetric copies $(\pi_j,B,A,D)$ is defined as the von Neumann subalgebra generated by the elements
\[s_q(a,h)=(n^{-\frac{1}{2}}\sum_{j=1}^n s_q(e_j \ten h) \ten \pi_j(a))_n, a \in BSB=\{b_1ab_2: b_1,b_2 \in B, a \in S\}, h \in H.\]
Here $\om$ is a free ultrafilter on the natural numbers and $\Gamma_q(\ell^2 \ten H)$ is the $q$-gaussian von Neumann algebra. When $H$ is finite dimensional, one needs to further apply a ''closure operation" (see Def.3.2 and Prop.3.14 for more details). The crucial observation here is that, since a Fock space model is not available, we are forced to introduce our generalized gaussians via an ultraproduct model. The generators $s_q(a,h)$ satisfy the following moment formula
\[\tau(s_q(a_1,h_1)\cdots s_q(a_m,h_m))=\delta_{m \in 2\nz} \sum_{\si \in P_2(m)} q^{\rm cr(\si)} \prod_{\{l,r\} \in \si} \lan h_l,h_r\ran \tau_D(\pi_{\phi_{\si}(1)}(a_1)\cdots \pi_{\phi_{\si}(m)}(a_m)),\]
as well as the $B$-valued moment formula
\[E_B(s_q(a_1,h_1)\cdots s_q(a_m,h_m))=\delta_{m \in 2\nz} \sum_{\si \in P_2(m)} q^{\rm cr(\si)} \prod_{\{l,r\} \in \si} \lan h_l,h_r\ran E_B(\pi_{\phi_{\si}(1)}(a_1)\cdots \pi_{\phi_{\si}(m)}(a_m)),\]
where for every pair partition $\si=\{\{k'_1,k''_1\},\ldots,\{k'_p,k''_p\}\} \in P_2(m)$, the function $\phi_{\si}:\{1,\ldots,m\}\to \{1,\ldots,p=\frac{m}{2}\}$ is chosen so that $\phi_{\si}(k'_1)=\phi_{\si}(k''_1)=1,\ldots,\phi_{\si}(k'_p)=\phi_{\si}(k''_p)=p$. In view of all of the above, the elements $s_q(a,h)$ could thus judiciously be called ``$B$-valued $q$-semicircular random variables having symmetric $B$-moments".  
\par When compared to pure $q$-gaussians, the generalized $q$-gaussian von Neumann algebras with coefficients can be viewed as an analogue of the cross-product von Neumann algebras $B \rtimes \Gamma$ as opposed to pure group von Neumann algebras $L(\Gamma)$. This analogy can be given some substance along the lines of \cite{ShlyakhtenkoAVS}. However, in the present work we do not pursue this insight and use this analogy merely as a guideline for the implementation of Popa's deformation-rigidity strategy.
The main result we prove about our generalized $q$-gaussian algebras is
\begin{main}\label{main1} Let $(\pi_j,B,A,D)$ be a sequence of symmetric independent copies, $H$ be a finite dimensional Hilbert space and $\mathcal A \subset M=\Gamma_q(B,S \ten H)$ be a diffuse von Neumann subalgebra which is amenable relative to $B$ inside $M$. For every $s \geq 0$, define $D_s(S)$ to be the following right $B$-submodule of $L^2(D)$:
\[\overline{\rm span}^{\|\cdot\|_2} \{E_{A_{\{1,\ldots,s\}}} (\pi_{\phi_{\si}(1)}(x_1)\cdots \pi_{\phi_{\si}(m)}(x_m)):m \geq 1, \si \in P_{1,2}(m), x_i \in BSB\},\]
where $m=s+2p$, $P_{1,2}(m)$ is the set of pair-singleton partitions of $\{1,\ldots,m\}$, $\si$ runs over all pair-singleton partitions in $P_{1,2}(m)$ having $s$ singletons and $p$ pairs and for such a partition $\si=\{\{k_1\},\ldots,\{k_s\},\{k'_1,k''_1\},\ldots,\{k'_p,k''_p\}\}$, the function $\phi_{\si}:\{1,\ldots,m\}\to \{1,\ldots,s+p\}$ satisfies $\phi_{\si}(k_1)=1,\ldots,\phi_{\si}(k_s)=s$ and $\phi_{\si}(k'_1)=\phi_{\si}(k''_1)=s+1,\ldots,\phi_{\si}(k'_p)=\phi_{\si}(k''_p)=s+p$. Assume that there exist constants $d, C>0$ such that $dim_B(D_s(S)) \leq Cd^s$ for all $s \geq 1$. Then at least one of the following statements is true:
\begin{enumerate} 
\item $\A \prec_M B$, or 
\item the von Neumann algebra $P=\N_M(\A)''$ generated by the normalizer of $\mathcal A$ in $M$ is amenable relative to $B$ inside $M$. 
\end{enumerate}
\end{main}
The technical condition on the dimension of the $B$-modules $D_k(S)$ implies in particular that the subspace of Wick words of length $k$ is finitely generated over $B$, for all $k\geq 1$ (see Thm. 3.16 and Prop. 3.20). This last condition in turn is the exact analogue of the group cocycle being proper in the case of cross-product von Neumann algebras.
\par As a consequence of our Theorem \ref{main1}, we find a number of examples of generalized $q$-gaussians which are strongly solid (when $B=\cz$ or finite dimensional) or strongly solid relative to $B$, for diffuse $B$. While the class of generalized $q$-gaussian von Neumann algebras with coefficients is huge (roughly speaking such a von Neumann algebra can be constructed starting from any action of the infinite symmetric group on another finite von Neumann algebra), the range of examples to which our Theorem A applies is greatly restricted by the technical assumptions we make. The examples in the corollary below are introduced in more detail in Section 4. 
\begin{mcor}\label{maincor1} The following von Neumann algebras are strongly solid relative to $B$:
\begin{enumerate}
\item (see 4.1) $B \bar{\ten} \Gamma_q(H)$, for $H$ a finite dimensional Hilbert space;
\item (see 4.3.2) $\Gamma_q(B,S \ten H)$ associated to the symmetric independent copies $(\pi_j,B,A,D)$ constructed in the following way: take a trace preserving action $\al$ of $\zz$ on a finite von Neumann algebra $N$. Let $\H=\lan g_j: j \geq 0 \ran $ be the Heisenberg group, take $\eta:\H \to \zz$ an onto group homomorphism and define $\beta:\H \curvearrowright N$ by
\[\beta_g(x)=\al_{\eta(g)}(x), g \in \mathcal H, x \in N.\]
Let $\mathcal{H}_1=\lan g_0,g_1 \ran$ and take $B=N \rtimes \zz=N \bar{\ten} L(\zz)$, $A=N\rtimes \mathcal{H}_1$ and $S=\{1,g_1,g_1^{-1}\}$. Define $\pi_j:A \to D$ by 
\[\pi_j(x u_{g_1})=\al_{\eta(g_j)}(x) u_{g_j}, \pi_j(x u_{g_0})=x u_{g_0}, x \in N, j,k \in \nz.\]
\item (see 4.4.1) $\Gamma_q(\cz,S \ten K)$ associated to the symmetric copies $(\pi_j,B=\cz,A=\Gamma_{q_0}(H),D=\Gamma_q(\ell^2 \ten H))$, where $\pi_j(s_{q_0}(h))=s_q(e_j \ten h)$ and $K$ is a finite dimensional Hilbert space;
\item (4.4.2) $\Gamma_q(B_d,S \ten H)$ associated to the symmetric copies $(\pi_j,B_d,A_d,D_d)$, where $B_d=L(\Si_{[-d,0]})$, $A_d=L(\Si_{[-d,1]})$, $D_d=L(\Si_{[-d,\infty)})=\{u_{\si}:\si \in \Si_{[-d,\infty)}\}''$ and $S=\{1,u_{(01)}\}$ for a fixed $d \in \nz \setminus \{0\}$; here $\Si_{\zz}$ is the group of finite permutations on $\zz$ and for a subset $F \subset \zz$, $\Si_F \subset \Si_{\zz}$ is the group of finite permutations on $F$ naturally embedded into $\Si_{\zz}$; the copies are defined by $\pi_j(a)=u_{(1j)}au_{(1j)}$, $a \in A_d$;
\item (see 4.2) $\Gamma_q(\cz, S \ten H)$ associated to the symmetric copies $(\pi_j,B=\cz,A=L(\zz)=\{u\}'',D=\ast_{\nz}L(\zz))$, where $u$ is a Haar unitary, the symmetric copies $\pi_j:A \to D$ are defined by the relations $\pi_j(u) =\ldots \ast 1\ast u\ast 1\ast \ldots$, and $S=\{1,u,u^*\}$.
\end{enumerate}
It follows that the examples in (3), (4) and (5) are strongly solid and hence solid non-amenable von Neumann algebras. In particular, they are prime von Neumann algebras. Note that when $q=0$ and $H$ is trivial, the example in (5) is *-isomorphic to $L(\mathbb F_\infty)$, thus reproving the strong solidity of the free group factors.
\end{mcor} 
Using Theorem \ref{main1} we also deduce the following
\begin{mcor}\label{maincor2} Let $M_i=\Gamma_{q_i}(B_i,S_i \ten H_i)$ be associated with two sequences of symmetric independent copies $(\pi_j^i,B_i,A_i,D_i)$ and two subsets $S_i \subset A_i$,  and $-1<q_i<1$, $i=1,2$. Assume that $ 2 \leq dim(H_i) < \infty$, $dim_{B_i}(D_k(S_i))\leq Cd^k$ for fixed constants $d,C>0$ and $B_i$ are amenable, for $i=1,2$. If $M_1 \subset M_2$, then $B_1 \prec_{M_2} B_2$. Moreover, if $M_1=M_2=M$, it follows that $B_1 \prec_M B_2$ and $B_2 \prec_M B_1$.
\end{mcor}
This result can be regarded as an analogue of the "uniqueness of Cartan subalgebra" results in the group measure space construction setting. Note however, that even when $B$ is abelian, it is not a MASA in $M = \Gamma_q(B,S \ten H)$. Indeed, $B$ always commutes with a copy of $\Gamma_q(H)$ inside $M=\Gamma_q(B,S \ten H)$, hence it can never be maximally abelian. Thus, even when $B_1$ and $B_2$ are both abelian diffuse, we cannot avail ourselves of Popa's results about unitary conjugacy of Cartan subalgebras (\cite{PoBe}, Appendix, Thm. A.1) to conclude that $B_1$ is unitarily conjugate to $B_2$, so this double intertwining result is optimal in our case. Finally, we deduce some non-isomorphism and non-embedability results for generalized $q$-gaussians.
\begin{mcor}\label{maincor3} Under the assumptions of Corollary \ref{maincor2}, if we moreover assume that 
\begin{enumerate}
\item $B_1$ is finite dimensional and $B_2$ is amenable diffuse, or
\item $B_1$ is abelian and $B_2$ is the hyperfinite $II_1$ factor, 
\end{enumerate}
then $M_2=\Gamma_{q_2}(B_2,S_2\ten H_2)$ cannot be realized as a von Neumann subalgebra of $M_1=\Gamma_{q_1}(B_1,S_1 \ten H_1)$. In particular $M_1$ and $M_2$ are not *-isomorphic.
\end{mcor}

\subsection{Comments on the proofs and structure of the article}
Finally, a couple of words about the main ideas behind the proof of Theorem A. We mention that actually Theorem A will be derived from the technical theorem 7.2 much along the lines of Thm. 3.1 in \cite{PoVaI}, whose statement and proof can be found in Section 7. We follow the approach of Popa and Vaes in \cite{PoVaI,PoVaII}, approach which in turn is a development of the original ground-breaking insight in \cite{OPCartanI, OPCartanII}. Let $\A \subset M=\Gamma_q(B,S\ten H)$ a diffuse von Neumann subalgebra which is amenable relative to $B$. The two main ingredients of the proof are, just as in \cite{PoVaI}
\begin{enumerate}
\item The fact that the embedding $\A \subset M$ is \emph{weakly compact relative to B}. This is the existence of a sequence of normal states viewed as unit vectors $\xi_n \in L^2(\N)$, where $\N \supset M$ is a suitable (in general non-tracial) von Neumann algebra, which are asymptotically invariant to the action of the ''tensor double" of the normalizer of $\A$ in $M$; the existence of these states is a consequence of the weak amenability (with Cowling-Haagerup constant 1) of the pure $q$-gaussian von Neumann algebras $\Gamma_q(H)$;
\item The existence of a 1-parameter group of *-automorphisms $(\al_{t})$ of a suitable dilation $\tilde{\N}$ of $\N$ having good properties.
\end{enumerate}
The proof proceeds by applying the deformation $\al_{t}$ to the vectors $\xi_n$. Then either the deformation significantly displaces the vectors, or it does not. The first case yields the amenability of $P=\N_M(\A)''$ relative to $B$, while the second implies that $\A \prec_M B$, via the fact that the maps $T_t$ (where $t \to T_t$ is the canonical semi-group of ucp maps on $M$) are compact over $B$, in the terminology of Popa and Ozawa. 
\par While it's true that conceptually we follow closely the approach of Popa and Vaes in \cite{PoVaI}, it has to be strongly emphasized that the technical difficulties of our approach are vastly larger. First of all, since our objects are much more elusive and complicated than cross-product von Neumann algebras, being defined as subalgebras of an ultraproduct to begin with, the proof of Theorem 5.1 (the existence of the invariant states), which is the key ingredient in the proof of the technical theorem, is ridden with daunting difficulties. Among these, constructing the von Neumann algebras involved in the deformation-rigidity argument (e.g. $\N$ and $\tilde \N$ above) and the spaces on which they act was a particularly challenging task. Also, the complete boundedness of certain maps used in the proof turns out to be surprisingly non-trivial and requires the use of delicate operator spaces techniques (in the pure Hilbert space setting, somewhat similar techniques have been used in \cite{Avsec, NouI, NouII}). Second, and just as important, we cannot use the reduction to the "trivial action case" (i.e. the tensor product case), as Popa and Vaes do. The reduction step plays a crucial role in their proof, because it is only in the tensor product setting that they are able to prove the relative weak compactness property and subsequently carry out the deformation-rigidity arguments. The reduction is essentially based on the use of the co-multiplication map in the cross-product case. Since we have no good substitute for the co-multiplication map, we cannot reduce to the tensor product case, and hence everything becomes much more complicated and technically involved, including the standard forms of the von Neumann algebras involved, which are in general non-tracial.
\par The article contains six sections beside the introduction, and is organized as follows: Section 2 contains some needed technical preliminaries. In Section 3 we introduce the generalized $q$-gaussian von Neumann algebras and prove their basic properties; among other things, we exhibit the canonical generators of $\Gamma_q(B,S\ten H)$ (the Wick words), prove that they actually belong to the algebra and prove a very useful reduction result about them. Section 4 lists a rather wide range of examples of generalized $q$-gaussian von Neumann algebras constructed from a variety of symmetric independent copies. We devote Section 5 to the proof of the relative weak compactness of the embedding $\A \subset M$; the second half of this section contains some technical results about the complete boundedness of certain multipliers used in the proof. In Section 6 we prove that under the assumption of sub-exponential growth of the dimensions of the modules $D_k(S)$ over $B$, the natural deformation bimodules used in the technical theorem are weakly contained in $L^2(M) \ten_B L^2(M)$, fact which will be further used in combination with the technical theorem to derive Theorem \ref{main1}. The proof is based on a novel and "non-deterministic" approach. Indeed, the calculation of the deformation bimodules in the $q$-gaussian setting is a real challenge even in the case of pure $\Gamma_q(H)$ von Neumann algebras, see \cite{Avsec}, and it becomes even more so when we allow $q$-gaussians with coefficients. Section 7 contains the proof of the main technical theorem and its applications. Beside many examples of strongly solid generalized $q$-gaussian von Neumann algebras, we also obtain some non-isomorphism and non-embedability results.
\section{Preliminaries}
\subsection{Popa's intertwining techniques.}
We will briefly review the concept of intertwining two subalgebras inside a finite von Neumann algebra, along with the main technical tools developed by Popa in \cite{PoBe,PoI}. Let $(M,\tau)$ be a finite von Neumann algebra, let $f \in \mathcal P(M)$ and $Q \subset fMf, B \subset M$ be two von Neumann subalgebras. We say that \emph{a corner of $Q$ can be intertwined into $B$ inside $M$} and denote it by $Q \prec_M B$ (or simply $Q \prec B$) if there exist two non-zero projections $q \in Q$, $p \in B$, a non-zero partial isometry $v\in qMp$, and a $*$-homomorphism $\psi:qQq \rightarrow pBp$ such that $v\psi(x)=xv$ for all $x\in qQq$. The partial isometry $v$ is called an intertwiner between $Q$ and $B$.
Popa proved in \cite{PoI} the following intertwining criterion:
\begin{theorem}[Corollary 2.3 in  \cite{PoI}]\label{intertwining-techniques} Let $M$ be a von Neumann algebra and let  $Q\subset fMf$, $B\subset M$ be diffuse subalgebras for some projection $f \in M$. Then the following are equivalent:
\begin{enumerate}
\item $Q \prec_M B$.
\item There exists a finite set $\mathcal F\subset fMf$ and $\delta>0$ such that for every unitary $v\in \mathcal U(Q)$ we have
\begin{equation*}\sum_{x,y\in\mathcal F} \|E_B(xvy^*)\|^2_2 \geq \delta. 
\end{equation*}  
\end{enumerate}
\end{theorem}
Let $(M,\tau)$ be a finite von Neumann algebra and $\Phi:M\to M$ a normal, completely positive map. We say that $\Phi$ is sub-tracial if $\tau \circ \Phi \leq \tau$. If $\Phi$ is sub-tracial, then, due to the Schwartz inequality, we automatically have
\[\|\Phi(x)\|_2^2=\tau(\Phi(x)^*\Phi(x))\leq \tau(\Phi(x^*x))\leq \tau(x^*x)=\|x\|_2^2,\]
i.e. $\Phi$ is automatically $\|\cdot\|_2$-contractive, and hence extends to a bounded operator on $L^2(M)$ defined by
\[T_{\Phi}:L^2(M)\to L^2(M), T_{\Phi}(\hat{x})=\widehat{\Phi(x)}, x \in M.\]
Let $B \subset (M,\tau)$ be an inclusion of finite von Neumann algebras. The basic construction (of $M$ with $B$) is defined by (see e.g. \cite{PoBe})
\[\lan M,e_B \ran =(M \cup \{e_B\})''=(JBJ)' \subset B(L^2(M)),\]
where $L^2(M)$ is the standard form of $M$ and $J:L^2(M) \to L^2(M)$ the associated conjugation. The definition of the \emph{compact ideal space} of the basic construction (more generally of any semi-finite von Neumann algebra) can be found in \cite{PoBe}, 1.3.3.
\begin{defi}Let $(M, \tau)$ be a finite von Neumann algebra, $B \subset M$ a von Neumann subalgebra and $\Phi:M \rightarrow M$ a normal, completely positive, $B$-bimodular, sub-unital, sub-tracial map. We say that $\Phi$ is compact over $B$ if the canonical operator $T_{\Phi}:L^2(M) \rightarrow L^2(M)$ belongs to the compact ideal space of the basic construction $\langle M, e_B \rangle$.
\end{defi}
The following result is Prop.2.7 in \cite{OPCartanI} (see also \cite{PoBe}, 1.3.3.).
\begin{prop}Let $(M,\tau)$ be a finite von Neumann algebra and let $B, P \subset M$ be two von Neumann subalgebras. Let $\Phi:M \rightarrow M$ be a normal, completely positive, sub-unital, sub-tracial map which is compact over $B$ and assume that
\[\inf_{u \in \mathcal{U}(P)} \|\Phi(u)\|_2 > 0. \]
Then $P \prec_M B$.
\end{prop}
\subsection{Bimodules over von Neumann algebras and weak containment.}
Let $M, Q$ be two von Neumann algebras. An $M-Q$ Hilbert bimodule $\mathcal K$ is simply a Hilbert space together with a pair of normal *representations $\lambda:M \to B(\mathcal K)$, $\rho:Q^{op} \to B(\mathcal K)$ with commuting ranges. To these one can associate a *-representation $\pi:M \ten_{\rm bin} Q^{op} \to B(\mathcal K)$ by
\[ \pi(\sum_k x_k \ten y_k^{op})\xi = \sum_k \lambda(x_k)\rho(y_k^{op})\xi, \quad x_k \in M, y_k \in Q, \xi \in \mathcal K.\]
\begin{defi} Let $M,Q$ be two von Neumann algebras and $\mathcal H, \mathcal K$ be two $M-Q$ bimodules. We say that \emph{$\mathcal K$ is weakly contained in $\mathcal H$} and denote it by $\mathcal K \prec \mathcal H$ if $\|\pi_{\mathcal K}(x)\|\leq \|\pi_{\mathcal H}(x)\|$ for all $x \in M \ten_{alg} Q$, where $\pi_{\mathcal H}, \pi_{\mathcal K}$ are the *-representations canonically associated to the left and right actions on $\mathcal H, \mathcal K$ respectively.
\end{defi}
Give an $M-Q$ bimodule $\mathcal K$ and an $Q-N$ bimodule $\mathcal H$ we will denote by $\mathcal K \ten_Q \mathcal H$ their Connes tensor product, which is an $M-N$ bimodule. For the definition and basic properties of the Connes tensor product, see sections 2.3, 2.4 in \cite{PoVaI}. The Connes tensor product is well behaved with respect to weak containment (see idem).
\begin{defi}[Def. 2.3 and Prop. 2.4 in \cite{PoVaI}] Let $(M,\tau_M)$ and $(Q,\tau_Q)$ be finite tracial von Neumann algebras and $P \subset M$ a von Neumann subalgebra. We say that an $M-Q$ bimodule $\mathcal K$ is \emph{left $P$-amenable} if one of the following equivalent conditions holds:
\begin{enumerate}
\item There exists a $P$-central state $\Om$ on $B(\mathcal K)\cap (Q^{op})'$ such that $\Om|_M=\tau_M$.
\item $L^2(M) \prec \mathcal K \ten_Q \overline{\mathcal K}$ as $M-P$ bimodules.
\end{enumerate}
\end{defi}
\begin{defi} Let $(M,\tau)$ be a tracial von Neumann algebra, and let $B, P \subset M$ be two von Neumann algebras. We say that \emph{$P$ is amenable relative to $B$ inside $M$} if one of the following equivalent conditions holds:
\begin{enumerate}
\item The $M-B$ bimodule $L^2(M)$ is left $P$-amenable;
\item $L^2(M) \prec L^2(M) \ten_B L^2(M)$ as $M-P$ bimodules.
\end{enumerate} 
\end{defi}
\begin{rem} Let $(M,\tau)$ be a finite von Neumann algebra and $B, P \subset M$ two von Neumann subalgebras. Let $\mathcal K$ be a left $P$-amenable $M-M$ bimodule such that $\mathcal K \prec L^2(M) \ten_B \mathcal H$ for some $B-M$ bimodule $\mathcal H$. Then $P$ is amenable relative to $B$ inside $M$. Indeed, we have that, as $M-P$ bimodules
\[L^2(M) \prec \mathcal K \ten_M \overline{\mathcal K} \prec (L^2(M) \ten_B \mathcal H) \ten_M (\bar{\mathcal H} \ten_B L^2(M)) \prec L^2(M) \ten_B L^2(M).\] 
\end{rem}
\subsection{Standard forms of non-tracial von Neumann algebras.}
In some instances we will have to consider non-tracial von Neumann algebras $M$ and their standard forms. 
Let us recall that a (hyper) standard form for a von Neumann algebra is given by $(M,H,J,P)$, where $J:H\to H$ is an antilinear unitary,  $P\subset  H$ is a self-dual cone such that
 \begin{enumerate}
 \item[i)] the map $ M \ni x \mapsto Jx^*J \in M'$ is a *-anti-isomorphism acting trivially on $\mathcal Z(M)$;
 \item[ii)] $J\xi=\xi$ for $\xi\in P$;
 \item[iii)] $xJxJ(P)\subset P$ for $x \in M$.
 \end{enumerate}
The standard form of $M$ is unique up to *-isomorphism, see e.g. \cite{HaagerupSF}. A particularly useful way of describing the standard form of $M$ is the abstract Haagerup $L^2(M)$ space, which we briefly describe below. The reader can find more details in \cite{HaagerupLP, Terp, HaagerupJungeXu}. Let $(M,\varphi)$ a von Neumann algebra endowed with a normal semi-finite faithful (n.s.f.) weight. Consider $\mathcal M = M \rtimes_{\si^{\varphi}} \rz$ the cross-product von Neumann algebra of $M$ with $\rz$ by the modular automorphism group $\si_t^{\varphi}$. Then $\mathcal M$ is semi-finite and there exists a n.s.f. trace $\tau$ on $\mathcal M$ such that
\[(D\hat{\varphi}:D\tau)_t=\lambda(t), \quad t \in \rz,\] 
where $\hat{\varphi}$ is the dual weight, $(D\hat{\varphi}:D\tau)_t$ is the Connes cocycle
and $\lambda(t)$ is the group of translations on $\rz$. Moreover, $\tau$ is the unique n.s.f. trace on $\mathcal M$ which satisfies 
\[ \tau \circ \hat{\si}_t^{\varphi} = e^{-t}\tau, \quad t \in \rz.\]
 Given another n.s.w. $\psi$ on $M$, denote by $h_{\psi}$ the Radon-Nikodym derivative of $\hat{\psi}$ with respect to $\tau$, i.e. the unique positive self-adjoint operator affiliated to $\mathcal M$ such that
\[\hat{\psi}(x)=\tau(h_{\psi}^{\frac{1}{2}}xh_{\psi}^{\frac{1}{2}}), \quad x \in \mathcal M_{+}.\]
Then the following condition holds:
\[\hat{\si}_t^{\varphi}(h_{\psi})=e^{-t}h_{\psi}, \quad t \in \rz.\]
Moreover, the map $\psi \mapsto h_{\psi}$ is a bijection from the set of n.s. weights on $M$ to the set of positive self-adjoint operators affiliated to $\mathcal M$ which satisfy the above condition. Let $L_0(\mathcal M, \tau)$ be the *-algebra consisting of all the operators on $L^2(\rz,H)$ which are measurable with respect to $(\mathcal M, \tau)$. For $p>0$, the Haagerup $L^p(M,\varphi)$ is defined by
\[L^p(M,\varphi)=\{x \in L_0(\mathcal M, \tau):\hat{\si}_t^{\varphi}(x)=e^{-\frac{t}{p}}x, \forall t \in \rz\}.\]
One can define a bi-continuous linear isomorphism from $M_*$ to $L^1(M,\varphi)$ as the linear extension of the map
\[M_*^+ \ni \psi \mapsto h_{\psi} \in L^1(M,\varphi).\]
The norm $\|\cdot\|_1$ on $L^1(M,\varphi)$ is defined by requiring that the above isomorphism be isometric. One can define a norm one linear functional $tr$ on $L^1(M,\varphi)$ by $ tr(h_{\psi})=\psi(1)$, and thus $\|h\|_1= tr(|h|), h \in L^1(M,\varphi)$. This "trace" is indeed tracial, i.e. 
\[tr(xy)=tr(yx), \quad \mbox{for} \quad x,y\in L^2(M).\]
Let $x=u|x|$ be the polar decomposition of an element $x \in L_0(\mathcal M,\tau)$. Then we have
\[x \in L^p(M,\varphi) \Leftrightarrow u \in M \quad \mbox{and} \quad |x| \in L^p(M,\varphi) \Leftrightarrow u \in M \quad \mbox{and} \quad |x|^p \in L^1(M,\varphi).\]
This allows one to introduce the $\|\cdot\|_p$-norm on $L^p(M,\varphi)$, by $\|x\|_p=\||x|^p\|_1^{\frac{1}{p}}$ for $x \in L^p(M,\varphi)$. Let's also remark that the weight $\varphi$ can be recovered from the trace. Define
\[N_{\varphi}=\{ x \in M:\varphi(x^*x)<\infty\}, \quad M_{\varphi} = N_{\varphi}^*N_{\varphi}=\mbox{span}\{y^*x: x,y \in N_{\varphi}\}.\] 
The dual weight $\hat{\varphi}$ has a Radon-Nikodym derivative with respect to $\tau$, which will be denoted $d_{\varphi}$. Then for every $x \in M_{\varphi}$ the operator $d_{\varphi}^{\frac{1}{2}}xd_{\varphi}^{\frac{1}{2}}$ is closable, its closure belongs to $L^1(M,\varphi)$ and we have the following relation
\[\varphi(x)=tr(d_{\varphi}^{\frac{1}{2}}xd_{\varphi}^{\frac{1}{2}}), \quad x \in M_{\varphi}.\]
If $\varphi$ is a bounded functional, then $d_{\varphi}\in L^1(M,\varphi)$ and the above identity becomes
\[\varphi(x)=tr(d_{\varphi}^{\frac{1}{2}}xd_{\varphi}^{\frac{1}{2}})=tr(xd_{\varphi}), \quad x \in M.\]
The Haagerup space $L^p(M,\varphi)$ does not depend on the choice of the n.s.f. weight $\varphi$ up to isomorphism, hence it can simply be denoted by $L^p(M)$. It's easy to see that $M$ is naturally represented in standard form on the Haagerup space $L^2(M)$ via the obvious left and right actions. When $M$ is finite and $\tau$ is a faithful trace on $M$, the Haagerup space $L^2(M) = L^2(M,\tau)$ coincides with the usual one.
\subsection{W*-Hilbert modules.}
We also have to recall some facts about (right) Hilbert $W^*$-modules. According to \cite{PaschkeI, PaschkeII} (see also \cite{JungeSherman}) a right Hilbert $C^*$- module $X$ over a \emph{von Neumann algebra} $M$ is self-dual if and only if admits a module basis, i.e. a family $\{\xi_{\al}\} \subset X$ such that 
\[X=\overline{\mbox{span}}\sum_{\al}\xi_{\al}M \quad \mbox{and} \quad \lan \xi_{\al}, \xi_{\beta} \ran=\delta_{\al \beta}e_{\al} \in \mathcal P(M).\] 
Here, $\lan \cdot,\cdot \ran$ denotes the $M$-valued inner product. In this situation, there exists an index set $I$, a projection $e\in B(\ell_2(I))\bar{\ten}M$,  and a right module isomorphism $u:X\to e(\ell_2(I)^c\bar{\ten}M)$. Indeed, for a basis $\xi_{\al}$ with $\langle \xi_{\al},\xi_{\al}\rangle =e_{\al}$ the map $u$ is given by $u(\sum_{\al}\xi_{\al}m_{\al})=[e_{\al}m_{\al}]$. Here
$\ell_2(I)^c\bar{\ten}M$ denotes the space of strongly convergent columns indexed by $I$. Then it is easy to see that the $C^*$-algebra $\mathcal{L}(X)$ of adjointable operators on $X$ is indeed a von Neumann algebra, and isomorphic to $e(B(\ell_2(I))\bar{\ten}M)e$. Moreover, the $M$-compact operators $\K(X)$ spanned by the maps $\Phi_{\xi,\eta}(\zeta)=\xi\langle \eta,\zeta\rangle$ are  weakly dense in $\mathcal{L}(X)$, because $K(\ell_2(I))\ten_{\min}M$ is weakly dense in $B(\ell_2(I))\bar{\ten}M$. With the help of a  normal faithful state, we can complete $X$ to the Hilbert space $L_2(X,\phi)$ with inner product $(\xi,\eta)=\phi(\langle \xi,\eta\rangle)$. Let $\iota_{\phi}:X\to L_2(X,\phi)$ the inclusion map. Then \[ \pi:\mathcal{L}(X) \to B(L_2(X,\phi))\pl,\pl 
\pi(T)(\iota_{\phi}(x)) \lel \iota_{\phi}(Tx) \] 
defines a normal faithful $^*$-homomorphism such that 
 \[ \pi(\mathcal{L}(X)) \lel B(L_2(X,\phi))\cap (M^{op})' \pl. \]
This is indeed very easy to check for $\mathcal{L}(X)=
e(B(\ell_2(I))\bar{\ten}M)e$. See \cite{PaschkeI, PaschkeII, JungeSherman} for more details and references. 
\section{The generalized gaussian von Neumann algebras with coefficients - definition and basic properties}
Throughout this section we will freely use the basic properties of the pure Hilbert space $q$-gaussian von Neumann algebras $\Gamma_q(H)$, as they can be found in Section 4 of \cite{JLU} (see also \cite {Avsec}).
The following result is due to Dykema, Nica and Voiculescu and can be found in \cite{VDN}.
\begin{prop} Let $(M,\varphi)$ and $(N,\psi)$ be two von Neumann algebras endowed with faithful normal tracial states. Let $(x_i)_{i=1}^{\infty}$ and $(y_j)_{j=1}^{\infty}$ be countable  systems of generators for $M$ and $N$, respectively. Assume that for every $m \geq 1$, every $i_1,\ldots,i_m \in \nz$ and every $\eps_i \in \lbrace 1, * \rbrace$ we have
\[\varphi(x_{i_1}^{\eps_1} \cdots x_{i_m}^{\eps_m})=\psi(y_{i_1}^{\eps_1}\cdots y_{i_m}^{\eps_m}).\]
Then there exists a *-isomorphism $\pi:M \to N$ such that $\psi \circ \pi =\varphi$ and $\pi(x_i)=y_i$ for all $i \geq 1$.
\end{prop}
\begin{defi} Let $A$ and $D$ be two finite tracial von Neumann algebras and $B$ a von Neumann subalgebra of $A \cap D$. Let $\pi_j:A \to D, j \in \nz$ be a countable family of unital, normal, faithful, trace-preserving *-homomorphisms. The 4-tuple $(\pi_j, B, A, D)$ is called a sequence of \emph{symmetric independent copies} of $A$ if the following properties hold:
\begin{enumerate}
\item $\pi_j |_B =id_B$, for all $j$;
\item $E_B(\pi_{j_1}(a_1)\ldots \pi_{j_m}(a_m))=E_B(\pi_{\si(j_1)}(a_1)\ldots \pi_{\si(j_m)}(a_m))$, for all finite permutations $\si$ on $\nz$, all indices $j_1,\ldots, j_m$ in $\nz$ and all $a_1,\ldots, a_m$ in $A$, where $E_B:D \to B$ is the canonical trace-preserving conditional expectation;
\item For $i \in \nz$ denote by $A_i=\pi_i(A) \subset D$ and for $I \subset \nz$, denote by $A_I=\bigvee_{i \in I} \pi_i(A)=\bigvee_{i \in I} A_i \subset D$ (by convention, set $A_{\emptyset}=B$); then, for any finite subsets $I \subset J \subset \nz$, $j \notin J$, $d \in A_I$ and $a, a' \in A$, we have 
\[E_{A_I}(\pi_j(a)d\pi_j(a')) =E_{A_J}(\pi_j(a)d\pi_j(a')),\]
 where $E_{A_I}:D \to A_I$ is the canonical conditional expectation;
\item for any finite subsets $I, J \subset \nz$, we have $E_{A_I}E_{A_J}=E_{A_{I \cap J}}$. Note that this automatically implies $E_{A_I}E_{A_J}=E_{A_J}E_{A_I}=E_{A_I \cap A_J}$ and in particular $A_I \cap A_J=A_{I\cap J}$.
\item $A_{\nz}=D$.
\end{enumerate}
If the 4-tuple $(\pi_j, B, A, D)$ only satisfies axioms (1) and (2), we call it a sequence of symmetric copies.
\end{defi}
The role played by the copies $\pi_j(A)$ is analogous to that of tensor copies in a classical product probability space, in fact such an infinite product probability space over a commutative or non-commutative base constitutes the first obvious example of symmetric independent copies. To be more precise, let $(A,\tau)$ a tracial von Neumann algebra and let $D=\bigotimes_{i\in\nz}(A_i,\tau_i)$, where $(A_i,\tau_i)=(A,\tau)$ for all $i\in\nz$, let $\pi_j$ be the obvious embedding of $A$ in $D$ as the $j$-th tensor copy and $B=\cz$. Then all the axioms (1) to (5) are satisfied. In particular, one could take $(A,\tau)=(L^{\infty}(X,\mu),\int_X d\mu)$ for a probability measure space $(X,\mu)$. Axiom (2), while convenient because it greatly simplifies some of our technical computations, doesn't seem to be indispensable to the development of a general theory of $B$-valued $q$-gaussian von Neumann algebras. Indeed, the generalized $q$-gaussian von Neumann algebras can still be introduced in the presence of a weaker ''sub-symmetry" assumption, but the technicalities become even more cumbersome, and it is unclear whether some of our results can still be obtained. Axioms (3) and (4) can both be viewed as describing some sort of independence of the copies over $B$, with (4) being the more obvious one, since for example it entails that for $I \cap J=\emptyset$, we have $E_{A_I}E_{A_J}=E_B$. In the case of an abelian $D$ and $B=\cz$, this amounts to classical probabilistic independence. Axiom (5), while added for completeness, can always be made redundant by shrinking the algebra $D$.\\
In what follows, the expectations $E_{A_I}$ will be denoted $E_I$.
\begin{prop} Let $(\pi_j,B,A,D)$ a sequence of symmetric copies. Let $\Si = \mathbb{S}(\infty)$ be the group of finite permutations on $\nz \setminus \{0\}$. Then for every $\si \in \Si$ there exist a trace preserving automorphism $\al_{\si}$ of $D_0=A_{\nz \setminus \{0\}} \subset D$ such that $\al_{\si}(\pi_{j_1}(x_1)\cdots \pi_{j_m}(x_m))=\pi_{\si(j_1)}(x_1)\cdots \pi_{\si(j_m)}(x_m)$, for all $x_1,\ldots,x_m \in A$ and $j_1,\ldots,j_m \in \nz$. Moreover
\[\Si \ni \si \mapsto \al_{\si} \in Aut(D_0,\tau)\]
is an action of $\Si$ on $D_0$ by trace-preserving automorphisms. Moreover, if the symmetric copies satisfy axiom 4, then the fixed points algebra of this action is $B$.
\end{prop}
\begin{proof} The map $V_{\si}:L^2(D_0)\to L^2(D_0)$ defined by
\[\sum \pi_{j_1}(x_1)\cdots \pi_{j_m}(x_m) \mapsto \sum \pi_{\si(j_1)}(x_1)\cdots \pi_{\si(j_m)}(x_m)\]
is easily seen to be a well-defined unitary because of axiom 2. Then $\al_{\si}=Ad(V_{\si})|_{D}$ is a trace preserving automorphism of $D$ which satisfies the required condition. The verification of the second statement is straightforward and we leave it to the reader.
\end{proof}
Symmetric copies can also be introduced in the following alternative way, which is a converse to the previous proposition: assume that $\al:\Si \to Aut(D,\tau)$ is a trace preserving action by *-automorphism of the finite von Neumann algebra $D$, where $\Si$ is now the finite permutation group on $\nz \cup \{0\}$ instead of $\nz$. Denote by $B=D^{\Si}$ the fixed points algebra of this action. Denote by $\Si_0=Stab_{\Si}(0)=\{\si \in \Si: \si(0)=0\}$. Set $A=D^{\Si_0}=\{d \in D:\al_{\si}(d)=d, \forall \si \in \Si_0\}$. Note that $\Si_0 \subset \Si$ is a subgroup isomorphic to $\mathbb{S}(\infty)$ and that $B \subset A \subset D$. For every $j \geq 1$, define $\pi_j:A \to D$ by the formula $\pi_j(a)=\al_{(0j)}(a), a \in A$, where $(0j) \in \Si$ is the transposition interchanging $0$ and $j$. Then $(\pi_j,B,A,D)$ represents a sequence of symmetric copies. Indeed, for any $j \geq 1$ and $b \in B$ we have $\pi_j(b)=\al_{(0j)}(b)=b$ because $B$ is the fixed points algebra of the action $\al$, so (1) is true. Note that $\al_{\si}(a)=a$ for every $\si \in \Si_0$ and $a \in A$ and $E_B \circ \al_{\si}=E_B$ for all $\si \in \Si$, due to (1) and the facts that $\al$ is trace preserving and the trace preserving conditional expectation $E_B:D \to B$ is unique. Then for every $\si \in \Si_0 \cong \mathbb{S}(\infty)$ and for all $j_1,\ldots,j_m \geq 1$ and $a_1,\ldots,a_m \in A$ we have
\begin{align*}
& E_B(\pi_{\si(j_1)}(a_1)\cdots \pi_{\si(j_m)}(a_m)) = E_B(\al_{(0\si(j_1))}(a_1)\cdots \al_{(0\si(j_m))}(a_m)) = \\
& E_B(\al_{\si (0j_1) \si^{-1}}(a_1)\cdots \al_{\si (0j_m) \si^{-1}}(a_m)) = E_B((\al_{\si} \circ \al_{(0j_1)} \circ \al_{\si^{-1}})(a_1)\cdots (\al_{\si} \circ \al_{(0j_m)} \circ \al_{\si^{-1}})(a_m))= \\
& E_B(\al_{\si}(\al_{(0j_1)}(a_1)\cdots \al_{(0j_m)}(a_m)))=E_B(\al_{(0j_1)}(a_1)\cdots \al_{(0j_m)}(a_m)) = E_B(\pi_{j_1}(a_1)\cdots \pi_{j_m}(a_m)),
\end{align*}
so (2) is also true. As noted before, we can also assume without loss of generality that $D=\bigvee_{j \geq 1} \pi_j(A)=\bigvee_{j \geq 1} A_j$, by simply replacing $D$ with a von Neumann subalgebra.
\par {\bf Notation.} Let $(j_1,\ldots,j_m)$ be an $m$-tuple with $1\leq j_k\leq n$, $1\leq k\leq m$. We denote by $P(m)$ the set of partitions of $\{1,\ldots,m\}$ and by $\dot{0}, \dot{1}$ the finest and the coarsest partition in $P(m)$, respectively. The notation $P_{1,2}(m)$ stands for the collection of all the partitions of $\{1,\ldots,m\}$ consisting only of singletons and pairs. For $\si\in P(m)$, we say that
\begin{enumerate}
\item $(j_1,\ldots,j_m)\leq \si$ if $j_i=j_k$ whenever $i,k\in A\in \si$;
\item $(j_1,\ldots,j_m)\geq\si$ if $j_i=j_k$ implies that there exists an $A\in\si$ with $i,k\in A$;
\item $(j_1,\ldots,j_m)=\si$ if $j_i=j_k$ exactly when there exists an $A\in\si$ such that $i,k\in A$.
\end{enumerate}
Given an $m$-tuple $(j_1,\ldots,j_m)=\dot{0}$ with $1\leq j_k\leq n$ for $1\leq k\leq m$, we denote by $\al_{j_1,\ldots,j_m}=\al_{\si_{j_1,\ldots,j_m}}$, where $\si_{j_1,\ldots,j_m}(i)=j_i, 1 \leq i \leq m$.

\begin{defi} Let $(\pi_j, B, A, D)$ be a sequence of symmetric independent copies, $S$ a subset of $A$ such that $1 \in S=S^*$, $H$ a Hilbert space and $\omega$ a free ultrafilter on $\nz$. Denote by $\lbrace e_j \rbrace$ the canonical orthonormal basis of $\ell^2=\ell^2(\nz)$. Let $-1<q<1$. Define 
\[\Gamma_q^0(B,S\ten H) = (B \cup \lbrace s_q(a,h):a \in S, h \in H \rbrace)'' \subset (\Gamma_q(\ell^2 \ten H) \bar{\ten} D)^{\omega},\] 
where
\[s_q(a,h)=(n^{-\frac{1}{2}}\sum_{j=1}^n s_q(e_j \ten h) \ten \pi_j(a))_n.\]
Finally define
\[\Gamma_q(B,S\ten H)=(E_{\Gamma_q(\ell^2_n \ten H)}\ten id)_n(\Gamma_q^0(B,S\ten K)),\]
 where $K$ is an infinite dimensional Hilbert space containing $H$, $\ell^2_n=$span$\lbrace e_1, \ldots, e_n \rbrace$ and for each $n$ 
 \[E_{\Gamma_q(\ell^2_n \ten H)}:\Gamma_q(\ell^2 \ten K) \to \Gamma_q(\ell^2_n \ten H)\] is the canonical conditional expectation.
\end{defi}
As $q$ will be fixed throughout this section, we will simply use the notation $s(x,h)$ instead of $s_q(x,h)$ from now on.
\begin{rem} Due to functoriality, the definition of $\Gamma_q(B,S \ten H)$ does not depend on the particular choice of $K \supset H$. When $H$ is infinite dimensional $\Gamma_q(B,S \ten H)=\Gamma_q^0(B,S \ten H)$.
\end{rem}
\begin{rem} $\Gamma_q^0(B,S \ten H)=(\lbrace s(a,h):a \in B \cup S, h \in H \rbrace)'' \subset (\Gamma_q(\ell^2 \ten H) \bar{\ten} D)^{\omega}$.
\end{rem}
\begin{rem} $\Gamma_q(B,S\ten H)$ is a von Neumann algebra. Indeed, since the map 
\[E=(E_{\Gamma_q(\ell^2_n \ten H)}\ten id)_n:(\Gamma_q(\ell^2 \ten K) \bar{\ten} D))^{\omega} \to (\Gamma_q(\ell^2 \ten H) \bar{\ten} D)^{\omega}\] 
is a normal linear projection (i.e. idempotent map) of norm one, it follows that $\Gamma_q(B, S\ten H)$ is an ultraweakly closed, self-adjoint subspace of $(\Gamma_q(\ell^2 \ten H)\bar{\ten}D)^{\omega}$ containing the identity. It's straightforward to see that the map $E$ has the following bimodularity property: $E(x)E(y)E(z)=E(E(x)yE(z))$, for all $x,y,z \in \Gamma_q^0(B,S\ten K)$. Thus, for $x, y \in \Gamma_q^0(B,S\ten K)$ we have $E(x)E(y)=E(E(x)y) \in \Gamma_q(B, S\ten H)$.
\end{rem}
The canonical generators $s_q(a,h)$ are not easy to work with in a variety of situations. The classical $q$-gaussians possess a system of generators, the so-called Wick words, whose linear span is an ultraweakly dense *-subalgebra. Generalized $q$-gaussians also have such a well-behaved system of linear generators, which will be called Wick words by analogy with the classical case.
In order to find these Wick words let us first define, for every $n \in \nz$, $x \in A$ and $h \in H$,
\[u_n(x,h)= n^{-\frac{1}{2}}(\sum_{j=1}^n s(e_j \ten h) \ten \pi_j(x)) \in \Gamma_q(\ell^2 \ten H) \bar{\ten} D.\]
It's easy to see that $s(x,h)=(u_n(x,h))_n \in (\Gamma_q(\ell^2 \ten H) \bar{\ten} D)^{\omega}$, for $x \in A, h \in H$. 
For $x_1,\ldots, x_m \in BSB=\lbrace b_1ab_2:b_1,b_2 \in B,a \in S \rbrace$ and $h_1,\ldots, h_m \in H$ we will analyze the product
 \begin{align*}
  & u_n(x_1,h_1)\cdots u_n(x_m,h_m)\\
 & =  n^{-\frac{m}{2}} \sum_{1\leq j_1,...,j_m \leq n} s(e_{j_1}\ten h_1)\cdots s(e_{j_m}\ten h_m)
 \ten \pi_{j_1}(x_1)\pi_{j_2}(x_2)\cdots \pi_{j_m}(x_m) \\
 &= \sum_{\si \in P(m)}  (n^{-\frac{m}{2}} \sum_{(j_1,...,j_m)=\si,1 \leq j_1,\ldots,j_m \leq n} s(e_{j_1}\ten h_1)\cdots s(e_{j_m}\ten h_m)\ten \pi_{j_1}(x_1)\cdots \pi_{j_m}(x_m)) \pl,
 \end{align*}
For $\si \in P(m)$ let's define
 \[ x_{\si}^n(x_1,h_1,\ldots ,x_m,h_m) \lel
 n^{-m/2}
 \sum_{1\leq j_1,\ldots,j_m \leq n,(j_1,...,j_m)=\si} s(e_{j_1}\ten h_1)\cdots s(e_{j_m}\ten h_m)\ten \pi_{j_1}(x_1)\cdots \pi_{j_m}(x_m) \pl ,\]
 and $x_{\si}(x_1,h_1,\ldots, x_m, h_m)=(x_{\si}^n(x_1,h_1,\ldots,x_m,h_m))_n \in (\Gamma_q(\ell^2 \ten H)\bar{\ten} D)^{\omega}$. To keep the notation less cumbersome, we will omit the parameters $x_k, h_k$ whenever they are clearly understood from the context. Next we see that 
 \[u_n(x_1,h_1)\ldots u_n(x_m,h_m)=\sum_{\si \in P(m)} x_{\si}^n,\]
 and also
 \[s(x_1,h_1)\ldots s(x_m,h_m)=(u_n(x_1,h_1)\ldots u_n(x_m,h_m))_n=\sum_{\si \in P(m)} x_{\si}.\]

\begin{lemma}\label{vn1} Let $(\pi_j,B,A,D)$ be a sequence of symmetric copies. Then 
 \begin{enumerate}
\item[o)] $\sup_n \|x_{\si}^n\|_{\infty} < \infty$ for all $m \geq 1$ and $\si \in P_{1,2}(m)$;
\item[i)] If $\si \notin P_{1,2}(m)$ and $0<p<\infty$ then
 \[ \lim_n \| x_{\si}^n\|_p \lel 0 \pl.\]
In particular $s(x_1,h_1)\ldots s(x_m,h_m)=\sum_{\si \in P_{1,2}(m)} x_{\si}$.
\end{enumerate}
\end{lemma}

\begin{proof} The proof is the same as that of Prop. 4.1. in \cite{JLU}.
\qd
\begin{prop} We have the following convolution formula for the multiplication of Wick words:
\begin{align*}
& x_{\si}(x_1,h_1,\ldots,x_m,h_m)x_{\theta}(y_1,k_1,\ldots,y_{m'},k_{m'})=\\
& \sum_{\gamma \in P_{1,2}(m+m'), \gamma_p|_{1...m}=\si_p, \gamma_p|_{1...m'}=\theta_p} x_{\gamma}(x_1,h_1,\ldots,y_{m'},k_{m'}).
\end{align*}
Moreover, item i) in the lemma above shows that in the summation we can restrict ourselves to pair-singleton partitions whose only additional pairings are between the singletons of $\si$ and $\theta$. In particular, the linear span of the Wick words is a *-algebra.
\end{prop}
\begin{proof} We have 
\begin{align*}
& x_{\si}(x_1,h_1,\ldots,x_m,h_m)x_{\theta}(y_1,k_1,\ldots,y_{m'},k_{m'})= \\
& (n^{-\frac{m+m'}{2}}\sum_{(j_1,\ldots,j_m)=\si,(l_1,\ldots,l_{m'})=\theta} s(e_{j_1}\ten h_1)\cdots s(e_{l_{m'}}\ten k_{m'}) \ten \pi_{j_1}(x_1)\cdots \pi_{l_{m'}}(y_{m'})) = \\
& \sum_{\gamma \in P_{1,2}(m+m'), \gamma_p|_{\lbrace 1,\ldots,m \rbrace}=\si_p, \gamma_p|_{\lbrace 1,\ldots,m' \rbrace}=\theta_p} (n^{-\frac{m+m'}{2}}\sum_{(j_1,\ldots,l_{m'})=\gamma} s(e_{j_1}\ten h_1)\cdots s(e_{l_{m'}}\ten k_{m'}) \\
& \ten \pi_{j_1}(x_1)\cdots \pi_{l_{m'}}(y_{m'})) = \\
& \sum_{\gamma \in P_{1,2}(m+m'), \gamma_p|_{\lbrace 1,\ldots,m \rbrace}=\si_p, \gamma_p|_{\lbrace 1,\ldots,m' \rbrace}=\theta_p} x_{\gamma}(x_1,h_1,\ldots,y_{m'},k_{m'}).
\end{align*}
Now if $\gamma \in P_{1,2}(m+m')$ connects a singleton in $\si$ with a leg of a pair in $\theta$ or the leg of pair in $\si$ with either a singleton or a leg of a pair in $\theta$, the resulting $x_{\gamma}$ is associated to a partition containing a 3-set or a 4-set and hence vanishes according to Lemma 3.8. So in the above sum we may only allow $\gamma$'s which preserve the pair sets of both $\si$ and $\theta$ and can only additionally pair singletons "on different sides of the marker", which ends the proof.
\end{proof}
Our next result provides a reduction method for the Wick words.
\begin{lemma}\label{elim} Let $\pi_j:A \to D$ be symmetric independent copies, and $1 \in S=S^* \subset A$. Let $x_1,\ldots, x_m \in BSB$, $\si \in P_{1,2}(m)$ having $s$ singletons and $p$ pairs and $\phi:\{1,...,m\}\to \{1,...,s+p\}$ which encodes $\si$, i.e. $\phi(k_t)=t$ for every singleton $\lbrace k_t \rbrace \in \si $, $1 \leq t \leq s$ and $\phi(k'_{t})=\phi(k''_{t})=t+s$, for every pair $\lbrace k'_t, k''_t \rbrace \in \si$, $1 \leq t \leq p$. Consider $(\varepsilon_k)$ a sequence of Rademacher variables, i.e. Bernoulli independent random variables on a probability space $(X,\mu)$ satisfying $\varepsilon_k :X \to \lbrace \pm 1 \rbrace$, $\mathbb E(\varepsilon_k=1)=\mathbb E(\varepsilon_k=-1)=\frac{1}{2}$. Then
 \begin{align*}
  &\|\sum_{(l_1,...,l_{s+p})=\dot{0}}
  \eps_{l_1}\cdots \eps_{l_s}\ten (
  \pi_{l_{\phi(1)}}(x_1)\cdots \pi_{l_{\phi(m)}}(x_m) -E_{l_1,...,l_s}(
  \pi_{l_{\phi(1)}}(x_1)\cdots \pi_{l_{\phi(m)}}(x_m))\|_2 \kl \\
  & C(m,x_j) n^{\frac{m-1}{2}} \pl.
  \end{align*}
In particular we have
 \begin{align*}
 & (n^{-\frac{m}{2}}\sum_{(l_1,...,l_s,l_{s+1},...,l_{s+p})=\dot{0}}
  \eps_{l_1}\cdots \eps_{l_s}\ten (\pi_{l_{\phi(1)}}(x_1)\cdots \pi_{l_{\phi(m)}}(x_m))\\
  &= (n^{-\frac{s}{2}}\sum_{(l_1,...,l_s)=\dot{0}}
   \eps_{l_1}\cdots \eps_{l_s}\ten E_{l_1,...,l_s}(\pi_{l_{\phi(1)}}(x_1)\cdots \pi_{l_{\phi(m)}}(x_m))= \\
& (n^{-\frac{s}{2}}\sum_{(l_1,...,l_s)=\dot{0}}
   \eps_{l_1}\cdots \eps_{l_s}\ten \al_{l_1,\ldots,l_s}(F_{\si}(x_1,\ldots,x_m))),
 \end{align*}
 where $F_{\si}(x_1,\ldots x_m)=E_{1,\ldots,s}(\pi_{\phi(1)}(x_1)\cdots \pi_{\phi(m)}(x_m))$
 and the last equality takes place in $(L^{\infty}(X)\bar{\ten} D)^{\omega}$.
\end{lemma}
\begin{proof}
Throughout the proof we endow $L^{\infty}(X) \bar{\ten} D$ with the natural trace $\mu \ten \tau$, where $\tau$ is the faithful trace on $D$. The $\|\cdot\|_2$ in the first statement is the one corresponding to $\mu \ten \tau$. The approach we take is somewhat similar to the one in \cite{JungeZeng}. 
\par {\bf Step 1.} Let $x_1,\ldots, x_m \in BSB$ and $n$ be fixed. Consider 
\[\Omega_n = \lbrace (C_1,\ldots,C_{s+p}):C_1 \sqcup \ldots \sqcup C_{s+p}=\lbrace 1, \ldots, n\rbrace, C_i \neq \emptyset, \forall i \rbrace.\]
Make $\Omega_n$ into a probability space with the normalized counting measure. For every $s+p$-tuple $(l_1,\ldots,l_{s+p})=\dot{0}$, consider the indicator function $\delta_{l_1,\ldots,l_{s+p}}:\Omega_n \to \lbrace 0,1 \rbrace$ which is $1$ if $l_i \in C_i$ for all $1 \leq i \leq s+p$ and $0$ otherwise. According to the proof of Lemma 3.6 in \cite{JungeZeng},
$\mathbb E(\delta_{l_1,\ldots,l_{s+p}})=(s+p)^{-s-p}=C$. Put

\[F(l_1,\ldots,l_{s+p})=\eps_{l_1}\cdots \eps_{l_s}\ten (\pi_{l_{\phi(1)}}(x_1)\cdots \pi_{l_{\phi(m)}}(x_m) -E_{l_1,...,l_s}(\pi_{l_{\phi(1)}}(x_1)\cdots \pi_{l_{\phi(m)}}(x_m)).\]
Then we have
\begin{align*}
& \|\sum_{(l_1,\ldots,l_{s+p})=\dot{0}} F(l_1,\ldots,l_{s+p})\|_2 = C^{-1} \|C\sum_{(l_1,\ldots,l_{s+p})=\dot{0}} F(l_1,\ldots,l_{s+p})\|_2 = \\
& C^{-1}\|\sum_{(l_1,\ldots,l_{s+p})=\dot{0}}C F(l_1,\ldots,l_{s+p})\|_2 = C^{-1} \|\sum_{(l_1,\ldots,l_{s+p})=\dot{0}}\mathbb E(\delta_{l_1,\ldots,l_{s+p}}) F(l_1,\ldots,l_{s+p})\|_2 =\\
& C^{-1}\|\frac{1}{|\Omega_n|}\sum_{(l_1,\ldots,l_{s+p})=\dot{0}} \sum_{(C_1,\ldots,C_{s+p}) \in \Omega_n} \delta_{l_1,\ldots,l_{s+p}}((C_1,\ldots,C_{s+p}))F(l_1,\ldots,l_{s+p})\|_2=\\
& C^{-1}\|\frac{1}{|\Omega_n|} \sum_{(C_1,\ldots,C_{s+p})\in \Omega_n} \sum_{l_1 \in C_1,\ldots,l_{s+p}\in C_{s+p}} F(l_1,\ldots,l_{s+p})\|_2 = C^{-1} \|\mathbb E(G)\|_2 \\
& \leq C \sup_{(C_1,\ldots,C_{s+p})\in \Omega_n}\|G((C_1,\ldots,C_{s+p}))\|_2=C\sup_{(C_1,\ldots,C_{s+p})\in \Omega_n} \|\sum_{l_1 \in C_1,\ldots,l_{s+p}\in C_{s+p}}F(l_1,\ldots,l_{s+p})\|_2,\\
\end{align*}
where we define $G:\Omega_n \to L^{\infty}(X) \bar{\ten} D$ by $G((C_1,\ldots,C_{s+p}))=\sum_{l_1 \in C_1,\ldots,l_{s+p} \in C_{s+p}} F(l_1,\ldots,l_{s+p})$.
\par {\bf Step 2.} It suffices thus to estimate $\|\sum_{l_1 \in C_1,\ldots,l_{s+p}\in C_{s+p}} F(l_1,\ldots,l_{s+p})\|_2$, for a fixed non-degenerate partition $C_1,\ldots,C_{s+p}$ of $\lbrace 1,\ldots,n \rbrace$. Fix such an arbitrary partition. We define the sets $I_l=C_1\cup \cdots C_{s+p-1} \cup(\{1,...,l\}\cap C_{s+p})$ and for $l \in C_{s+p}$
 \begin{align*}
  & d_{l} \lel
 \sum_{l_1\in C_1,...,l_s\in C_s,l_{s+1}\in C_{s+1},...,l_{p-1}\in C_{s+p-1}} \eps_{l_1}\cdots \eps_{l_s} \ten (\pi_{l_{\phi(1)}}(x_1)\cdots \pi_l(x_{k'_p})\cdots \pi_l(x_{k''_p}) \cdots \pi_{l_{\phi(m)}}(x_m)) \\
 & -E_{I_{l-1}}((\pi_{l_{\phi(1)}}(x_1)\cdots \pi_l(x_{k'_p})\cdots \pi_l(x_{k''_p}) \cdots \pi_{l_{\phi(m)}}(x_m))) \pl .
 \end{align*}
Note that $D_l=L^{\infty}(X) \bar{\ten} A_{I_l}$, $l \in C_{s+p}$, form an increasing finite sequence of von Neumann subalgebras of $L^{\infty}(X) \bar{\ten} D$. Now $d_l \in D_l$ and $E_{D_{l-1}}(d_l)=0$, for all $l \in C_{s+p}$. The orthogonality together with the Cauchy-Schwartz inequality yields
\[\|\sum_{l \in C_{s+p}} d_l\|_2 \leq n^{\frac{1}{2}}\sup_{l \in C_{s+p}}\|d_l\|_2.\]
On the other hand, since the products $\eps_{l_1}\cdots \eps_{l_s}$ are mutually orthogonal for different s-tuples $(l_1,\ldots,l_s)$, we see that
\begin{align*}
& \|d_l\|_2 = \|\sum_{l_1\in C_1,...,l_s\in C_s,l_{s+1}\in C_{s+1},...,l_{s+p-1}\in C_{s+p-1}} \eps_{l_1}\cdots \eps_{l_s} \ten \\
 &\quad    (\pi_{l_{\phi(1)}}(x_1)\cdots \pi_l(x_{k'_p})\cdots \pi_l(x_{k''_p}) \cdots \pi_{l_{\phi(m)}}(x_m)) -E_{I_{l-1}}((\pi_{l_{\phi(1)}}(x_1)\cdots \pi_l(x_{k'_p})\cdots \pi_l(x_{k''_p}) \cdots \pi_{l_{\phi(m)}}(x_m)))\|_2 \\
& = \|\sum_{l_1 \in C_1,\ldots,l_s \in C_s} \eps_{l_1}\cdots \eps_{l_s} \ten (\sum_{l_{s+1} \in C_{s+1},\ldots,l_{s+p-1}\in C_{s+p-1}}\\
& \quad (\pi_{l_{\phi(1)}}(x_1)\cdots \pi_l(x_{k'_p})\cdots \pi_l(x_{k''_p}) \cdots \pi_{l_{\phi(m)}}(x_m)) -E_{I_{l-1}}((\pi_{l_{\phi(1)}}(x_1)\cdots \pi_l(x_{k'_p})\cdots \pi_l(x_{k''_p}) \cdots \pi_{l_{\phi(m)}}(x_m)))\|_2 \\
& \leq n^{\frac{s}{2}}\|\sum_{l_{s+1} \in C_{s+1},\ldots,l_{s+p-1}\in C_{s+p-1}} (\pi_{l_{\phi(1)}}(x_1) \cdots \pi_{l_{\phi(m)}}(x_m)) -E_{I_{l-1}}((\pi_{l_{\phi(1)}}(x_1) \cdots \pi_{l_{\phi(m)}}(x_m)))\|_2 \leq \\
& \leq n^{\frac{s}{2}}n^{p-1}\|\pi_{l_{\phi(1)}}(x_1)\cdots \pi_{l_{\phi(m)}}(x_m)\|_{\infty} \leq n^{\frac{m-2}{2}}\|x_1\|_{\infty}\cdots \|x_m\|_{\infty}.
\end{align*}
According to axiom (3) we have
 \[ E_{I_{l-1}}((\pi_{l_{\phi(1)}}(x_1)..\pi_l(x_{k'_p})..\pi_l(x_{k''_p}).. \pi_{l_{\phi(m)}}(x_m))
 \lel E_{C_1,...,C_{s+p-1}}((\pi_{l_{\phi(1)}}(x_1)..\pi_l(x_{k'_p}).. \pi_l(x_{k''_p})..\pi_{l_{\phi(m)}}(x_m)))\]
hence
\begin{align*}
& \|\sum_{l_1 \in C_1,...,l_{s+p}\in C_{s+p}} \eps_{l_1}...\eps_{l_s} \ten (\pi_{l_{\phi(1)}}(x_1)...\pi_{l_{\phi(m)}}(x_m))-E_{C_1 \cup \cdots \cup C_{s+p-1}}(\pi_{l_{\phi(1)}}(x_1)...\pi_{l_{\phi(m)}}(x_m)))\|_2 = \\
& \|\sum_{l_1\in C_1,...,l_{s+p}\in C_{s+p}} \eps_{l_1}\cdots \eps_{l_s} \ten    (\pi_{l_{\phi(1)}}(x_1)\cdots \pi_{l_{\phi(m)}}(x_m)) -E_{I_{l-1}}((\pi_{l_{\phi(1)}}(x_1)\cdots \pi_{l_{\phi(m)}}(x_m)))\|_2 = \|\sum_{l \in C_{s+p}} d_l \|_2\\
& \leq \|x_1\|_{\infty}\cdots \|x_m\|_{\infty}n^{\frac{m-1}{2}}=C'(x_1,\ldots,x_m)n^{\frac{m-1}{2}}.
\end{align*}
Steps 1 and 2 so far imply that
\[ \|\sum_{l_1 \in C_1,...,l_{s+p}\in C_{s+p}} \eps_{l_1}...\eps_{l_s} \ten (\pi_{l_{\phi(1)}}(x_1)...\pi_{l_{\phi(m)}}(x_m))-E_{C_1 \cup \cdots \cup C_{s+p-1}}(\pi_{l_{\phi(1)}}(x_1)...\pi_{l_{\phi(m)}}(x_m)))\|_2 \leq C'n^{\frac{m-1}{2}}.\]
{\bf Step 3.} Now we may proceed inductively. Denote by $y=\pi_{l_{\phi(1)}}(x_1)\cdots \pi_{l_{\phi(m)}}(x_m)$. Then, using axiom (4) and because the conditional expectations commute, we see that
 \begin{align*}
 & y-E_{C_1\cup\cdots  \cup C_{s+p-2}}(y)\lel
    y-E_{C_1\cup \cdots  \cup C_{s+p-1}}(y)+
 E_{C_1\cup \cdots \cup C_{s+p-1}}(y)-E_{C_1\cup \cdots \cup  C_{s+p-2}}(y) \\
 &= y-E_{C_1\cup \cdots  \cup C_{s+p-1}}(y)
 +
  E_{C_1\cup \cdots \cup C_{s+p-1}}(y)-E_{C_1\cup \cdots \cup_{C_{s+p-2}} \cup  C_{s+p}}(
 E_{C_1\cup \cdots \cup  C_{s+p-1}}(y))=\\
 & = y-E_{C_1\cup \cdots  \cup C_{s+p-1}}(y) + E_{C_1\cup \cdots \cup C_{s+p-1}}(y - E_{C_1\cup \cdots \cup C_{s+p-2} \cup C_{s+p}}(y)).
 \end{align*}
Using the previous steps and the fact that the conditional expectations are $\|\cdot\|_2$-contractive, we obtain
\begin{align*}
& \|\sum_{l_1 \in C_1,\ldots,l_{s+p}\in C_{s+p}} \eps_{l_1}\cdots \eps_{l_s} \ten (\pi_{l_{\phi(1)}}(x_1)\cdots \pi_{l_{\phi(m)}}(x_m)-E_{C_1 \cup \cdots \cup C_{s+p-2}}(\pi_{l_{\phi(1)}}(x_1)\cdots \pi_{l_{\phi(m)}}(x_m)))\|_2 \leq \\
& \|\sum_{l_1 \in C_1,\ldots,l_{s+p}\in C_{s+p}} \eps_{l_1}\cdots \eps_{l_s} \ten (\pi_{l_{\phi(1)}}(x_1)\cdots \pi_{l_{\phi(m)}}(x_m)-E_{C_1 \cup \cdots \cup C_{s+p-1}}(\pi_{l_{\phi(1)}}(x_1)\cdots \pi_{l_{\phi(m)}}(x_m)))\|_2 + \\
& + \|(id \ten E_{C_1 \cup \cdots \cup C_{s+p-1}})(\sum_{l_1 \in C_1,\ldots,l_{s+p}\in C_{s+p}} \eps_{l_1}\cdots \eps_{l_s} \ten \\
& (\pi_{l_{\phi(1)}}(x_1)\cdots \pi_{l_{\phi(m)}}(x_m)-E_{C_1 \cup \cdots \cup C_{s+p-2}\cup C_{s+p}}(\pi_{l_{\phi(1)}}(x_1)\cdots \pi_{l_{\phi(m)}}(x_m))))\|_2 \leq \\
& \|\sum_{l_1 \in C_1,\ldots,l_{s+p}\in C_{s+p}} \eps_{l_1}\cdots \eps_{l_s} \ten (\pi_{l_{\phi(1)}}(x_1)\cdots \pi_{l_{\phi(m)}}(x_m)-E_{C_1 \cup \cdots \cup C_{s+p-1}}(\pi_{l_{\phi(1)}}(x_1)\cdots \pi_{l_{\phi(m)}}(x_m)))\|_2 + \\
& \|\sum_{l_1 \in C_1,\ldots,l_{s+p}\in C_{s+p}} \eps_{l_1}\cdots \eps_{l_s} \ten (\pi_{l_{\phi(1)}}(x_1)\cdots \pi_{l_{\phi(m)}}(x_m)-E_{C_1 \cup \cdots \cup C_{s+p-2}\cup C_{s+p}}(\pi_{l_{\phi(1)}}(x_1)\cdots \pi_{l_{\phi(m)}}(x_m)))\|_2 \leq \\
& \leq 2C'n^{\frac{m-1}{2}}.
\end{align*}
After using the triangle inequality $p$ times, we get
\begin{align*}
& \|\sum_{l_1 \in C_1,\ldots,l_{s+p}\in C_{s+p}} \eps_{l_1}\cdots \eps_{l_s} \ten (\pi_{l_{\phi(1)}}(x_1)\cdots \pi_{l_{\phi(m)}}(x_m)-E_{C_1 \cup \cdots \cup C_{s}}(\pi_{l_{\phi(1)}}(x_1)\cdots \pi_{l_{\phi(m)}}(x_m)))\|_2 \leq \\
& \leq pC'n^{\frac{m-1}{2}}=C''n^{\frac{m-1}{2}}.
\end{align*}
Now we claim that
 \[ E_{C_1\cup \cdots \cup C_s}(\pi_{l_{\phi(1)}}(x_1)\cdots \pi_{l_{\phi(m)}}(x_m)) \lel E_{\lbrace l_1,...,l_s \rbrace}(\pi_{l_{\phi(1)}}(x_1)\cdots \pi_{l_{\phi(m)}}(x_m)) \pl .\]
This can be established using axioms 3 and 4. Indeed, since $l_{s+p} \notin C_1 \cup \cdots \cup C_{s+p-1} \supset \lbrace l_1,\ldots,l_{s+p-1}\rbrace$, by applying axiom 3 we see that
\begin{align*}
& E_{\lbrace l_1,\ldots,l_{s+p-1}\rbrace}(\pi_{l_{\phi(1)}}(x_1)\cdots \pi_{l_{\phi(m)}}(x_m))=\pi_{l_{\phi(1)}}(x_1)\cdots E_{\lbrace l_1,\ldots,l_{s+p-1}\rbrace}(\pi_{l_{s+p}}(x_{k'_p})\cdots \pi_{l_{s+p}}(x_{k''_p}))\cdots \pi_{l_{\phi(m)}}(x_m)\\
& =\pi_{l_{\phi(1)}}(x_1)..E_{C_1 \cup \ \cdots \cup C_{s+p-1}}(\pi_{l_{s+p}}(x_{k'_p})\cdots \pi_{l_{s+p}}(x_{k''_p}))..\pi_{l_{\phi(m)}}(x_m)=E_{C_1\cup \cdots \cup C_{s+p-1}}(\pi_{l_{\phi(1)}}(x_1)\cdots \pi_{l_{\phi(m)}}(x_m)),
\end{align*}
and then
\begin{align*}
& E_{C_1\cup \cdots \cup C_s}(\pi_{l_{\phi(1)}}(x_1)\cdots \pi_{l_{\phi(m)}}(x_m))=E_{C_1\cup \cdots \cup C_s}(E_{C_1\cup \cdots \cup C_{s+p-1}}((\pi_{l_{\phi(1)}}(x_1)\cdots \pi_{l_{\phi(m)}}(x_m))))=\\
& E_{C_1\cup \cdots \cup C_s}(E_{l_1,\ldots,l_{s+p-1}}(\pi_{l_{\phi(1)}}(x_1)\cdots \pi_{l_{\phi(m)}}(x_m)))=E_{(C_1\cup \cdots \cup C_s)\cap \lbrace l_1,\ldots,l_{s+p-1} \rbrace}(\pi_{l_{\phi(1)}}(x_1)\cdots \pi_{l_{\phi(m)}}(x_m))=\\
& =E_{\lbrace l_1,\ldots,l_s \rbrace}(\pi_{l_{\phi(1)}}(x_1)\cdots \pi_{l_{\phi(m)}}(x_m)),
\end{align*} 
which proves the claim. Now the claim together with the last inequality entail
\[\|\sum_{l_1 \in C_1,\ldots,l_{s+p}\in C_{s+p}} \eps_{l_1}\cdots \eps_{l_s} \ten (\pi_{l_{\phi(1)}}(x_1)\cdots \pi_{l_{\phi(m)}}(x_m)-E_{\lbrace l_1,\ldots,l_s \rbrace}(\pi_{l_{\phi(1)}}(x_1)\cdots \pi_{l_{\phi(m)}}(x_m)))\|_2 \leq C''n^{\frac{m-1}{2}}.\] 
Step 1 now implies
\[\|\sum_{(l_1,\ldots,l_{s+p})=\dot{0}} \eps_{l_1}\cdots \eps_{l_s} \ten (\pi_{l_{\phi(1)}}(x_1)\cdots \pi_{l_{\phi(m)}}(x_m)-E_{\lbrace l_1,\ldots,l_s \rbrace}(\pi_{l_{\phi(1)}}(x_1)\cdots \pi_{l_{\phi(m)}}(x_m)))\|_2 \leq CC''n^{\frac{m-1}{2}},\] 
which proves the first statement in the lemma. For the second statement, we first note that $E_{\lbrace l_1,\ldots,l_s \rbrace}(\pi_{l_{\phi(1)}}(x_1)\cdots \pi_{l_{\phi(m)}}(x_m))$ only depends on $l_1,\ldots,l_s$, and not on $l_{s+1},\ldots,l_{s+p}$. Indeed, let $(l_1,\ldots,l_s,l'_{s+1},\ldots,l'_{s+p})=\dot{0}$ another $s+p$-tuple with the same first $s$ entries. Take a finite permutation $\si$ such that $\si(l_i)=l_i, i \leq s$ and $\si(l_{s+i})=l'_{s+i},i\leq p$. Then $\al_{\si}$ is the identity on $A_{l_1,\ldots,l_s}$, hence
\begin{align*}
& E_{\lbrace l_1,\ldots,l_s \rbrace}(\pi_{l_{\phi(1)}}(x_1)\cdots \pi_{l_{s+i}}(x_{k'_i}) \cdots \pi_{l_{s+i}}(x_{k''_i}) \cdots \pi_{l_{\phi(m)}}(x_m))=\\
& (E_{\lbrace l_1,\ldots,l_s \rbrace} \circ \al_{\si})(\pi_{l_{\phi(1)}}(x_1)\cdots \pi_{l_{s+i}}(x_{k'_i}) \cdots \pi_{l_{s+i}}(x_{k''_i}) \cdots \pi_{l_{\phi(m)}}(x_m))=\\
& E_{\lbrace l_1,\ldots,l_s \rbrace}(\pi_{l_{\phi(1)}}(x_1)\cdots \pi_{l'_{s+i}}(x_{k'_i}) \cdots \pi_{l'_{s+i}}(x_{k''_i}) \cdots \pi_{l_{\phi(m)}}(x_m)),
\end{align*}
which proves the claim. Now the first statement of the lemma together with an easy counting argument shows that
\begin{align*}
& (n^{-\frac{m}{2}}\sum_{(l_1,\ldots,l_{s+p})=\dot{0}} \eps_{l_1}\cdots \eps_{l_s} \ten \pi_{l_{\phi(1)}}(x_1)\cdots \pi_{l_{\phi(m)}}(x_m))=\\
& = (n^{-\frac{m}{2}}\sum_{(l_1,\ldots,l_{s+p})=\dot{0}} \eps_{l_1}\cdots \eps_{l_s} \ten E_{\lbrace l_1,\ldots,l_s \rbrace}(\pi_{l_{\phi(1)}}(x_1)\cdots \pi_{l_{\phi(m)}}(x_m)))=\\
& (n^{-\frac{s}{2}}\sum_{(l_1,\ldots,l_{s})=\dot{0}} \eps_{l_1}\cdots \eps_{l_s} \ten E_{\lbrace l_1,\ldots,l_s \rbrace}(\pi_{l_{\phi(1)}}(x_1)\cdots \pi_{l_{\phi(m)}}(x_m))).
\end{align*}
Finally, let's note that
\[E_{l_1,\ldots,l_s} \circ \al_{l_1,\ldots,l_{s+p}}=\al_{l_1,\ldots,l_s} \circ E_{1,\ldots,s},\]
which implies 
\begin{align*}
& E_{l_1,\ldots,l_s}(\pi_{l_{\phi(1)}}(x_1) \cdots \pi_{l_{\phi(m)}}(x_m)) = (E_{l_1,\ldots,l_s}\circ \al_{l_1,\ldots,l_{s+p}})(\pi_{\phi(1)}(x_1)\cdots \pi_{\phi(m)}(x_m))= \\
& (\al_{l_1,\ldots,l_s}\circ E_{1,\ldots,s})(\pi_{\phi(1)}(x_1) \cdots \pi_{\phi(m)}(x_m))= \al_{l_1,\ldots,l_s}(F_{\si}(x_1,\ldots,x_m)).
\end{align*}
\qd
\begin{theorem}Let $(\pi_j,B,A,D)$ be a sequence of symmetric independent copies, $x_1,\ldots,x_m \in A$, $\si \in P_{1,2}(m)$ having $s$ singletons and $p$ pairs and $\phi:\lbrace 1,\ldots,m \rbrace \to \lbrace 1,\ldots,s+p \rbrace$ which encodes $\si$. Then
\begin{align*}
& x_{\si}(x_1,h_1,\ldots,x_m,h_m)=(n^{-\frac{m}{2}}\sum_{(j_1,\ldots,j_m)=\si} s(e_{j_1} \ten h_1)\cdots s(e_{j_m} \ten h_m)\ten \pi_{j_1}(x_1)\cdots \pi_{j_m}(x_m))=\\
& f_{\si}(h_1,\ldots,h_m)(n^{-\frac{s}{2}}\sum_{(l_1,\ldots,l_s)=\dot{0}} s(e_{l_1} \ten h_{k_1}) \cdots s(e_{l_s} \ten h_{k_s}) \ten \al_{\lbrace l_1,\ldots,l_s \rbrace}(F_{\si}(x_1,\ldots,x_m)))= \\
& = f_{\si}(h_1,\ldots,h_m)W_{\si}(x_1,h_1,\ldots,x_m,h_m).
\end{align*}
where $F_{\si}(x_1,\ldots,x_m)=E_{\lbrace 1,\ldots,s \rbrace}(\pi_{\phi(1)}(x_1)\cdots \pi_{\phi(m)}(x_m))$, $f_{\si}(h_1,\ldots,h_m)= q^{\rm cr(\si)} \prod_{\{k,l\}\in \si} \lan h_k,h_l\ran $ and $\lbrace k_1,\ldots,k_s \rbrace$ are the singletons of $\si$. The elements
\[W_{\si}(x_1,h_1,\ldots,x_m,h_m)=(n^{-\frac{s}{2}}\sum_{(l_1,\ldots,l_s)=\dot{0}} s(e_{l_1} \ten h_{k_1}) \cdots s(e_{l_s} \ten h_{k_s}) \ten \al_{\lbrace l_1,\ldots,l_s \rbrace}(F_{\si}(x_1,\ldots,x_m)))\]
will be called reduced Wick words.
\end{theorem}
\begin{proof}
We will use the previous lemma. Let $\hat{B}=B$, $\hat{A}=\Gamma_q(H)\bar{\ten} A$, $\hat{D}=\Gamma_q(\ell^2 \ten H) \bar{\ten} D$ and $\hat{\pi}_j:\hat{A}\to \Gamma_q(\ell^2 \ten H)\bar{\ten} D$ be the *-homomorphisms given by 
 \[ \hat{\pi}_j(s(h)\ten x) \lel s(e_j\ten h)\ten \pi_j(x) \pl .\] 
Then $(\hat{\pi}_j,B,\hat{A},\hat{D})$ represents a sequence of independent symmetric copies. Moreover, it is easy to see that  $\hat{A}_I = \Gamma_q(\ell^2(I) \ten H)\bar{\ten}A_I$. Now according to the previous lemma we have
\begin{align*}
& (n^{-\frac{m}{2}}\sum_{(j_1,\ldots,j_m)=\si} \eps_{j_{k_1}}\cdots \eps_{j_{k_s}} \ten s(e_{j_1} \ten h_1)\cdots s(e_{j_m} \ten h_m)\ten \pi_{j_1}(x_1)\cdots \pi_{j_m}(x_m))=\\
& (n^{-\frac{m}{2}}\sum_{(l_1,\ldots,l_{s+p})=\dot{0}} \eps_{l_1}\cdots \eps_{l_s} \ten s(e_{l_{\phi(1)}} \ten h_1)\cdots s(e_{l_{\phi(m)}} \ten h_m) \ten \pi_{l_{\phi(1)}}(x_1) \cdots \pi_{l_{\phi(m)}}(x_m)) =\\
& (n^{-\frac{m}{2}}\sum_{(l_1,\ldots,l_{s+p})=\dot{0}} \eps_{l_1}\cdots \eps_{l_s} \ten \hat{\pi}_{l_{\phi(1)}}(s(h_1)\ten x_1)\cdots \hat{\pi}_{l_{\phi(m)}}(s(h_m) \ten x_m))=\\
& (n^{-\frac{s}{2}}\sum_{(l_1,\ldots,l_s)=\dot{0}} \eps_{l_1}\cdots \eps_{l_s} \ten \hat{\al}_{l_1,\ldots,l_s}(\hat{F}_{\si}(s(h_1)\ten x_1,\ldots,s(h_m)\ten x_m)))=\\
& (n^{-\frac{s}{2}}\sum_{(l_1,\ldots,l_s)=\dot{0}} \eps_{l_1}\cdots \eps_{l_s} \ten \hat{\al}_{l_1,\ldots,l_s}(E_{\hat{A}_{1,\ldots,s}}(\hat{\pi}_{\phi(1)}(s(h_1)\ten x_1)\cdots \hat{\pi}_{\phi(m)}(s(h_m) \ten x_m)))= \\
& (n^{-\frac{s}{2}}\sum_{(l_1,\ldots,l_s)=\dot{0}} \eps_{l_1} \cdots \eps_{l_s} \ten \hat{\al}_{l_1,\ldots,l_s}(E_{\Gamma_q(\ell_{s}^2 \ten H) \bar{\ten} A_{1,\ldots,s}} (s(e_{\phi(1)} \ten h_1)\cdots s(e_{\phi(m)}\ten h_m) \ten \pi_{\phi(1)}(x_1) \cdots \pi_{\phi(m)}(x_m))) \\
& = f_{\si}(h_1,\ldots,h_m) (n^{-\frac{s}{2}} \sum_{(l_1,\ldots,l_s)= \dot{0}} \eps_{l_1} \cdots \eps_{l_s} \ten \hat{\al}_{l_1,\ldots,l_s}(s(e_{1} \ten h_{k_1})\cdots s(e_{s} \ten h_{k_s}) \ten E_{1,\ldots,s}(\pi_{\phi(1)}(x_1)\cdots \pi_{\phi(m)}(x_m))) \\
& =f_{\si}(h_1,\ldots,h_m)  (n^{-\frac{s}{2}} \sum_{(l_1,\ldots,l_s)= \dot{0}} \eps_{l_1} \cdots \eps_{l_s} \ten (s(e_{l_1} \ten h_{k_1})\cdots s(e_{l_s} \ten h_{k_s}) \ten \al_{l_1,\ldots,l_s}( E_{1,\ldots,s}(\pi_{\phi(1)}(x_1)\cdots \pi_{\phi(m)}(x_m)))) \\
& =f_{\si}(h_1,\ldots,h_m)(n^{-\frac{s}{2}}\sum_{(l_1,\ldots,l_s)=\dot{0}} \eps_{l_1} \cdots \eps_{l_s} \ten s(e_{l_1}\ten h_{k_1})\cdots s(e_{l_s}\ten h_{k_s}) \ten \al_{l_1,\ldots,l_s}(F_{\si}(x_1,\ldots,x_m))).
\end{align*}
To see why the equality on line 6 is true, note that 
\[s(e_{\phi(1)} \ten h_1)\cdots s(e_{\phi(m)} \ten h_m) =\sum_{\theta \in P_{1,2}(m)} f_{\theta}(e_{\phi(1)} \ten h_1 \ten \cdots \ten e_{\phi(m)} \ten h_m)W((e_{\phi(1)} \ten h_1 \ten \cdots \ten e_{\phi(m)} \ten h_m)_{\theta}),\]
where the notation $()_{\theta}$ means that the pair positions of $\theta$ have been removed. After the application of $E_{\Gamma_q(\ell_{s}^2 \ten H)}$, we see that the only surviving partition is $\theta=\si$ and 
\begin{align*} 
& E_{\Gamma_q(\ell_{s}^2 \ten H)}(s(e_{\phi(1)}\ten h_1)\cdots s(e_{\phi(m)}\ten h_m))=f_{\si}(h_1,\ldots,h_m) W(e_1 \ten h_{k_1} \cdots e_s \ten h_{k_s})= \\ 
& f_{\si}(h_1,\ldots,h_m)s(e_1 \ten h_{k_1})\cdots s(e_s \ten h_{k_s}).
\end{align*}
Now, let's define
\[\tilde s(x,h)=(n^{-\frac{1}{2}} \sum_{j=1}^n \eps_j \ten s(e_j\ten h)\ten \pi_j(x)) \in (L^{\infty}(X) \bar{\ten} \Gamma_q(\ell^2 \ten H) \bar{\ten} D)^{\om}.\] 
We claim that the new Wick words $\tilde x_{\si}$ associated to the variables $\tilde s(x,h)$ have the same moments as $x_{\si}$ and hence they generate an isomorphic von Neumann algebra. Indeed, fix $\si \in P_{1,2}(m)$. Note that for $(l_1,\ldots,l_s)=\dot{0}$, we have $\mu(\eps_{l_1}\cdots \eps_{l_s})=\mu(\eps_{l_1})\cdots \mu(\eps_{l_s})=\delta_{s=0}$, due to the fact that $\eps_j$ are mean-zero, independent random variables. Then
\begin{align*}
& \tau_{\om}(\tilde x_{\si}(x_1,h_1,\ldots,x_m,h_m)) = \\
& \tau_{\om}((n^{-\frac{m}{2}}\sum_{(l_1,\ldots,l_{s+p})=\dot{0}}\eps_{l_1}\cdots \eps_{l_s} \ten s(e_{l_{\phi(1)}} \ten h_1) \cdots s(e_{l_{\phi(m)}}\ten h_m) \ten \pi_{l_{\phi(1)}}(x_1)\cdots \pi_{l_{\phi(m)}}(x_m))= \\
& \lim_n (n^{-\frac{m}{2}}\sum_{(l_1,\ldots,l_{s+p})=\dot{0}}\mu(\eps_{l_1}\cdots \eps_{l_s}) \tau(s(e_{l_{\phi(1)}} \ten h_1) \cdots s(e_{l_{\phi(m)}}\ten h_m)) \tau_D(\pi_{l_{\phi(1)}}(x_1)\cdots \pi_{l_{\phi(m)}}(x_m))= \\
& \delta_{s=0} \lim_n (n^{-\frac{m}{2}}\sum_{(l_1,\ldots,l_{s+p})=\dot{0}} \tau(s(e_{l_{\phi(1)}} \ten h_1) \cdots s(e_{l_{\phi(m)}}\ten h_m)) \tau_D(\pi_{l_{\phi(1)}}(x_1)\cdots \pi_{l_{\phi(m)}}(x_m))= \\
& \delta_{\si \in P_2(m)} \lim_n (n^{-\frac{m}{2}}\sum_{(l_1,\ldots,l_{s+p})=\dot{0}}\tau(s(e_{l_{\phi(1)}} \ten h_1) \cdots s(e_{l_{\phi(m)}}\ten h_m)) \tau_D(\pi_{l_{\phi(1)}}(x_1)\cdots \pi_{l_{\phi(m)}}(x_m))= \\
& \tau_{\om} (x_{\si}(x_1,h_1,\ldots,x_m,h_m)).
\end{align*}
Define $\mathcal M \subset (\Gamma_q(\ell^2 \ten H) \bar{\ten} D)^{\om}$ to be the von Neumann algebra generated by all the Wick words $x_{\si}$. Also define $\tilde{\mathcal M} \subset (L^{\infty}(X) \bar{\ten} \Gamma_q(\ell^2 \ten H) \bar{\ten} D)^{\om}$ to be the von Neumann algebra generated by the elements $\tilde x_{\si}$. Using the claim, the convolution formula and Proposition 3.1 we see that the map
\[\mathcal M \ni \sum x_{\si} \mapsto \sum \tilde x_{\si} \in \tilde{\mathcal M}\]
is a *-isomorphism. Applying the inverse of this isomorphism to the equality
\begin{align*}
& (n^{-\frac{m}{2}}\sum_{(l_1,\ldots,l_{s+p})=\dot{0}} \eps_{l_1} \cdots \eps_{l_s} \ten s(e_{l_{\phi(1)}} \ten h_1)\cdots s(e_{l_{\phi(m)}} \ten h_m) \ten \pi_{l_{\phi(1)}}(x_1) \cdots \pi_{l_{\phi(m)}}(x_m))= \\
& =f_{\si}(h_1,\ldots,h_m)(n^{-\frac{s}{2}}\sum_{(l_1,\ldots,l_s)=\dot{0}}\eps_{l_1} \cdots \eps_{l_s} \ten s(e_{l_1}\ten h_{k_1})\cdots s(e_{l_s}\ten h_{k_s}) \ten \al_{l_1,\ldots,l_s}(F_{\si}(x_1,\ldots,x_m))),
\end{align*}
we obtain the desired identity.
\end{proof}
\begin{prop} Let $x_1,...,x_m\in A$, $h_1,\ldots,h_m \in H$. Then we have the following moment formula
\begin{align*}
 \tau(s(x_1,h_1)\cdots s(x_m,h_m))
 &= \delta_{m\in 2\nz}\sum_{\si\in P_2(m)} q^{\rm cr(\si)} \prod_{\lbrace l,r \rbrace \in \si}\lan h_l,h_r\ran \tau(\pi_{j^{\si}_1}(x_1) \cdots \pi_{j^{\si}_m}(x_m)), \pl
 \end{align*}
 as well as the $B$-valued moment formula
 \[ E_B(s(x_1,h_1)\cdots s(x_m,h_m)) = \delta_{m\in 2\nz}\sum_{\si\in P_2(m)} q^{\rm cr(\si)} \prod_{\lbrace l,r \rbrace \in \si}\lan h_l,h_r\ran E_B(\pi_{j^{\si}_1}(x_1) \cdots \pi_{j^{\si}_m}(x_m)), \pl\]
where for every $\si\in P_2(m)$, the $j_1^{\si},\ldots,j_m^{\si}$ are chosen such that $(j_1^{\si},\ldots,j_m^{\si})=\si$.
\end{prop}
\begin{proof} It's a straightforward application of the reduction formula.
\end{proof}
\begin{rem} Proposition 3.1 shows that $M=\Gamma_q^0(B,S\ten H)$ could be introduced abstractly as the tracial von Neumann algebra $(M,\tau)$ generated by elements $s(x,h), x \in BSB, h \in H$ which satisfy the above moment formula.
\end{rem}
\begin{prop}
Let $K$ be infinite dimensional and $x_1,\ldots,x_m \in BSB$, $h_1,\ldots,h_m \in K$, $\si \in P_{1,2}(m)$.  Then $x_{\si}(x_1,h_1,\ldots,x_m,h_m) \in \Gamma_q^0(B,S\ten K)$. For every Hilbert space $H$, all the Wick words $x_{\si}(x_1,h_1,\ldots,x_m,h_m),x_i \in BSB,h_i \in H$, are in $M=\Gamma_q(B, S\ten H)$. In particular, $M$ is the ultraweakly closed linear span of the (reduced) Wick words and $L^2(M)$ is the $\|\cdot\|_2$-closed span of the (reduced) Wick words.
\end{prop}
\begin{proof}
We need a basic fact about infinite dimensional Hilbert spaces.
\par {\bf Fact.} Let $K$ be an infinite dimensional Hilbert space and $\lambda_1,\ldots, \lambda_p \in \cz$. Then there exist norm bounded sequences $\xi_n^k, \eta_n^k \in K$, for $1\leq k \leq p$ such that $\xi_n^k \to 0$, $\eta_n^k \to 0$ weakly and $\lan\xi_n^k,\eta_n^k\ran=\lambda_k$, for all $1\leq k \leq p$, and moreover $\xi_n^k, \eta_n^k \perp \xi_n^j, \eta_n^j$ for $k \neq j$. Indeed, let $(e_n)$ be an orthonormal infinite sequence in $K$. Define $\xi_n^1=\lambda_1 e_n, \eta_n^1=e_n, \xi_n^2=\lambda_2 e_{n+1}, \eta_n^2=e_{n+1}, \ldots, \xi_n^p=\lambda_p e_{n+p-1}, \eta_n^p=e_{n+p-1}$.
\par To prove the proposition we will use induction on $s$, the numbers of singletons in $\si$. For $s=0$, $x_{\si}(x_1,h_1,\ldots,x_m,h_m)\in B$ due to the Wick word reduction formula, so the statement is trivial. For a given $\si$ with pairs $B_1,...,B_p$ and $B=\{l,r\}$ we use the Fact to find uniformly norm bounded vectors $h_{l,B}(k),h_{r,B}(k) \in K$ which converge to $0$ weakly and such that $\lan h_{l,B}(k),h_{r,B}(k)\ran =\lan h_l,h_r\ran$ for all pairs $B=\lbrace l,r \rbrace$, and such that the $h_{l/r,B}(k)$'s are orthogonal for different pairs $B$. Let us define $\tilde{h}_i(k)=h_i$ for any singleton $\lbrace i \rbrace \in \si$ and $\tilde{h}_i(k)=h_{l/r,B}(k)$ if $i\in B$ and $i=l$ or $i=r$. For every other Wick word $x'_{\si'}(y_1,f_1,\ldots,y_{m'},f_{m'})$, with $y_j \in BSB$, $f_j \in K$, we have
\begin{align*}
& \lim_{k \to \infty}\tau(s(x_1,\tilde{h}_1(k)) \cdots s(x_m,\tilde{h}_m(k))x'_{\si'})=\lim_{k \to \infty}\sum_{\theta \in P_{1,2}(m)} \tau(x_{\theta}(x_1,\tilde{h}_1(k),\ldots,x_m,\tilde{h}_m(k))x_{\si'}')= \\
& \tau(x_{\si}x'_{\si'})+\lim_{k \to \infty}\sum_{\theta_p \supset \si_p,|\theta_s|<s} \tau(x_{\theta}(x_1,\tilde{h}_1(k),\ldots,x_m,\tilde{h}_m(k))x_{\si'}').
\end{align*}
Indeed, for every $\theta \in P_{1,2}(m)$ which does not contain all the pairs of $\si$, we use the convolution and the moment formulas to obtain
\begin{align*}
& \tau(x_{\theta}x'_{\si'})= \sum_{\nu \in P_{2}(m+m')} \tau(x_{\nu}(x_1,\tilde{h}_1(k),\ldots,y_{m'},f_{m'})) = \\
& =\sum_{\nu \in P_{2}(m+m')} f_{\nu}(\tilde{h}_1(k),\ldots,\tilde{h}_m(k),f_1,\ldots,f_{m'})\tau(W_{\nu}(x_1,\tilde{h}_1(k),\ldots,y_{m'},f_{m'})),
\end{align*}
the sum being taken over all $\nu$ that preserve the pairs of $\theta$ and $\si'$ and additionally pair all the singletons of $\theta$ and $\si'$. Now since $\theta$ does not contain all the pairs of $\si$ there must be a leg $l$ of a pair $\lbrace l,r \rbrace=B \in \si$ which is connected by $\theta$ to something else than its other leg in $\si$. There are three possibilities:
\begin{enumerate}
\item $\theta$ connects $l$ to a leg $l'$ of another pair $B'=\lbrace l',r' \rbrace \in \si$. Then $\lan\tilde h_l(k),\tilde h_{l'}(k)\ran =0$, hence for every $\nu$ in the sum above we have $f_{\nu}(h_1,\ldots,f_{m'})=0$.
\item $\theta$ connects $l$ to a singleton $\lbrace i \rbrace \in \si$. Then, since $\tilde h_l(k) \to 0$ weakly, we have $\lan \tilde h_l(k),h_i\ran \to 0$, hence for every $\nu$ we also have that $f_{\nu}(h_1,\ldots,f_{m'})\to 0$ as $k \to \infty$.
\item $\lbrace l \rbrace$ is a singleton of $\theta$. In this case, every $\nu \in P_{1,2}(m+m')$ which appears in the sum has to connect $l$ to a singleton $j \in \lbrace 1,\ldots,m' \rbrace$. Thus, $\lan \tilde h_l(k),f_j\ran\to 0$ and again $f_{\nu}(h_1,\ldots,f_{m'})\to 0$ as $k \to \infty$.
\end{enumerate}
Summing up, we see that for every $\theta$ such that $\si_p \nsubseteq \theta_p$, we have $\tau(x_{\theta}(k)x'_{\si'}) \to 0$ as $k \to \infty$.
 Thus, when letting $k \to \infty$, only those $\theta$'s containing the pairs of $\si$ make a non-zero contribution. Among them, there is exactly one which has $s$ singletons, namely $\si$, all the others have more pairs and hence less than $s$ singletons.
We deduce that 
 \[x_{\si}= w-\lim_{k \to \infty}(s(x_1,\tilde{h}_1(k)) \cdots s(x_m,\tilde{h}_m(k))-\sum_{\theta_p \supset \si_p,|\theta_s|<s} x_{\theta}(x_1,\tilde{h}_1(k),\ldots,x_m,\tilde{h}_m(k))).\]
Since by the induction hypothesis all the $x_{\theta}$'s, with $|\theta_s|<s$, are in $\Gamma_q^0(B,S\ten K)$, this proves the statement. For the second statement, let $H$ be any Hilbert space and $K$ an infinite dimensional Hilbert space containing $H$. Let $x_i \in BSB$, $h_i \in H$ and $\si \in P_{1,2}(m)$. Then, by the first part, $x_{\si}(x_1,h_1,\ldots,x_m,h_m) \in \Gamma_q^0(B,S\ten K)$. But $x_{\si}=(E_{\Gamma_q(\ell^2_n \ten H)} \ten id)_n(x_{\si})$, hence $x_{\si} \in \Gamma_q(B,S\ten H)$.
\end{proof}
\begin{rem} The reader can now better appreciate why we needed the ''closure operation" in the definition of $\Gamma_q(B,S\ten H)$. Indeed, Definition 3.4 ensures that the Wick words belong to $M=\Gamma_q(B,S\ten H)$ for every Hilbert space $H$, finite or infinite dimensional. Also, Prop. 3.14 shows that $M=\Gamma_q(B,S\ten H)$ could have been defined as the ultra-weakly closed span of the Wick words.
\end{rem}
In the following we use the notation $L^2_k(M)$ for the $\|\cdot\|_2$-closed span of the Wick words of degree $k$ and $W_k(M)$ for the linear span of the Wick words of degree $k$.
\begin{theorem} Let $(\pi_j,B,A,D)$ be a sequence of symmetric independent copies, $1 \in S=S^* \subset A$, $H$ a Hilbert space and $M=\Gamma_q(B,S \ten H)$. Denote by $\tilde H=H \oplus H$. Take an infinite dimensional Hilbert space $K \supset H$ and denote by $\tilde K = K \oplus K$.
\begin{enumerate}
\item For every angle $\theta$, let $o_{\theta}$ be the canonical rotation on $\tilde K$. Then
\[\theta \mapsto \al_{\theta}=(\Gamma_q(id \ten o_{\theta}) \ten id)_n \in Aut((\Gamma_q(\ell^2 \ten \tilde K) \bar{\ten} D)^{\omega})\]
defines by restriction a one parameter group of automorphisms of $\tilde M=\Gamma_q(B,S\ten \tilde H)$. Moreover, for every Wick word $x_{\si}(x_1,\tilde{h}_1,\ldots,x_m,\tilde{h}_m) \in \tilde M$ we have
\[\al_{\theta}(x_{\si}(x_1,\tilde{h}_1,\ldots,x_m,\tilde{h}_m))=x_{\si}(x_1,o_{\theta}(\tilde{h}_1),\ldots,x_m,o_{\theta}(\tilde{h}_m)).\]
\item For every Wick word $x_{\si}(x_1,h_1,\ldots,x_m,h_m) \in M$, the following formula holds
\[(E_M \circ \al_{\theta})(x_{\si}(x_1,h_1,\ldots,x_m,h_m))=(\cos(\theta))^s x_{\si}(x_1,h_1,\ldots,x_m,h_m),\]
where $E_M:\tilde M \to M$ is the conditional expectation and $s$ is the number of singletons of $\si$.
\item For every $\theta \in [0,\frac{\pi}{2})$, let $t=-$ln$(\cos(\theta))$. Then $t \mapsto T_t = E_M \circ \al_{\theta}|_M$ defines a one parameter semi-group of normal, trace preserving, ucp maps on $M$. Moreover, for every Wick word $x_{\si}\in M$ we have $T_t(x_{\si})= e^{-ts}x_{\si}$, where $s$ is the number of singletons of $\si$. Hence, when viewed as a contraction on $L^2(M)$, we have $T_t=\sum_{s\geq 0} e^{-ts}P_s$, where $P_s$ is the orthogonal projection of $L^2(M)$ on $L^2_s(M)$ and the series is $\|\cdot\|_{\infty}$-convergent, for every $t>0$. In particular, if $L^2_s(M)$ is finitely generated as a right $B$ module for every $s$, then $T_t$ is compact over $B$ for every $t>0$.
\item The generator $N$ of $T_t$ is a positive, self-adjoint, densely defined operator in $L^2(M)=\bigoplus_{k=0}^{\infty} L^2_k(M)$, acting by
\[N(x_{\si}(x_1,h_1,\ldots,x_m,h_m))=kx_{\si}(x_1,h_1,\ldots,x_m,h_m),\]
for every $x_{\si}(x_1,h_1,\ldots,x_m,h_m) \in L^2_k(M)$. The spectrum of N is the set of non-negative integers $\nz$, all of which are eigenvalues. N is called the number operator.
\end{enumerate}
\end{theorem}
\begin{proof} The formula $\al_{\theta}(x_{\si}(x_1,\tilde{h}_1,\ldots,x_m,\tilde{h}_m))=x_{\si}(x_1,o_{\theta}(\tilde{h}_1),\ldots,x_m,o_{\theta}(\tilde{h}_m))$, for $x_i \in BSB, \tilde{h}_i \in \tilde H$, is easily checked, due to entry-wise functoriality, and it shows that $\al_{\theta}$ restricts to a one -parameter group of automorphisms on $\tilde M=\Gamma_q(B,S\ten \tilde H)$. This proves (1). Then, using the reduction formula and the functoriality in each entry, we see that
\begin{align*}
& (E_M \circ \al_{\theta})(x_{\si}(x_1,h_1,\ldots,x_m,h_m))= \\
& f_{\si}(h_1,\ldots,h_m)(E_M \circ \al_{\theta})((n^{-\frac{s}{2}}\sum_{(l_1,\ldots,l_s)=\dot{0}} s(e_{l_1}\ten h_{k_1})\cdots s(e_{l_s}\ten h_{k_s}) \ten \al_{l_1,\ldots,l_s}(F_{\si}(x_1,\ldots,x_m)))= \\
& f_{\si}(h_1,\ldots,h_m)(E_M \circ \al_{\theta})((n^{-\frac{s}{2}}\sum_{(l_1,\ldots,l_s)=\dot{0}} W(e_{l_1}\ten h_{k_1}\cdots e_{l_s}\ten h_{k_s}) \ten \al_{l_1,\ldots,l_s}(F_{\si}(x_1,\ldots,x_m)))= \\
& f_{\si}(h_1,\ldots,h_m)((n^{-\frac{s}{2}}\sum_{(l_1,\ldots,l_s)=\dot{0}} W(e_{l_1}\ten P_H \al_{\theta}(h_{k_1})\cdots e_{l_s}\ten P_H \al_{\theta}(h_{k_s})) \ten \al_{l_1,\ldots,l_s}(F_{\si}(x_1,\ldots,x_m)))= \\
& (\cos(\theta))^s f_{\si}(h_1,\ldots,h_m)((n^{-\frac{s}{2}}\sum_{(l_1,\ldots,l_s)=\dot{0}} s(e_{l_1}\ten h_{k_1})\cdots s(e_{l_s}\ten h_{k_s}) \ten \al_{l_1,\ldots,l_s}(F_{\si}(x_1,\ldots,x_m)))= \\
& (\cos(\theta))^s x_{\si}(x_1,h_1,\ldots,x_m,h_m),
\end{align*}
which establishes (2). (3) is straightforward using (2). To obtain (4), we calculate \[\lim_{t \to 0} \frac{1}{t}(T_t(x_{\si})-x_{\si})=\lim_{t \to 0} \frac{e^{-st}-1}{t}x_{\si}=-sx_{\si},\] 
for any Wick word $x_{\si}$ of degree $s$. The rest of the statements are straightforward.
\end{proof}
\begin{rem} Due to (4), we have that for every $x \in M$, the function
\[[0,\frac{\pi}{2}) \ni \theta \mapsto \|\al_{\theta}(x)-x\|_2 \]
is increasing.
\end{rem}
\begin{defi} We denote by $D_k(S) \subset L^2(D)$ the $\|\cdot\|_2$-closed linear span of the expressions
 \[ F_{\si}(x_1,\ldots,x_m)\lel  E_{1,...,k}
 (\pi_{\phi(1)}(x_1)\cdots \pi_{\phi(m)}(x_m)), \]
for all $m \geq 1$, $x_1,\ldots,x_m \in BSB$, $\si \in P_{1,2}(m)$ having $k$ singletons and $\phi$ which encodes $\si$.
 \end{defi}
\begin{lemma}\label{apriori} Let $y(j_1,\ldots,j_k) \in L^p(D)$ be such that $\sup_{j_1,..,j_k} \|y(j_1,...,j_k)\|_p<\infty$ and $h_1,\ldots,h_k \in H$. Then
 \[ \sup_n \|n^{-k/2} \sum_{(l_1,\cdots ,l_k)=\dot{0}} s_{l_1}(h_1)\cdots s_{l_k}(h_k) \ten y(l_1,...,l_k)\|_p <\infty \pl. \]
\end{lemma}

\begin{proof} It suffices to consider
 \[ \|\sum_{l_1\in C_1,...,l_k\in C_k} s_{l_1}(h_1)\cdots s_{l_k}(h_1)\ten y(l_1,...,l_k) \| \]
with $C_1\cup \cdots\cup C_k=\{1,...,n\}$. Using the martingale decomposition from Lemma \ref{elim} we deduce
 \begin{align*}
 \|\sum_{l_1\in C_1,...,l_k\in C_k} s_{l_1}(h_1)\cdots s_{l_k}(h_k) \ten y(l_1,...,l_k) \|_p
 &\le c(p) \sqrt{n} \sup_{l\in C_k} \|\sum_{l_1,...,l_{k-1}} s_{l_1}(h_1) \cdots s_{l_k}(h_k) \ten y(l_1,...,l_k) \|_p \pl .
  \end{align*}
Iterating this  procedure we get
 \[ \|\sum_{l_1\in C_1,...,l_k\in C_k} s_{l_1}(h_1) \cdots s_{l_k}(h_k)\ten y(l_1,...,l_k)\|_p
  \kl c(p)^k n^{k/2} \sup_{l_1,...,l_k}
  \|s_{l_1}(h_1)\cdots s_{l_k}(h_k)\|_p \|y(l_1,...,l_k)\|_p \pl .\]
Since the products $s_{l_1}(h_1)\cdots s_{l_k}(h_k)$ are uniformly bounded in the $p$-norm, we obtain the assertion.
\qd
\begin{prop}\label{fingen} Let $(\pi_j,B,A,D)$ be a sequence of independent symmetric copies, $H$ a finite dimensional Hilbert space and $1 \in S=S^* \subset A$, and assume that $D_s(S)$ is finitely generated as a right $B$-module. Then $L^2_s(M)$ is finitely generated as a right $B$-module. In particular, when $D_s(S)$ is finitely generated over $B$ for every $s$, the maps $T_t$ are compact over $B$, for every $t>0$.
\end{prop}

\begin{proof} Let $N$ be the dimension of $D_s$ as a right $B$-module, and let $\lbrace \xi_1, \ldots,\xi_N \rbrace$ be a basis of $D_s$ over $B$. Then, for every $\si \in P_{1,2}(m)$ having $s$ singletons, and every $x_1,\ldots,x_m \in BSB$, we can find coefficients $b_k(\si,x_1,\ldots,x_m)\in B$ such that
 \[  F_{\si}(x_1,\ldots,x_m) \lel \sum_{k=1}^{N} \xi_k b_k(\si,x_1,\ldots,x_m). \]
For every $(l_1,\ldots,l_s)=\dot{0}$ we have
 \[\al_{l_1,\ldots,l_s}(F_{\si}(x_1,\ldots,x_m)) \lel \sum_{k=1}^N
 \al_{l_1,\ldots,l_s}(\xi_k)b_k(\si,x_1,\ldots,x_m) \pl .\]
Fix a finite basis $\mathcal B$ of $H$. Then, for every $\si$ having $s$ singletons, every $x_1,\ldots,x_m \in BSB$ and every $h_1,\ldots,h_m \in \mathcal B$ we have, due to the reduction formula
\begin{align*}
 & x_{\si}(x_1,h_1,\ldots,x_m,h_m) = \\
 & \lel (n^{-s/2} \sum_{(l_1,...,l_s) =\dot{0}} s_{l_1}(h_{i_1})\cdots s_{l_s}(h_{i_s}) \ten \al_{l_1,\ldots,l_s}(F_{\si}(x_1,\ldots,x_m)))= \\
& \sum_{k=1}^N (n^{-s/2} \sum_{(l_1,...,l_s) =\dot{0}} s_{l_1}(h_{i_1})\cdots s_{l_s}(h_{i_s}) \ten \al_{l_1,\ldots,l_s}(\xi_k))b_k(\si,x_1,\ldots,x_m).
 \end{align*}
Thus $L^2_s(M)$ is spanned over $B$ by at most $N|\mathcal B|^s=N(dim(H))^s$ elements, namely
\[(n^{-s/2} \sum_{(l_1,...,l_s) =\dot{0}} s_{l_1}(h_{i_1})\cdots s_{l_s}(h_{i_s}) \ten \al_{l_1,\ldots,l_s}(\xi_k)),\]
with $h_i \in \mathcal B$ and $1 \leq k \leq N$. These elements belong to $L^2_s(M)$ by the previous lemma, and this finishes the proof.
\qd
\begin{rem} Since the dimension of $D_s(S)$ over $B$ is finite, the basis elements $\xi_k \in D_s \subset L^2(D)$ could be chosen in fact to be bounded, i.e. $\xi_k \in D$, due to \cite{PaschkeI, PaschkeII}. This implies that $L^2_s(M)$ admits a basis over $B$ consisting of elements in $M$. 
\end{rem}
\begin{cor} Assume moreover that the dimension $N_s$ of $D_s(S)$ over $B$ has sub-exponential growth, i.e. there exist constants $d,C>0$ such that $N_s \leq Cd^s$ for all $s$. Then the dimension of $L^2_s(M)$ over $B$ is less then $C(dim(H)d)^s$ for all $s$, i.e. the dimension of $L^2_s(M)$ over $B$ also has sub-exponential growth. 
\end{cor}
The following argument is essentially due to Sniady (see \cite{SniadyII}) and Krolak (see \cite{Kro2}). \\
\begin{prop}Let $M=\Gamma_q(B,S\ten H)$. There exists $d=d(q)$ such that for dim$(H)\geq d$ we have $\mathcal Z(M) \subset \mathcal Z(B)$. In particular, $M$ is a factor whenever $B$ is.
\end{prop}
\begin{proof} Let $\{e_i\}_{1\leq i \leq k}$ be an orthonormal set in $H$. We consider the operator $T:L^2(M) \to L^2(M)$ given by
\[T=\sum_{i=1}^k (L_{s(1,e_i)}-R_{s(1,e_i)})^2.\]
Here $L_x$ and $R_x$, where $x \in M$ are the canonical left and right multiplication operators, respectively, on $L^2(M)$. We see that
\[T-2k{\rm id}=\sum_{i=1}^k (L_{s(1,e_i)^2-1}-R_{s(1,e_i)^2-1})-2\sum_{i=1}^k L_{s(1,e_i)}R_{s(1,e_i)}.\]
Since $s(1,e_i)^2-1$ is a mean zero element, we deduce from \cite{NouI} that
\[\|\sum_{i=1}^k s(1,e_i)^2 -1 \|_{\infty} \leq c_q \sqrt{k}.\]
Let us denote by $V=\sum_{i=1}^k L_{s(1,e_i)}R_{s(1,e_i)}$ and $\iota:L^2(M) \to (\F_q(\ell^2 \ten H) \ten L^2(D))^{\om}$ the natural embedding given by the definition. Then we see that
\[\iota(V\xi)=(V_n\iota(\xi)_n)_n, \xi \in L^2(M),\]
where
\[V_n=\frac{1}{n}\sum_{1\leq i \leq k, 1 \leq j,j' \leq n} L_{s(1,e_i\ten e_j)}R_{s(1,e_i\ten e_{j'})}.\]
Now we can easily modify the argument from \cite{Kro2} to show that
\begin{enumerate}
\item $\|\sum_{k,j,j'} l^+(e_i \ten e_j)R_{s(1,e_i\ten e_{j'})}\| \leq c_q\sqrt{kn^2}$;
\item $\|\sum_{k,j,j'} r^+(e_i \ten e_j)L_{s(1,e_i\ten e_{j'})}\| \leq c_q\sqrt{kn^2}$;
\item $\|\sum_{k,j\neq j'} l^-(e_i \ten e_j)R_{s(1,e_i\ten e_{j'})}\| \leq c_q\sqrt{kn^2}$;
\item $\|\sum_{k,j} l^-(e_i \ten e_j)r^+(e_i \ten e_j)|_{\cz^{\perp}}\| \leq q+c_q\sqrt{kn}$.
\end{enumerate}
Here $l^+,l^-,r^+,r^-$ are the left and right creation operators on the $q$-Fock space coming from the decomposition $L_{s(h)}=l^+(h)+l^-(h)$, $R_{s(h)}=r^+(h)+r^-(h)$. The main estimate is derived from
\[l^-(h)r^+(k)(\xi)=q^{|\xi|}\xi+l^-(h)(\xi)\ten k.\]
The second part can then be estimated via (2). This yields
\[\|(T-2k{\rm id})(({\rm id}-E_B)(\xi))\|\leq 2qk\|({\rm id}-E_B)(\xi)\|+2c_q\sqrt{k}\|({\rm id}-E_B)(\xi)\|.\]
Now take $z \in \mathcal Z(M)$ with $E_B(z)=0$. Thus $T(z)=0$ and also
\[0=\|T(z)\|=\|2kz-(T(z)-2kz)\|\geq 2k\|z\|-2qk\|z\|-C_q\sqrt{k}\|z\|=(2k(1-q)-C_q\sqrt{k})\|z\|.\]
Thus for $2k(1-q)-C_q\sqrt{k}>0$, i.e. $k>\sqrt{\frac{C_q}{2(1-q)}}$, we have that $z=0$. This implies $z=E_B(z)$, for all $z \in \mathcal Z(M)$, hence $\mathcal Z(M) \subset B$ and also $\mathcal Z(M) \subset \mathcal Z(B)$.
\end{proof}
\subsection{H-less generalized q-gaussians}
Finally, let us mention that there is an $H$-less version of the generalized $q$-gaussians, which can be described as follows: let $(\pi_j,B,A,D)$ a sequence of symmetric independent copies. For $1\in S=S^* \subset A$, define the von Neumann algebra $\Gamma_q(B,S) \subset (\Gamma_q(\ell^2) \bar{\ten} D)^{\om}$ as being generated by the elements $s_q(x)=(n^{-\frac{1}{2}}\sum_{j=1}^n s_q(e_j) \ten \pi_j(x))_n$, for $x \in BSB$. This is equivalent to taking $H$ to be 1-dimensional in the Def. 3.4 above, hence the $H$-less $q$-gaussians are a particular case of Def. 3.4. Surprisingly, the $H$ generalized $q$-gaussians can also be obtained as a particular case of this construction. Indeed, let $H$ be a (real) Hilbert space and $(\pi_j,B,A,D)$ a sequence of symmetric independent copies. Let $(X,\mu)$ be a standard probability measure space and define a new sequence of symmetric independent copies $(\tilde \pi_j,\tilde B,\tilde A,\tilde D)$ by taking $\tilde B=B$, $\tilde A=A\bar{\ten} L^{\infty}(X)$, $\tilde D=D\bar{\ten}(\overline{\bigotimes}_1^{\infty} L^{\infty}(X))$ and $\tilde \pi_j:\tilde A \to \tilde D$ by
\[\tilde{\pi}_j(a \ten f)= \pi_j(a) \ten (1\ten 1\ten \cdots \ten \underbrace{f}_{\mbox{\scriptsize{j-th position}}}\ten \cdots \ten 1 \ten \cdots), a \in A, f \in L^{\infty}(X).\]
Using Rademacher variables, we see that there exists a dense subspace $H_0 \subset H$ and an isometric embedding $\iota:H_0 \to L^{\infty}(X) \subset L^2(X)$. Take $\tilde S = S \ten \iota(H_0)=\{a \ten \iota(h): a \in S, h \in H_0\} \subset \tilde A$. The reader can check that
\[\Gamma_q(B,S\ten H)=\Gamma_q(\tilde B, \tilde S).\]

\section{Examples}

We will discuss several type of examples of generalized $q$-gaussian von Neumann algebras. The underlying idea in all these cases is that whenever we have a finite von Neumann algebra on which the symmetric group acts, we can construct a sequence of symmetric copies. In particular, countable tensor or (amalgamated) free products von Neumann algebras or the pure $q$-gaussian von Neumann algebras $\Gamma_q(H)$, for an infinite dimensional $H$, constitute obvious candidates, since the symmetric group acts naturally on them.

\subsection{Tensor products}

Let $B$ and $C$ be finite von Neumann algebras. Define $A=B \bar{\ten} C$ and $D=B \bar{\ten} C^{\nz} = B \bar{\ten} (\overline{\bigotimes}_{\nz} C)$. Define $\pi_j:A \to D$ by the formula
\[ \pi_j(b\ten a)\lel b\ten 1\ten 1\ten \cdots \ten \underbrace{a}_{\mbox{\scriptsize{j-th position}}}\ten \cdots \ten 1 \ten \cdots .\]
Then it's easy to check that $(\pi_j,B,A,D)$ is a sequence of symmetric independent copies. It's likewise easy to see that
\[ \Gamma_q(B,A\otimes H) \lel B \bar{\ten} \Gamma_q(L_{sa}^2(C)\ten H).\]
For any finite subset $S \subset L_{sa}^2(C)\ten H$, the space $D_k(S)$ has finite dimension over $B$.

\subsection{Free products with amalgamation}

Let $B \subset A$ be an inclusion of finite tracial von Neumann algebras. Take $D=\ast_B A_j$ the amalgamated free product of a countable number of copies $A_j, j \in \nz$ of $A$. Define $\pi_j:A \to D$ by the formula
\[ \pi_j(a)\lel 1\ast 1\ast \cdots \ast \underbrace{a}_{\mbox{\scriptsize{j-th position}}} \ast \cdots \ast 1 \ast \cdots .\]
Then $(\pi_j,B,A,D)$ represents a sequence of independent symmetric copies. To see why this is true it suffices to consider elements $a_i$ such that $E_B(a_i)=0$. Then we have to calculate
 \[ \tau_{\si}(a_1,...,a_m)\lel \tau(\pi_{j_1}(a_1)\cdots \pi_{j_m}(a_m))  \]
such that $(j_1,...,j_m)=\si$. If $\si$ has no crossings, we can inductively replace neighboring pairs by $E_B(\pi_{j_i}(a_i)\pi_{j_i}(a_{i+1}))=E_B(a_ia_{i+1})$ and finally find an element in $B$. For a non-crossing pair partition we can also join all the pairs, but then we find an expression of the form
 \[ \tau(b_1\pi_{j_{i_1}}(a_{i_1})b_2\pi_{j_{i_2}}(a_{i_2})\cdots \pi_{j_{i_k}}(a_{i_k})) \lel 0 \pl .\]
Thus in the moment formula we only have to expand over non-crossing pair partitions. Now take $S=\{1,u,u^*\}$, for $u\in A$ a Haar unitary such that $E_B(u^n)=0$, for all $n\neq 0$. It's easy to see that $D_k(S)$ is the closed linear span of all the expressions $b_1\pi_1(u^{\epsilon_1})b_2\ldots\pi_k(u^{\epsilon_k})b_{k+1}$, with $b_i\in B$ and $\epsilon_i\in\{0,1,\ast\}$. In particular, when $B=\cz$ or is finite dimensional, we have dim$_B(D_k(S))\leq C 2^{2k}$.

\subsection{Group actions}

\subsubsection{Second quantization}
Let $G \curvearrowright_{\al} C$ be a trace preserving action of the discrete group $G$ on the finite von Neumann algebra $C$. Also let $\nu:G \to \mathcal O(H_{\rz})$ be an orthogonal representation of $G$ on a real Hilbert space $H_{\rz}$. Let $(\Om,\mu)$ be the gaussian construction associated to $\nu$ (see e.g. \cite{PetersonSinclair}) . We also denote the corresponding action $G \curvearrowright L^{\infty}(\Om)$ by $\nu$. Then define $B=C \rtimes_{\al} G$, $A=(C \bar{\ten} L^{\infty}(\Om))\rtimes_{\rho} G$, $D=(C \bar{\ten} L^{\infty}(\Om^{\nz}))\rtimes_{\rho} G$ where the action $\rho$ is given by $\rho_g(d\ten f)=\al_g(d) \ten \nu_{g}(f)$. Define the *-homomorphisms $\pi_j:A \to D$ by
 \[ \pi_j((d\ten f)u_g) \lel (d \ten 1\ten 1\ten \cdots \ten \underbrace{f}_{\mbox{\scriptsize{j-th position}}} \ten \cdots \ten 1 \ten \cdots)u_g.\]
Then it is easy to see that the fixpoint algebra is $C \rtimes_{\al} G$. Again the moments only depend on the inner product. Moreover, the gaussian functor yields a map ${\rm Br}:H \to L^2(\Om)$. Then we find
 \[ M\lel \Gamma_q(C\rtimes G,{\rm Br}(H)) \lel (C \bar{\ten} \Gamma_q(H))\rtimes G \pl .\]
The spaces $D_k(S)$ are finite dimensional modules over $B=C\rtimes G$ if $L^2_k(H)\rtimes G$ has a finite basis over $G$. For $k=1$ this means that $H$ is finite dimensional. In a forthcoming paper we will also analyze the case of profinite actions and / or representations, i.e. when $H$ can be written as $H=\overline{\bigcup_i H_i}$ such that every $H_i$ is a finite dimensional $G$-invariant Hilbert subspace. However, discrete subgroups of $\mathcal O_n=\mathcal O(\rz^n)$ provide a large class of non-trivial, non-amenable  examples. The examples in \cite{JLU} are subalgebras of $M$.

\subsubsection{Symmetric group action}

Throughout this subsection $\Si$ will denote the group of finite permutations on $\nz$. Let us consider a countable discrete group $G$ on which $\Si$ acts by automorphisms. Examples for such a symmetric action are given by the natural action of $\Si$ on the free group with countably many generators, or by the natural action of $\Si$ on the direct product groups $\prod_{\nen}G$. More generally, let $R\subset \mathbb{F}_{\infty}$ be a set of generators which is invariant under the action of $\Si$, and assume that $\langle R \rangle \subset \mathbb{F}_{\infty}$ is a normal subgroup. Then $G=\mathbb{F}_{\infty}/\langle R \rangle$ is a group on which $\Si$ acts. A perfect example is given by an amalgamated free product $\ast_HG_j$ where $G_j=G$. To make things more concrete, we may consider the discrete Heisenberg group $\mathcal{H}=\langle \zz,\zz^{\infty} \rangle$ with generators $\{ g_k \}_{k \geq 0}$ such that $\zz = \langle g_0 \rangle$, $\zz^{\infty} = \langle g_k, k \geq 1 \rangle$ and the following relations hold
  \[ g_k^{-1}g_jg_k=g_0g_j, \quad k\neq j  \pl .\]
Then $\Si$ acts on $\mathcal H$ by permuting the generators $g_k$ for $k\gl 1$, and leaving $g_0$ fixed.
Now we assume that such a  $G$, with action $\Si \curvearrowright_{\beta} G$,  acts trace-preservingly on a finite von Neumann algebra $A$ and $B$ is the fixed points algebra of this action $\al$. Let $g \in G$ be an arbitrary element and $g_j=\beta_{(1j)}(g)$. We can then construct a sequence of symmetric copies $(\pi_j,B,A,D)$ by defining $\pi_j:A\to A$, via $\pi_j(x)=\al_{g_j}(x)$. Working in the crossed product $(A \rtimes_{\al} G)\rtimes_{\beta}\Si$ it is easy to see that the $\pi_j$'s are symmetric copies, and that $B$  is the fixed points algebra for these symmetric copies. In fact we  may and will always assume that $G$ is generated by the $g_j$'s and then $\pi_j(x)=x$ for all $j$ is exactly the fixed points algebra of the action.  In general $\pi_j(A)=A$ and hence we find an example of symmetric, but not necessarily independent copies. In general independent copies are obtained from considering a suitable subalgebra $B \subset A_1\subset A$. More generally for a subset $S\subset A$ we may however consider the algebras
  \[ A_j(S) \lel \{\pi_j(x)|x\in S, j\in A\}  \pl. \]
This is particularly interesting for a single selfadjoint $x$. Then independence depends on the mixing properties of the sequence $\pi_j(x)$, and has to be analyzed on a case by case basis. A more specific example can be constructed starting from a trace preserving action $\al$ of $\zz$ on a finite von Neumann algebra $N$. Take $D=N \rtimes_{\beta} \mathcal{H}$ where the action $\beta$ is obtained by lifting the action of $\zz$ via the group homomorphism $\pi:\mathcal H \to \zz$ given by $\pi(g_0)=0$ and $\pi(g_j)=1$ for $j \geq 1$. In other words,
\[\beta_g(x)=\al_{\pi(g)}(x), g \in \mathcal H, x \in N.\]
Let $\mathcal{H}_1$ be the group generated by $g_0$ and $g_1$ and take $B=N \rtimes \zz=N \bar{\ten} L(\zz)$ and $A=N\rtimes \mathcal{H}_1$. Define $\pi_j:A \to D$ by 
\[\pi_j(x u_{g_1})=\al_{\pi(g_j)}(x) u_{g_j}, \pi_j(x u_{g_0})=x u_{g_0}, x \in N, j,k \in \nz.\]
Then $(\pi_j,B,A,D)$ is a sequence of symmetric independent copies. In full generality the dimensions of the spaces $D_k(S)$ or $L^2_k(M)$, where $M=\Gamma_q(B,A \ten H)$, cannot be controlled. If we restrict ourselves to a small set of generators, e.g. $S=\{1,g_1,g_1^{-1}\}$, then we get a more well-behaved example. The space $D_k(S)$ is the closed linear span of the expressions of the form
 \[ \pi_{j_1}(u_{g_1})\cdots \pi_{j_k}(u_{g_k}) u_{g_0}^{l(\si)} \al_{n(\si)}(x).\]
Thus $\dim_B(D_k(S))\le (2\dim(H))^{2k}$. For more general group actions and $S \subset L(G)$, we find coefficients in $B=L([G,G]) \bar{\ten} N$ and finite dimension over $B$ as long as we have finite generating sets. Note however, that $L([G,G])$ is in general not invariant under the action of $\Si$, and hence a more detailed case by case analysis is required. Again a particularly nice class of examples comes from one step nilpotent groups with commutators in the center, such as the Heisenberg groups.

\subsection{Colored Brownian motion}

\subsubsection{Top up $q$-gaussians}

Let $H$ be a Hilbert space and $q_0\in [-1,1]$. Symmetric independent copies can be obtained from second quantization, or simply by defining $\pi(s_{q_0}(h))=s_{q_0}(e_j\ten h)$.  This provides symmetric copies of $A=\Gamma_{q_0}(H)$ into $D=\Gamma_q(\ell_2(H))$. By looking at Wick words it is easy to see that the fixpoint algebra is $\cz$. Moreover, independence follows from the moment formula for $q_0$-gaussian random variables. Let $S=\{x_1,\ldots,x_p\}$ be a finite, selfadjoint  subset, where $x_i=s_{q_0}(h_1(i))\cdots s_{q_0}(h_{l(i)}(i))$, for $1\leq i \leq p$. Then we see that
 \begin{align*}
 \tau(s_q(k_1,x_1)\cdots s_q(k_m,x_m))
 &= \sum_{\si \in P_2(m)} q^{{\rm cr}(\si)} f_{\si}(k_1,...,k_m)
 \tau(\pi_{j_1^\si}(x_1)\cdots \pi_{j_m^\si}(x_m)) \pl,
 \end{align*}
where for every $\si$, we choose an $m$-tuple $(j_1^\si,..,j_m^\si)$ depending on $\si$ such that $(j_1^\si,\ldots,j_m^\si)=\si$. Now we may use the formula for $q_0$-gaussians and find  for $L=\sum_{i=1}^p l(i)$ that
 \[ \tau(\pi_{j_1^\si} (x_1)\cdots \pi_{j_m^\si} (x_m))
 \lel \sum_{\si'\in P_2(L),\si'\leq\phi(\si) }q_0^{{\rm cr}(\si')}
  f_{\si'}(h_1(1),....,h_{l(1)}(1),....,h_{1}(m),...,h_{l(m)}(m)) \]
Here $\phi(\si)$ is the block partition which gives the same color to the union of two blocks in $\si$ connected via pairs in $\si'$. This means
 \[  \tau(s_q(k_1,x_1)\cdots s_q(k_m,x_m))\lel
  \sum_{\si \in P_2(m),\si'\le \phi(\si)}
  q^{{\rm cr}(\si')}q_0^{{\rm cr}(\si')}
  f_\si (k_1,...,k_m)f_{\si'} (h_1,...,h_{L}) \]
where $\si'$ runs over the partitions of $\{1,\ldots,L\}$ and $\{h_1,\ldots,h_L\}$ is a re-labeling of $\{h_j(i)\mid 1\leq i\leq p, 1\leq j\leq l(i)\}$. Note that $\Gamma_{q}(\cz,\Gamma_{q_0}(H) \otimes K)$ contains both $\Gamma_{q}(K)$ and $\Gamma_{q_0}(H)$ if $s_{q_0}(H)\subset S$.
Using a decomposition into minimal links, we deduce that the space $D_k(S)$ is the closed linear span of the elements
  \[ c(\si,x_1,...,x_r) \pi_{j_{i_1}}(x_{i_1})\cdots \pi_{j_{i_r}}(x_{i_r}) \pl, \]
where $c(\si,x_1,...,x_r)$ is a scalar. This means for a finite set $S$ of generators, the dimension of $D_k(S)$ over $B=\cz$ is less than $(|S|\dim(K))^{2k}$. One could call these algebras ``mixed'' gaussian algebras, but the reader should not mistake them for the mixed $Q$-gaussian algebras (introduced in \cite{JungeZeng}) which we use in 6.3.

\subsubsection{Actions of $\Si$ by conjugation}
Let us consider the finite permutations group $\Si_{\zz}$ acting on $\zz$ instead of $\nz \setminus \{0\}$. For every subset $F \subset \zz$ we can identify $\Si_F$, the permutations group on $F$, with a subgroup of $\Si_{\zz}$ by viewing the elements of $\Si_F$ as acting non-trivially only on $F$ and acting as the identity on $\zz \setminus F$. For convenience, we use interval notation for the subsets of $\zz$. In particular we have $\Si =\Si_{[1,\infty)} \subset \Si_{\zz}$ in this way. Let $\Si$ act on $\Si_{\zz}$ by conjugation. This gives rise to an action $\al$ of $\Si$ on the von Neumann algebra $L(\Si_{\zz})$ (which is in fact isomorphic to the hyperfinite factor). We denote the canonical unitaries generating $L(\Si_{\zz})$ by $u_{\si}, \si \in \Si_{\zz}$. The fixed points algebra of this action is $B=L(\Si_{(-\infty,0]})$. Take $A=L(\Si_{(-\infty,1]})=B \vee \{u_{(01)}\}''$, $D=L(\Si_{\zz})$ and define $\pi_j:A \to D$ by $\pi_j(a)=\al_{(j1)}(a)$ for $a \in A$ and $j \geq 2$, where $(j1)$ is the transposition interchanging $j$ and $1$, and $\pi_1=id$. Then $(\pi_j,B,A,D)$ is a sequence of symmetric independent copies. Indeed, we recall that $A$ is generated by transpositions $(k1)$, $k\le 0$ and that for $j \geq 2$ we have
 \[ (j1)(k1)(j1) \lel (kj) \pl .\]
This means $A_j=B \vee \{u_{(0j)}\}''$ and $A_{1,\ldots,j}=L(\Si_{(-\infty,j]})$. In particular, we have a coset representation $\si=\si'(j1)$ with $\si'\in \Si_{(-\infty,0]}$. The algebras $A_I$ are generated by $\Si_{I}$, $\Si_{(-\infty,0]}$ and one generator $(j1)$ for $j\in I$. This easily implies independence. We take $S =\{1, u_{(01)} \} \subset A$ and define $M=\Gamma_q(B,S\otimes H)$. Fix $\si \in P_{12}(m)$ having $k$ singletons and $p$ pairs, take $\phi:\{1,\ldots,m\}\to \{1,\ldots,k+p\}$ which encodes $\si$. This means $\phi(j_t)=t$, where $\{j_t\}, 1\leq t \leq k$ are the singletons of $\si$, and $\phi(j_t')=\phi(j_t'')=k+t$, where $\{j_t',j_t''\}, 1\leq t \leq p$ are the pairs of $\si$. Then $D_k(S)$ is the closed span of elements of the form
 \begin{align*}
 & E_{1,\ldots,k}(u_{(\phi(1)0)}u_{\gamma_1}u_{(\phi(2)0)}u_{\gamma_2}\cdots u_{\gamma_{m}}u_{(\phi(m)0)}u_{\gamma_{m+1}}) \\
 & =E_{1,\ldots,k}(u_{(\phi(1)0)}{\rm ad}(u_{\gamma_1})(u_{(\phi(2)0)})...
 {\rm ad}(u_{\gamma_1...\gamma_{m}}) (u_{(\phi(m)0)})u_{\gamma_1 \cdots \gamma_{m+1}})\\
 &= E_{1,\ldots,k}(u_{(\phi(1)\gamma_1(0))} u_{(\phi(2)\gamma_1\gamma_2(0))} \cdots u_{(\phi(m)(\gamma_1 \cdots \gamma_{m})(0))}u_{\gamma_1 \cdots \gamma_{m+1}}) \\
 & = E_{1,\ldots,k}(u_{(\phi(1)s_1)} u_{(\phi(2)s_2)} \cdots u_{(\phi(m)s_m)}) u_{\gamma_1 \cdots \gamma_{m+1}},
\end{align*}
where $\gamma_1,\ldots, \gamma_{m+1} \in \Si_{(-\infty,0]}$ are arbitrary. Here $s_1=\gamma_1(0), s_2=\gamma_1\gamma_2(0), \ldots, s_m=\gamma_1\gamma_2 \cdots \gamma_m(0)$ in $(-\infty,0]$ depend only the $\gamma_i$'s. In full generality the modules $D_k(S)$ do not have finite dimensions over $B$. If we however replace $B$ by $B_d=L(\Si_{[-d,0]}) \cong L(\mathbb{S}_{d+1})$, $A$ by $A_d=L(\Si_{[-d,1]})=B_d \vee \{u_{(01)}\}''$ and $D$ by $D_d=L(\Si_{[-d, \infty)})$ for a fixed $d \in \nz \setminus \{0\}$, then we obtain a new sequence of symmetric independent copies 
$(\pi_j,B_d,A_d,D_d)$ and in this case we have at most $(d+1)^k$ different choices for the $s_j$'s. After repeated conjugation with the unitaries on the pair positions, the above expression becomes
\[u_{(s_{j_1}'1)} \cdots u_{(s_{j_k}'k)}E_{1,\ldots,k}(u_{(s_{j_{k+1}}'k+1)}\cdots u_{(s_{j_{k+p}}'k+p)}),\] 
for some new indices $s_i' \in (-\infty,0] \cap \zz$ which in general depend on the $\gamma_i$'s and $\si$. Since for an inclusions of groups $H \subset G$ and $g \in G$ we have $E_{L(H)}(u_g)=\delta_{g \in H}u_g$, and the product $(s_{j_{k+1}}'k+1)\cdots (s_{j_{k+p}}'k+p)$ belongs to $\Si_{(-d,k]}$ only if it's equal to $1$, we see that a spanning set of $D_k(S)$ over $B_d$ is given by the elements
\[u_{(s_{j_1}'1)}u_{(s_{j_2}'2)} \cdots u_{(s_{j_k}'k)},\]
for all choices of $-d \leq s_i' \leq 0, 1 \leq i \leq k$, which in particular implies that the dimension of $D_k(S)$ over $B_d$ is at most $(d+1)^{2k}$. Note that $B_d$ and $A_d$ are finite dimensional von Neumann algebras. Thus, for the von Neumann algebras $M(d)=\Gamma_q(B_d,S \otimes H)$, the spaces $D_k(S)$ have sub-exponential growth of their dimensions over $B_d$. This remains true for any finite subset $1 \in S=S^* \subset A_d$.

\subsection{Operator-valued gaussians}

This example is motivated by Shlyahktenko's A-valued semicircular algebras and derived from the tensor product construction. Let $x_k\in N$ be selfadjoint operators and $X=\sum_k g_k x_k$. We consider $A_1=L^{\infty}(\rz)$ and  the independent symmetric copies over $N$ given by
 \[ \pi_j(f) \lel f(\sum_{k} g_{k,j}x_k) \pl ,\]
where  $g_{k,j}$ are i.i.d. gaussians (we could also work with $q$-gaussians). The copies are independent over $N$. Let $D$ be the von Neumann algebra generated by the $\pi_j(f)$'s and $B$ be the tail algebra
 \[  B\lel \bigcap_{m\geq 0} \bigvee_{j\geq m} \pi_j(L^{\infty}(\rz)).\]
One can show that the copies $\pi_j$ are independent symmetric in the sense of our definition 3.2. Note that $N$ is invariant under the shift from tensor product construction and hence $B\subset N$. Thus $M=\Gamma_q(B,S \otimes H)$ is a legitimate example where $S=\sum_k g_kx_k$  is obtained by approximating $X$ with bounded functions. Since $X\in \bigcap_{1\leq p<\infty} L^p(\rz)$ one can actually directly work with one generator $x$. The dimension of the $L_{k}^2(M)$ over $B$ is in general hard to determine. The case of $N=M_m(\cz)$ and $X \lel \sum_{r,s} g_{rs} (\frac{e_{rs}+e_{sr}}{2})$ has been considered by Avsec and Speicher.

\begin{rem} The examples in 41., 4.2, 4.4.1, 4.4.2 for $d=0$ and 4.4.3 are all factors if $B$ is a factor and dim$(H)\geq d(q)$. 
\end{rem}
\section{Weak amenability produces approximately invariant states}
Let $(\pi,B,A,D)$ a sequence of symmetric independent copies, $1 \in S=S^* \subset A$ and assume that $D_s(S)$ is finitely generated over $B$ for all $s \geq 1$. Let $M=\Gamma_q(B,S \ten H)$ for a finite dimensional space $H$, $\mathcal A \subset M$ be a von Neumann subalgebra which is amenable relative to $B$ inside $M$, and let $P=\mathcal{N}_M(\mathcal A)''$. Define $\mathcal{M}=(\Gamma_q(\ell^2 \otimes H)\bar{\otimes}D)\vee M \subset (\Gamma_q(\ell^2 \otimes H)\bar{\otimes}D)^{\omega}$, where $ \Gamma_q(\ell^2 \otimes H)\bar{\otimes}D$ is embedded as constant sequences. Let
\[\mathcal H \subset ((L^2(\mathcal M) \otimes_{\mathcal A} L^2(P))\ten \mathcal F_q(\ell^2 \otimes H))^{\omega}\]
 be the $\|\cdot\|$-closed span of the sequences
 \[(n^{-\frac{m}{2}}\sum_{(j_1,\ldots,j_m)=\si} (\pi_{j_1}(x_1)\cdots \pi_{j_m}(x_m)y \ten_{\mathcal A} z)\ten s(e_{j_1}\ten h_1)\cdots s(e_{j_m}\ten h_m)),\]
for all $m \geq 1$, $\si \in P_{1,2}(m)$, $x_i \in BSB$, $y \in M, z \in P$ and $h_1,\ldots,h_m \in H$.
Define two *-representations $\pi:M \to B(\mathcal H)$, $\theta:P^{op} \to B(\mathcal H)$ by
\begin{align*}
& \pi(x_{\si'}) (n^{-\frac{m}{2}}\sum_{(j_1,\ldots,j_m)=\si} (\pi_{j_1}(x_1)\cdots \pi_{j_m}(x_m)y \ten_{\mathcal A} z)\ten s(e_{j_1}\ten h_1)\cdots s(e_{j_m}\ten h_m))= \\
& (n^{-\frac{m+m'}{2}}\sum_{(i_k)=\si',(j_l)=\si} (\pi_{i_1}(y_1)\cdots \pi_{j_m}(x_m)y \ten_{\mathcal A} z) \ten s(e_{i_1}\ten k_1) \cdots s(e_{j_m}\ten h_m))
\end{align*}
and
\begin{align*}
& \theta(w^{op})(n^{-\frac{m}{2}}\sum_{(j_1,\ldots,j_m)=\si} (\pi_{j_1}(x_1)\cdots \pi_{j_m}(x_m)y \ten_{\mathcal A} z) \ten s(e_{j_1}\ten h_1)\cdots s(e_{j_m}\ten h_m))=\\
& (n^{-\frac{m}{2}}\sum_{(j_1,\ldots,j_m)=\si} (\pi_{j_1}(x_1)\cdots \pi_{j_m}(x_m)y \ten_{\mathcal A} zw) \ten s(e_{j_1}\ten h_1)\cdots s(e_{j_m}\ten h_m))
\end{align*}
where $x_{\si'}=(n^{-\frac{m'}{2}}\sum_{(i_1,\ldots,i_{m'})=\si'}\pi_{i_1}(y_1)\cdots \pi_{i_{m'}}(y_{m'})\ten s(e_{i_1} \ten k_1)\cdots s(e_{i_{m'}}\ten k_{m'}))\in M$ is a Wick word in $M$ and $w \in P$.
 Define $\mathcal N = \pi(M) \vee \theta(P^{op}) \subset B(\mathcal H)$. Note that $\pi(M)$ and $\theta(P^{op})$ commute.

\begin{theorem}There exists a sequence of normal states $\omega_n \in \mathcal{N}_*$ satisfying the following properties
\begin{enumerate}
\item $\omega_n(\pi(x)) \to \tau(x), x \in M$.
\item $\omega_n(\pi(a)\theta(\bar{a})) \to 1, a \in \mathcal{U}(\mathcal A)$.
\item $\|\omega_n \circ Ad(\pi(u)\theta(\bar{u})) - \omega_n\| \to 0, u \in \mathcal{N}_M(\mathcal A)$.
\end{enumerate}
\end{theorem}
\begin{proof} Throughout the proof $m_n$ will be the completely contractive finite rank multipliers on $\Gamma_q(\ell^2 \otimes H)$ given by multiplication with a positive finitely supported function $f_n$ constructed by Avsec in \cite{Avsec} and $\varphi_n:=(m_n \otimes id):M \to M$ the corresponding cb map on $M$. Take
\[\mathcal K \subset (L^2(\mathcal M) \otimes_{D} L^2(\mathcal M))^{\omega}\]
to be the $\|\cdot\|$-closed span of the sequences
\[(n^{-\frac{m}{2}}\sum_{(j_1,\ldots,j_m)=\si} (\pi_{j_1}(x_1)\cdots \pi_{j_m}(x_m) \ten s(e_{j_1}\ten h_1) \cdots s(e_{j_m}\ten h_m)) \ten_D y))=(x_{\si}^n \ten_D y),\]
where $x_i \in BSB$ and $y \in M$.
Note that $\mathcal K$ is naturally an $M-M$ bimodule with the actions
\begin{align*}
& x_{\si'}\cdot (n^{-\frac{m}{2}}\sum_{(j_1,\ldots,j_m)=\si} (\pi_{j_1}(x_1)\cdots \pi_{j_m}(x_m) \ten s(e_{j_1}\ten h_1) \cdots s(e_{j_m}\ten h_m)) \ten_D y)) \cdot z = \\
& (n^{-\frac{m+m'}{2}}\sum_{(i_k)=\si',(j_l)=\si} (\pi_{i_1}(y_1)\cdots \pi_{j_m}(x_m) \ten s(e_{i_1}\ten k_1) \cdots s(e_{j_m}\ten h_m))\ten_D yz),
\end{align*}
where $x_{\si'}=x_{\si'}(y_1,k_1,\ldots,y_{m'},k_{m'})\in M$ and $z \in M$.
Denote by $\mathcal S_{\mathcal A} = \lambda(M)\vee \rho(\mathcal A^{op}) \subset B(\mathcal K)$, where $\lambda $ and $\rho$ are the representations of $M$ and $M^{op}$ canonically associated to the left and right actions on $\mathcal K$, respectively.
\par {\bf Step 1.} There exists a normal, unital, completely positive map $\mathcal E:\mathcal N \to \mathcal S_{\mathcal A}$ such that
\[\mathcal E (\pi(x)\theta(y^{op})) = \lambda(x)\rho(E_{\mathcal A}(y)^{op}), x \in M, y \in P.\]
Indeed, define an isometry $V: \mathcal K \to \mathcal H$ by
\begin{align*}
&(n^{-\frac{m}{2}}\sum_{(j_1,\ldots,j_m)=\si} (\pi_{j_1}(x_1)\cdots \pi_{j_m}(x_m) \ten s(e_{j_1}\ten h_1)\cdots s(e_{j_m}\ten h_m))\ten_D y) \mapsto \\
 & (n^{-\frac{m}{2}}\sum_{(j_1,\ldots,j_m)=\si} (\pi_{j_1}(x_1)\cdots \pi_{j_m}(x_m) y\ten_{\mathcal A} 1) \ten s(e_{j_1}\ten h_1)\cdots s(e_{j_m}\ten h_m)).
\end{align*}
Then $\mathcal E$ can be defined by $\mathcal E(z)=V^*zV, z \in \mathcal N$.
\par {\bf Step 2.} There exist normal functionals $\mu_n^{\mathcal A} : \mathcal S_{\mathcal A} \to \mathbb{C}$ such that
\[\mu_n^{\mathcal A}(\lambda(x)\rho(a^{op}))=\tau(\varphi_n(x)a), x \in M, a \in \mathcal A.\]
We need two lemmas. Recall the formulas for Wick words and reduced Wick words introduced in Thm. 3.11.
\begin{lemma}\label{12} $L^2(M) \otimes_B L^2(M)$ embeds as an $M-M$ bimodule into $\mathcal K$.
\end{lemma}
\begin{proof} The map
\begin{align*}
& L^2(M) \ten_B L^2(M) \ni (n^{-\frac{m}{2}}\sum_{(j_1,\ldots,j_m)=\si} \pi_{j_1}(x_1)\cdots \pi_{j_m}(x_m)\ten s(e_{j_1}\ten h_1) \cdots s(e_{j_m}\ten h_m))\ten_B y \mapsto \\
& (n^{-\frac{m}{2}}\sum_{(j_1,\ldots,j_m)=\si} (\pi_{j_1}(x_1)\cdots \pi_{j_m}(x_m)\ten s(e_{j_1}\ten h_1) \cdots s(e_{j_m}\ten h_m))\ten_D y) \in \mathcal K,
\end{align*}
or in other words $(x_{\si}^n)\ten_B y \mapsto (x_{\si}^n \ten_D y)$, is an $M-M$ bimodular isometry. The bimodularity is obvious, so it remains to check that it preserves inner products, in other words that
\[\lan (x_n)\ten_B y,(x_n')\ten_B y'\ran=\lan (x_n \ten_D y),(x_n' \ten_D y')\ran.\]
Let's denote by $E_D:\mathcal M \to D$  and by $E_{D \ten 1}: \Gamma_q(\ell^2 \ten H)\bar{\ten}D\to D \ten 1$ the canonical conditional expectations. Since $D= D \ten 1 \subset \mathcal M \subset (\Gamma_q(\ell^2 \ten H) \bar{\ten} D)^{\om}$ is embedded as constant sequences, for every $(x_n) \in \mathcal M$ we have
\[E_D((x_n))=w-\lim_{n \to \om} E_{D \ten 1}(x_n).\]
We now claim that for any $(x_n)\in M \subset \mathcal M$ we have $E_B((x_n))=E_D((x_n))$. It suffices to prove this for $(x_n)=W_{\si} \in M$ a reduced Wick word. Let $s$ be the number of singletons in $\si$. Let
\[W_{\si}=(n^{-\frac{s}{2}}\sum_{(l_1,\ldots,l_s)=\dot{0}} \al_{l_1,\ldots,l_s}(F_{\si}(x_1,\ldots,x_m))\ten  s(e_{l_1}\ten h_{k_1})\cdots s(e_{l_s}\ten h_{k_s})).\]
We have two possibilities. If $s=0$, then $W_{\si}=F_{\si}(x_1,\ldots,x_m)=E_B(\pi_{\phi(1)}(x_1)\cdots \pi_{\phi(m)}(x_m))\in B$, hence $E_D(W_{\si})=W_{\si}=E_B(W_{\si})$. If $s >0$, then $E_B(W_{\si})=0$. On the other hand, according to our previous remark, we have
\begin{align*}
& E_D(W_{\si})=w-\lim_n E_{D \ten 1}(n^{-\frac{s}{2}}\sum_{(l_1,\ldots,l_s)=\dot{0}} \al_{l_1,\ldots,l_s}(F_{\si}(x_1,\ldots,x_m)) \ten s(e_{l_1}\ten h_{k_1})\cdots s(e_{l_s}\ten h_{k_s})) \\
&= w-\lim_n n^{-\frac{s}{2}}\sum_{(l_1,\ldots,l_s)=\dot{0}} \tau(s(e_{l_1}\ten h_{k_1}) \cdots s(e_{l_s}\ten h_{k_s})) \al_{l_1,\ldots,l_s}(F_{\si}(x_1,\ldots,x_m))\\
&= w-\lim_n n^{-\frac{s}{2}}\sum_{(l_1,\ldots,l_s)=\dot{0}} \tau(W(e_{l_1}\ten h_{k_1} \cdots e_{l_s}\ten h_{k_s})) \al_{l_1,\ldots,l_s}(F_{\si}(x_1,\ldots,x_m))=0.
\end{align*}
This proves our claim. Now, for $(x_n),(x'_n),y,y' \in M$ we have
\begin{align*}
& \lan (x_n) \ten_B y,(x'_n)\ten_B y' \ran = \tau_M(E_B((x_n'^*x_n))yy'^*)= \tau_{\mathcal M}(E_D((x_n'^*x_n))yy'^*)=\\
& \lim_n \tau_{\mathcal M}(E_{1 \ten D}(x_n'^*x_n)yy'^*) = \lim_n \lan x_n \ten_D y,x'_n \ten_D y' \ran = \lan (x_n \ten_D y), (x'_n \ten_D y') \ran,
\end{align*}
which finishes the proof of the lemma.
\end{proof}
\begin{lemma} There exists an orthonormal basis $Y_{\alpha}$ of $L^2(M)$ over $B$ such that for every $n$, $f_n(Y_{\alpha})=0$ for all but finitely many $\alpha$'s, where we denote somewhat abusively $f_n(Y_{\al})=f_n(s)$, $s$= the degree of $Y_{\al}$.
\end{lemma}
\begin{proof} Since $D_s$ is finitely generated over $B$ for all $s$, according to Proposition 3.20, for every $s \geq 0$ we can find a finite orthonormal basis $(Y_{\beta}^s)$ of $L^2_s(M)$ over $B$. The union $(Y_{\al})$ of all the $Y_{\beta}^s$'s is a basis of $L^2(M)$ over $B$. For a fixed $n$, there exists $s=s(n)$ such that $f_n(\xi)=0$ for all $\xi \in H^{\ten k}$, for $k>s(n)$. For any $t\geq 0$ and $Y_{\alpha}\in L_t^2(M)$ we have $f_n(Y_{\alpha})=f_n(t)$ and also for every      $Y_{\al} \in \bigoplus_{k>s(n)} L^2_k(M)$ we have $f_n(Y_{\al})=0$, both due to the reduction formula. On the other hand, the set of those $Y_{\al} \in \bigoplus_{k=0}^{s(n)} L^2_k(M)$ is finite, which finishes the proof.
\end{proof}
Denote by $\iota$ the $M$-bimodular embedding in Lemma 5.2 and define
\[\mu_n^{\mathcal A}(T)=\sum_{\alpha} f_n(Y_{\alpha}) \langle T \iota(1 \otimes_B 1), \iota(Y_{\alpha}^* \otimes_B Y_{\alpha}^*)\rangle, T \in \mathcal S_{\mathcal A}.\]
Then $\mu_n^{\mathcal A} \in (\mathcal S_{\mathcal A})_*$ satisfies all the required properties.
\par {\bf Step 3.} Set $\gamma_n = \mu_n^{\mathcal A} \circ \mathcal E \in \mathcal N_*$, and $\omega_n = \|\gamma_n\|^{-1} |\gamma_n|$. We will prove that the $\omega_n$'s satisfy all the required properties. First note that, by construction,
\[\gamma_n(\pi(x)\theta(y^{op}))=\tau(\varphi_n(x)E_{\mathcal A}(y)), x \in M, y \in P.\]
Toward proving the required properties of the $\omega_n$'s, we will first establish the following two claims: \\
{\bf Claim 1.} $\limsup_n \|\mu_n^{\mathcal A}\|=1$;\\
{\bf Claim 2.} $\lim_n \|\mu_n^{\mathcal A} \circ Ad(\lambda(u)\rho(\bar{u}))-\mu_n^{\mathcal A}\|=0, u \in \mathcal{N}_M(\mathcal A)$.\\
\emph{proof of the first claim}. Fix a von Neumann subalgebra $Q \subset P$ which is amenable over $B$. Just as in Step 2 above one can construct normal functionals $\mu_n^Q$ on $\mathcal S_Q=\lambda(M) \vee \rho(Q^{op}) \subset B(\mathcal K)$ satisfying $\mu_n^Q(\lambda(x)\rho(y^{op}))=\tau(\varphi_n(x)y)$, for $ x \in M, y \in Q$. We will show that $\limsup \|\mu_n^Q\|=1$, and this will help us establish both claims. Since $\mu_n^Q$ is normal, it suffices to estimate its norm on an ultraweakly dense C$^*$-subalgebra of $\mathcal S_Q$. Denote by $S_Q$ the ultraweakly dense C$^*$-subalgebra of $\mathcal S_Q$ generated by $\lambda(x_{\si})$, for $x_{\si} \in M$ the Wick words and $\rho(Q^{op})$. First we note that there exist cb maps $\tilde \varphi_n: S_Q \to S_Q$ such that
\[\tilde \varphi_n(\lambda(x_{\si})\rho(y^{op}))=\lambda(\varphi_n(x_{\si}))\rho(y^{op}), x_{\si} \in M, y \in Q,\]
and $\|\tilde \varphi_n\|_{cb}=\|\varphi_n\|_{cb}$. To prove this take $\tilde{\mathcal K} \subset L^2((\mathcal M \bar{\ten} \Gamma_q(\ell^2 \ten H))^{\om})$ to be the $\|\cdot\|_2$-closed linear span of the sequences
\[(n^{-\frac{m}{2}}\sum_{(j_1,\ldots,j_m)=\si} \pi_{j_1}(x_1)\cdots \pi_{j_m}(x_m)y \ten s(e_{j_1}\ten h_1)\cdots s(e_{j_m}\ten h_m))=(x_{\si}^n(y \ten 1))\]
for all $x_i \in BSB, h_i \in H, y \in M$.
Now define an unitary operator $U:\mathcal K \to \tilde{\mathcal K}$ by
\[(x_{\si}^n \ten_D y) \mapsto (x_{\si}^n(y \ten 1)).\]
We can then define
\[\tilde \varphi_n(z) = U^*(id \otimes m_n)^{\om}(UzU^*)U, z \in S_Q.\]
Then the maps $\tilde \varphi_n$ satisfy all the required properties. The complete boundedness of the $\tilde \varphi_n$ is a delicate matter and it will be addressed in the subsection 5.1 below. On the other hand, since $Q$ is amenable relative to $B$, we see that the $M-Q$ bimodule $L^2(M)$ is weakly contained in $L^2(M) \otimes_B L^2(M)$, which in turn is contained in $\mathcal K$. This produces a *-homomorphism $\Theta:S_Q \to B(L^2(M))$ such that $\Theta(\lambda(x)\rho(y^{op}))=\lambda_M(x)\rho_M(y^{op})$, where $\lambda_M, \rho_M$ are the natural actions of $M$ on $L^2(M)$. But then
\[\mu_n^Q(z)= \langle \Theta(\tilde \varphi_n(z))1,1 \rangle, z \in S_Q,\]
and this implies that $\limsup \|\mu_n^Q\|=1$. Then by taking $Q=\mathcal A$ we get $\limsup \|\mu_n^{\mathcal A}\|=1$, which finishes the proof of the first claim.\\
\emph{proof of the second claim}. Fix a unitary $u \in \mathcal{N}_M(\A)$. The algebra $Q=\langle \A, u \rangle \subset P$ is amenable relative to $B$, so by the proof of Claim 1 $\limsup \|\mu_n^Q\|=1$. Now since $\mu_n^Q(1)=\tau(\phi_n(1))\to 1$ and $\mu_n^Q(\lambda(u)\rho(\bar{u}))=\tau(\phi_n(u)u^*) \to 1$, we see that $\|\mu_n^Q \circ Ad(\lambda(u)\rho(\bar{u}))-\mu_n^Q\| \to 0$, hence by restricting to $S_{\A}$ we get $\|\mu_n^{\A} \circ Ad(\lambda(u)\rho(\bar{u})) - \mu_n^{\A}\| \to 0$. Using the fact that $Ad(\lambda(u)\rho(\bar{u})) \circ \mathcal E = \mathcal E \circ Ad(\pi(u)\theta(\bar{u}))$ and the fact that $\gamma_n = \mu_n^{\A} \circ \mathcal E$ we see at once that $\|\gamma_n \circ Ad(\pi(u)\theta(\bar{u}))-\gamma_n\| \to 0$. But since $\gamma_n(1)=\tau(\phi_n(1))\to 1$ and $\limsup \|\gamma_n\|=1$ we see that $\|\gamma_n -\omega_n\| \to 0$. This further implies $\|\omega_n \circ Ad(\pi(u)\theta(\bar{u})) - \omega_n\| \to 0$, which establishes the third required property, and the other two follow easily.
\end{proof}

\subsection{CB-estimates for the multipliers}

Here we will prove that some multipliers defined on certain $C^*$-algebras or von Neumann algebras are completely bounded. The first case is that of the maps $\tilde \varphi_n$ which were used in the proof of Thm. 5.1. above. In the second case we prove the cb boundedness of some normal multipliers on the von Neumann algebra $\mathcal N$ introduced above, which are needed to construct a concrete standard form for $\mathcal N$. We recall some notation. 
\par {\bf Notation:} $\mathcal{M}=(\Gamma_q(\ell^2 \ten H) \bar{\ten} D)\vee M \subset (\Gamma_q(\ell^2 \ten H)\bar{\ten} D)^{\om}$, where we regard $\Gamma_q(\ell^2 \ten H)$ and $D$ as constant sequences.
Let $K=L^2(\mathcal M)$ or $K=L^2(\mathcal M) \ten_{\mathcal A} L^2(P)$. We introduce the subspace
 \[\mathcal L \subset (K \ten \mathcal{F}_q(\ell_2 \ten H))^{\om} \]
as the $\|\cdot\|$-closed linear span of the sequences
\[(n^{-\frac{m}{2}}\sum_{(j_1,\ldots,j_m)=\si} \pi_{j_1}(x_1)\cdots \pi_{j_m}(x_m)y \ten s(e_{j_1}\ten h_1) \cdots s(e_{j_m}\ten h_m))=(x_{\si}^n(y\ten 1))\in (K \bar{\ten} \Gamma_q(\ell^2 \ten H))^{\om},\]
for $m \geq 1, \si \in P_{12}(m), x_i \in BSB, h_i \in H, y \in M$. Let's define the extended Wick words $x_{\si}=x_{\si}(x_1,h_1,\ldots,x_m,h_m,y^{op})$ by
 \[x_{\si} \lel (n^{-\frac{m}{2}}\sum_{(j_1,\ldots,j_m)=\si} \pi_{j_1}(x_1)\cdots \pi_{j_m}(x_m)y^{op} \ten s(e_{j_1}\ten h_1)\cdots s(e_{j_m}\ten h_m)), \]
where $m \geq 1, \si \in P_{1,2}(m), x_i \in BSB, h_i \in H, y \in P$, viewed as operators in $B(\mathcal K)$, i.e. acting naturally on sequences in $\mathcal L$. The reader can check that
\begin{enumerate}
\item[a)] $\mathcal L$ is invariant to the natural action of the extended Wick words;
\item[b)] $\mathcal{L}=\overline{\rm span}\lbrace \lambda(x_{\si})\rho(y^{op})(1 \ten 1), x_{\si} \in M, y \in M \rbrace$ when $K=L^2(\mathcal M)$ and $\mathcal L=\overline{\rm span} \lbrace \pi(x_{\si})(1 \ten y)\theta(z^{op})((1 \ten_{\mathcal A}1)\ten 1), x_{\si} \in M, y \in M, z \in P \rbrace$ when $K=L^2(\mathcal M)\ten_{\mathcal A}L^2(P)$;
\item[c)] $\mathcal L$ is invariant to the natural action by orthogonal transformations of $H$ given by
\[\mathcal{O}(H) \ni o \to \al_o=(id \ten \Gamma_q(id \ten o)) \in Aut((\mathcal M \ten \Gamma_q(\ell^2 \ten H))^{\om}).\]
\end{enumerate}
Let $C(H) \subset B(\mathcal L)$ be the $C^*$-algebra generated by the elements
\[(n^{-\frac{m}{2}}\sum_{(j_1,\ldots,j_m)=\si} \pi_{j_1}(x_1)\cdots \pi_{j_m}(x_m)y^{op} \ten s(e_{j_1}\ten h_1) \cdots s(e_{j_m}\ten h_m))=(x_{\si}^n(y^{op}\ten 1)),\]
where $x_i \in BSB, h_i \in H, y \in M, \si \in P(m)$.
 Also let $\hat{C}(H)\subset (B(K) \ten_{\rm min} \Gamma_q(\ell^2 \ten H))^{\om}$ be the C*-algebra generated by the elements
\[(n^{-\frac{m}{2}}\sum_{(j_1,\ldots,j_m)=\si} \pi_{j_1}(x_1)\cdots \pi_{j_m}(x_m)y^{op} \ten s(e_{j_1}\ten h_1) \cdots s(e_{j_m}\ten h_m))=(x_{\si}^n(y^{op} \ten 1)),\]
where $x_i \in BSB, y \in M, h_i \in H, \si \in P(m)$, the ultraproduct being the $C^*$-algebra ultraproduct.
\begin{rem} Let $m_{\al}$ be the multipliers on $\Gamma_q(H)$ associated to the non-negative finite support functions $f_{\al}:\nz \to \rz$.
\begin{enumerate}
 \item One may assume that for every $k$, $f_{\al}(k)=1$ for $\al$ large enough and that $\limsup_{\al} \|m_{\al}\|_{cb} \lel 1$;
 \item $(id \ten m_{\al}):\hat{C}(H)\to (B(K) \ten_{\rm min} \Gamma_q(\ell^2 \ten H))^{\om}$ are completely bounded, and the restriction of a normal map.
 \end{enumerate}
\end{rem}
\begin{lemma} Let $\hat{C}(H)$, $C(H)$ and $m_{\al}$ be defined as above.
\begin{enumerate}
\item Let $\rho:(B(K) \ten_{min} \Gamma_q(\ell^2 \ten H))^{\om} \to B((K \ten \mathcal F_q(\ell^2 \ten H))^{\om})$ be the *-homomorphism defined by $\rho((T_n))(\xi_n)=(T_n \xi_n)$. Then $\rho(\hat{C}(H))(\mathcal L)\subset \mathcal L$, so $[\rho(\hat{C}(H)),P_{\mathcal L}]=0$.
\item The map $\Phi:\hat{C}(H) \to C(H)$ defined by $\Phi(T)=\rho(T)P_{\mathcal L}$ is a surjective *-homomorphism.
\item If $\si \notin P_{1,2}(m)$, then $\Phi((x_{\si}^n(y^{op}\ten 1))=0$. In particular, $C(H)=\Phi(\hat{C}(H))$ is spanned by the elements $\Phi((x_{\si}^n(y^{op}\ten 1)))$, for $m \geq 1, \si \in P_{1,2}(m)$.
\item If $(x_n)=(x'_n) \in M$, then $\Phi((x_n(y^{op}\ten 1)))=\Phi((x'_n(y^{op}\ten 1)))$. In particular, $C(H)$ is spanned by the elements $\Phi((W_{\si}(y^{op}\ten 1)))$, where $W_{\si} \in M, \si \in P_{1,2}(m)$ are the reduced Wick words.
\end{enumerate}
\end{lemma}

\begin{proof} Take $(x_{\si}^n(y^{op}\ten 1))\in \hat{C}(H)$, $(x_{\si'}^n(z \ten 1))\in \mathcal L$. Due to the convolution rule we have
\[\Phi((x_{\si}^n(y^{op}\ten 1)))(x_{\si'}^n(z \ten 1))=(x_{\si}^nx_{\si'}^n(zy\ten1))=\sum_{\gamma \in P(m+m')}(x_{\gamma}^n(zy\ten 1)),\]
the summation being taken over all those $\gamma$'s which preserve the connections of both $\si$ and $\si'$, i.e. if some indices are connected by $\si$ or $\si'$, they will remain connected in $\gamma$. Now for all $\gamma \notin P_{1,2}(m+m')$, the corresponding term vanishes, because $\|x_{\gamma}^n(zy\ten 1)\|_2 \leq \|zy\|_{\infty}\|x_{\gamma}^n\|_2 \to 0$. Thus
\[\Phi((x_{\si}^n(y^{op}\ten 1)))(x_{\si'}^n(z \ten 1))=\sum_{\gamma \in P_{1,2}(m+m')}(x_{\gamma}^n(zy\ten 1)) \in \mathcal L,\]
which proves 1. Also, if $\si \notin P_{1,2}(m)$ to begin with, every $\gamma$ in the sum will also not be in $P_{1,2}(m+m')$, hence the whole sum vanishes, which proves 3. The second statement is trivial. If $(x_n),(x'_n) \in M$ such that $\lim \|x_n - x'_n \|_2 =0$, then for every $(y_{\si}^n(z\ten 1))\in \mathcal K$, we have
$\|x_ny_{\si}^n(zy\ten 1)-x'_ny_{\si}^n(zy \ten 1)\|_2 \leq \|y_{\si}^n\|_{\infty}\|zy\|_{\infty}\|x_n-x'_n\|_2 \to 0$, i.e. $\Phi((x_n(y^{op}\ten 1)))=\Phi((x'_n(y^{op}\ten 1)))$. The last statement then follows from the reduction formula.
\qd

Our goal is to prove that under certain conditions the maps $(id \ten m_{\al})$ descend to a multiplier on the quotient algebra, namely $C(H)$. This is done via  a careful analysis of $\Phi_*$.

\begin{lemma}\label{opdual} There exists a complete contraction
 \[ \psi: (\overline{K\ten L^2(\Gamma_q(\ell^2 \ten H)))}^r\ten_h  (K\ten L^2(\Gamma_q(\ell^2 \ten H))^c\to \overline{L^1(B(K)\ten \Gamma_q(\ell^2 \ten H))} \]
such that
  \[  \psi((h\ten a)\ten (k\ten b)) \lel  (h\ten \bar{k}) \ten ab^* \]
and $tr((S\ten T)(\psi(   (h\ten a)\ten (k\ten b))^*)=((S\ten T)(k\ten b),h\ten a)$. Here
$(k\ten h)$ is the  rank one operator with entries $(k_ih_j)$ in a given basis and $\ten_h$ denotes the Haagerup tensor product of operator spaces.
\end{lemma}

\begin{proof} We recall that for a semifinite von Neumann algebra $M$ the space $M=\overline{L^1(M,tr)}^*$ is the antilinear dual with respect to trace $\lan T,\rho\ran=tr(T\rho^*)$.  Moreover, for $M=B(H)$ one usually considers linear duality with respect to the transposed $\rho^t$ of a density $\rho$:
 \[ \lan\lan T,\rho\ran\ran_{B(H),S_1(H)} \lel tr(T\rho^t) \lel tr(T\overline{\rho}^*) \lel \lan T,\bar{\rho}\ran_{B(H),\overline{S_1(H)}} \pl .\]
Using the description of $S_1(H)=H^r\ten_h H^c$ as a Haagerup tensor product, we find a natural map $\om:
H^r\ten_h H^c\to B(H)^*$ given by
 \[ \om(h\ten k)(T)
 \lel tr(T(\sum_{ij} h_{i} k_je_{ij})^t)
 \lel tr(T(\sum_{ij} h_{i} k_je_{ji}))
 \lel \sum_{ij} T_{ij}h_ik_j \lel (T(k),\bar{h}) \pl .\]
Let $M$ be a semifinite von Neumann algebra and $(\xi_j)$ be an orthonormal basis. Then we may define the antilinear map $v(a)=\sum_j \langle \xi_j,a\ran \xi_j$ and observe that
  \begin{align*}
  (b,\overline{v(a)})&= \sum_j \lan b,\xi_j\ran \lan v(a),\xi_j\ran
  \lel \sum_j \lan b,\xi_j\ran \lan \xi_j,a\ran
  \lel \tau(ba^*) \pl .
  \end{align*}
Therefore $\bar{m}=\om(v\ten id):\overline{L^2_r(M)}\ten_h L^2_c(M)\to B(L^2(M))^*$ satisfies
 \[ m(a\ten b)(T) \lel (T(b),\overline{v(a)})
 \lel \tau(Tba^*) \lel \tau(T(ab^*)^*))
 \lel \lan T,(ab^*)\ran  \]
for all $T\in M$. This shows that $m(a\ten b)=ab^*$ is a complete contraction from $\overline{L^2_r}(M)\ten_h L^2_c(M)\to \overline{L^1(M)}$. Now we repeat the argument for $H=K\ten L^2(M)$  and $V(h\ten b)=\bar{k}\ten v(b)$. Then we obtain a complete contraction $\psi=\om(V\ten id):\overline{(K\ten L^2(M))}^r\ten_h K\ten L^2(M)\to B(K\ten L^2(M))^*$ such that for $S\in B(K)$ and $T\in M$
 \begin{align*}
 & \psi((h\ten a)\ten (k\ten b))(S\ten T))
   (S(k),\overline{ \bar{h} }) \tau(Tba^*)
 \lel (S(k),h) (Tb,a)
 \lel (S\ten T(k\ten b),h\ten a)  \\
 &=  (tr\ten \tau)((S\ten T) (k\ten \bar{h})
\end{align*}
Here $(\al\ten \beta)=\sum_{ij} e_{ij} \al_i \beta_j$ is the density of the corresponding rank one operator.
Therefore the map $\psi((k\ten a)\ten (h\ten b))= (h\ten \bar{k})\ten  ab^*$ does the job. \qd

\begin{cor} Let $\mathcal{L}\subset (K\ten L^2(\Gamma_q(\ell^2 \ten H))^{\om}$ be defined as above. Then there exists a completely contractive map
 \[ \Psi: \overline{\mathcal L}^r \ten_h \mathcal L^c \to (\overline{L_1(B(K)\ten \Gamma_q(\ell^2 \ten H)}))^{\om} \]
and a complete contraction $q: (\overline{L_1(B(K)\ten \Gamma_q(\ell^2 \ten H)})^{\om}\to [(B(K)\bar{\ten}\Gamma_q(\ell^2 \ten H)^{\om}]^*$ such that
 \[   (q \circ \Psi)(k\ten h)(T) \lel  \lan T(k),h\ran
  \pl .\]
In particular $\Psi^*|_{\hat{C}(H)}=\Phi$.
\end{cor}

\begin{proof} For $\xi, \eta \in  \mathcal{L}$ given by $\xi=(\xi_n)_n$, $\eta=(\eta_n)_n$ we may define
 \[ \Psi(\xi \ten \eta) \lel (\psi(\xi_n\ten \eta_n))_n \pl ,\]
where $\psi$ is the map from Lemma \ref{opdual}. Now $\Psi$ obviously extends by linearity, thanks to the definition of the Haagerup tensor product and the well-known fact that $M_m((X_n)^{\om})=(M_m(X_n))^{\om}$ (see \cite{PisierOS}).  The map $q$ is given by the limit
 \[ q((\overline{\zeta_n})_n)(T_n)_n \lel \lim_{n\to\om} (tr\ten \tau)(T_n \zeta_n^*) \]
Now the assertion follows from Lemma \ref{opdual} and the fact that the duality pairing is given by the limit along the ultraproduct. \qd

\begin{rem}\label{Hinff} {\rm Let $H$ be an infinite Hilbert space and $H\subset H'$. Thanks to the definition of the $C^*$-algebra $\hat{C}(H)$ as a subalgebra of the ultraproduct, we clearly have an isometric inclusion $\hat{C}(H)\subset \hat{C}(H')$.
The $C^*$-algebra $C(H)\subset B(\mathcal L(H))$ depends on our minimalistic definition of $\mathcal L(H)$. Certainly, $\mathcal L(H)\subset \mathcal L(H')$ and hence the tautological map $\iota(x_{\si}) =x_{\si}$, $\iota(y^{op})=y$ produces  a larger norm on $\mathcal L(H')$ than on $\mathcal L(H)$. Let us consider a noncommutative polynomial $p$ in a finite number of $x_{\si}$'s and $y^{op}$'s, and we may  assume that the $x_{\si}$ only contain vectors from a finite dimensional subspace $H_0 \subset H$. Then we can find norm attaining vectors $\xi,\eta\in \mathcal L(H')$ for $p$. Then we write $H'=H_0 \oplus H_0^{\perp}$ and may also assume that the $\xi$ and $\eta$ are linear combination of elements in $\mathcal L(H_0)$ and $\mathcal L(H_1)$ where $H_1\subset H_0^{\perp}$ is a finite dimensional subspace. Using the moment formula, we see that  the inner product remains unchanged after applying an orthogonal transformation $o$ which sends $H_1$ to a finite dimensional subspace of $H$ orthogonal to $H_0$ and leaves $H_0$ invariant. This implies that
 \[ \|p\|_{C(H')} \lel \sup_{\|\xi\|\leq 1,\|\eta\|\leq 1}  |\lan \xi,p\eta \ran |
 \lel \sup_{\|\xi\|\leq 1,\|\eta\|\leq 1} |\lan \al_o(\xi),p\al_o(\eta)\ran |
 \kl \|p\|_{C(H)} \pl .\]
Let us denote by $q_H=\Phi|_{\hat{C}(H)}:\hat{C}(H)\to C(H)$ the quotient map. Then we obtain a commutative diagram
 \[ \begin{array}{ccc}
     \hat{C}(H) & \stackrel{q_H}{\to} & C(H) \\
     \downarrow & & \downarrow \\
   \hat{C}(H') & \stackrel{q_{H'}}{\to} & C(H')
   \end{array} \pl ,\]
where the left hand downward arrow is the natural ultraproduct inclusion and the right hand downward arrow is the tautological inclusion (which is well-defined and injective). This allows us to identify elements in the kernel of  $q_H$ by considering $q_{H'}$.
}\end{rem}

We recall that thanks to Avsec's result, the orthogonal projection $P_k:\Gamma_q(H)\to \Gamma_q(H)$ onto Wick words of length $k$ is a normal completely bounded map. We use the same notation $P_k:L^1(\Gamma_q(H))\to L^1(\Gamma_q(H))$  and $id\ten P_k:\overline{L^1(B(K)\ten \Gamma_q(H))}\to
\overline{L^1(B(K)\ten \Gamma_q(H))}$. Let us note that one can take $P_k^{\om}:\prod \overline{L^1(B(K)\ten \Gamma_q(H))}$, the extension to the ultraproduct of $L^1$ spaces, which satisfies
 \[\langle (id\ten P_k)((T_n)),(\xi_n)\rangle
 \lel \lan (T_n),(id \ten P_k)((\xi_n))\ran \]
with respect to the anti-linear bracket given by the ultraproduct trace (see also \cite{Raynaud}).

\begin{lemma}\label{kernel} The kernel of $\Phi\circ P_k^{\om}$ contains the kernel of $q_H$.
\end{lemma}
\begin{proof} The map $\Phi \circ P_k^{\om}$ is normal. According to Remark \eqref{Hinff} it therefore suffices to show that for $\xi,\eta\in \mathcal L$ we have
 \[ (id\ten P_k^{\om})(\Psi(\xi\ten \eta)) \in {\rm Im}( \psi_{H'})  \]
for some potentially larger Hilbert space $H'$. Let us now consider Wick words $(x_{\si}^n)_n$, $\tilde{x}_{\si}^n$, and $y^{op}$, $\tilde{y}^{op}$. We have to consider
 \begin{align*}
  \Psi( (\tilde{x}_{\si}\tilde{y}^{op})_n \ten  (x_{\si}y^{op})_n)
  &= (\psi(\tilde{x}_{\si}^ny\tilde{y}^{op} \ten x_{\si}^ny^{op}))_n \pl .
  \end{align*}
For fixed $\nen$ we see that
 \begin{align*}
  &\Psi((\tilde{x}_{\tilde{\si}}^n y)\ten  (x_{\si}^ny^{op}))
  \lel n^{-(m+\tilde{m})/2} \sum_{(\tilde{j_k})=\tilde{\si},(j_k)=\si}
   (\overline{\vec{\pi}(a)y^{op}}\ten
 \vec{\pi}_{\tilde{j}_1}(\tilde{a})\tilde{y}^{op})\ten  \vec{s}_{\tilde{j}}(\tilde{h})\vec{s}_{j}^{\p *}
  \\
 & \lel \sum_{\si'\in P(m+m')} \Psi^{\si'}(\tilde{x}_{\tilde{\si}}\tilde{y}^{op}\ten
 x_{\si}y^{op})
 \pl,
  \end{align*}
where
\[\Psi^{\si'}(\tilde{x}_{\tilde{\si}}\tilde{y}^{op}\ten
 x_{\si}y^{op})=\sum_{(\tilde{j}_1,...,\tilde{j}_{\tilde{m}},j_m,...,j_1)=\si'}  (\overline{\vec{\pi}(a)y^{op}}\ten
 \vec{\pi}_{\tilde{j}_1}(\tilde{a})\tilde{y}^{op})\ten  \vec{s}_{\tilde{j}}(\tilde{h})\vec{s}_{j}^{\p *}.\]

Note also that $\si'$ has to be obtained by joining singletons from $\tilde{\si}$ and  $\si$. In this context we observe again that is enough to consider $\si'\in P_{1,2}(\tilde{m}+m)$. In the following example we see that
 \begin{align*}
  &\|\sum_{j_1} (\overline{\pi_{j_1}(a_1)\pi_{j_1}(a_2)y^{op}}\ten  \pi_{j_1}(\tilde{a}_1)\tilde{y}^{op})\ten s_{j_1}^2s_{j_1}\|_1 \\
  &\le    \|\sum_{j_1}  \pi_{j_1}(a_1)\pi_{j_1}(a_2)y^{op} \ten s_{j_1}^2 \ten e_{1,j_1} \|
  \|\sum_{j_1} \pi_{j_1}(\tilde{a}_1)\tilde{y}^{op}\ten s_{j_1} e_{j_1,1}\|  \\
  &\le c_q(a,\tilde{a}) n \ll n^{3/2}
  \pl
  \end{align*}
is much smaller than the predesigned $n^{3/2}$ and hence vanishes in the limit. For more complicated configurations, we may assume that $\si$ and $\tilde{\si}$ are pair/singleton partitions, and that new links in $\si'\in P_{1,2}(m+\tilde{m})$ are obtained from joining pairs or singletons  in $\si$ with pairs in $\tilde{\si}$ (or the other way round). All the joint pairings can be estimated using the definition of the Haagerup tensor product as above which yields the bound
 \begin{align*}
 &\|\sum_{(\tilde{j}_k)=\tilde{\si},(j_k)=\si,(\tilde{j}_k,j_k)=\si'}
 (\overline{\vec{\pi}_j(a)y^{op}}\ten \vec{\pi}_j(\tilde{a})\tilde{y}^{op})\ten \vec{s}_{\tilde{j}}(\tilde{h})^*\vec{s}_{j}(h)\| \\
 &\le  c_q n^{f(\si,\tilde{\si},\si')} \sup_j \|a_j\|\sup_j\|\tilde{a}_j\| \|y^{op}\|\|\tilde{y}^{op}\| \pl .
 \end{align*}
The function $f$ is obtained as follows. Let $\al$ be the number of pairs in $\si$ being linked to either a pair or singleton in $\tilde{\si}$, and similarly $\beta$ be the number of linked pairs. Then we find
 \[ f(\si,\tilde{\si},\si')\lel \frac{|\si_s|}{2}+|\si_p|-\al + \al/2 +
 \frac{|\tilde{\si}_s|}{2}+|\tilde{\si}_p|-\beta + \beta/2
 \lel \frac{m+\tilde{m}}{2}-\frac{\al+\beta}{2} \]
using row and column vectors $e_{1,i_1,...,i_{l}}$, $e_{i_1,...,i_l,1}$ for the number $l$ of links in $\si'$. Thus for $\al+\beta>0$ we obtain $0$ in the limit and therefore only those $\si'$ which link singletons to singletons give a contribution in the limit.  Now we use Pisier's version \cite[Sublemma 3.3]{Port} of  the  M\"obius transform. Let $\si'$ be a fixed partition with pairs $\{\{l_1,r_1\},...,\{l_p,r_p\}\}$. Then there are unitaries $\la_j^{\si'}$ in a product of free group factors such that
 \[ S_j(h) \lel s_j\ten \la_j^{\si'} \]
satisfies
 \begin{align*}
  a(\si') \pl:=\pl  \sum_{\si''\gl \si'} \Psi^{\si'}(\tilde{x}_{\tilde{\si}}\tilde{y}^{op}\ten
 x_{\si}y^{op})
 &=   \sum_{ \tilde{j}_k,j_k}
 (\overline{\vec{\pi}_j(a)y^{op}}\ten \vec{\pi}_j(\tilde{a})\tilde{y}^{op})\ten (id\ten E) (\vec{S}_j(\tilde{h})^*\vec{S}(h))  \\
 &= (id\ten E) \psi( \tilde{X}_{\tilde{\si}}^{\si'}\ten X_{\si}^{\si'})
 \pl .
 \end{align*}
Here
 \[ X_{\si}^{\si'} \lel (n^{-m/2} \sum_{(j_1,...,j_m)\le \si} \pi_{j_1}(a_1)\cdots \pi_{j_m}(a_m) \ten S_{j_1}(h_1)\cdots S_{j_m}(h_m)) \]
and the corresponding expression for $\tilde{X}_{\tilde{\si}}^{\si'}$ depends on $\si'$. Moreover,
there exists a M\"obius function $\mu(\cdot,\cdot)$ such that (see \cite[Proposition 1]{Port})
 \[  \Psi^{\si'}(\tilde{x}_{\tilde{\si}}\tilde{y}^{op}\ten
 x_{\si}y^{op}) \lel \sum_{\pi\gl \si'} \mu(\si',\pi) a(\pi) \pl .\]
The advantage of this representation comes from the fact that we can actually calculate $P_k$ for such a fixed $\si'$. Recall that we may assume that $\si'$ is a pair/ singleton partition. For fixed $\nen$ and an element $\eta_n=S\ten W_n$, $W_n$ a Wick word of length $k$ we obtain
 \[  tr(S((\overline{\vec{\pi}_j(a)y^{op}}\ten \vec{\pi}_j(\tilde{a})\tilde{y}^{op})) \tau(W_n
 s_{\tilde{j}_1}\cdots s_{\tilde{j}_{\tilde{m}}}s_{j_m}\cdots s_{j_1})  \]
such that $(\tilde{j}_1,...,\tilde{j}_{\tilde{m}},j_m,...,j_1)=\si'$. Then we obtain a non-zero term only if $|\si'_s|=k$ has exactly $k$ singletons. Hence we find that
 \begin{align*}
  (id\ten P_k)(\psi(\xi\ten \eta))
 &=  \sum_{|\si'_s|=k}
 \Psi^{\si'}(\tilde{x}_{\tilde{\si}}\tilde{y}^{op}\ten
 x_{\si}y^{op})   \lel
 \sum_{|\si'_s|=k} \sum_{\pi \gl \si'} \mu(\si,\pi) a(\pi) \pl .
 \end{align*}
Therefore we are left to consider
 \[ a(\pi) \lel (id\ten E)(\Psi(\tilde{X}_{\tilde{\si}}^{\pi}\ten X_{\si}^{\pi}))  \pl .\]
In order to use Remark \ref{Hinff} we have to modify the variables $X_{\si}^{\pi}$. Indeed, for every pair $p=\{l,r\}$ in $\pi$ we introduce a label $e_{p}$ and replace $s(e_{j_l}\ten h_l)$ by $s(e_{j_l}\ten e_p\ten h_l)$, and $s(e_{\tilde{j}_r}\ten \tilde{h_r})$ by $s(e_{\tilde{j}_r}\ten e_p\ten h_r)$. For the remaining singletons we replace  $s(e_j\ten h_j)$ by $S_j=s(e_j\ten e_0)$ and work in the Hilbert space $H'=H\ten \ell_2$. Using the so modified  $X_\si^{\pi}$'s we still have
 \[ a(\pi)  \lel (id\ten E_{\Gamma_q(\ell_2\ten H\ten e_0)})_n \Psi( \tilde{X}_{\tilde{\si}}^{\pi}\ten X_{\si}^{\pi})
 \lel \lim_{j\to \infty} \Psi( \al_{o_j}(\tilde{X}_{\tilde{\si}}^{\pi}) \ten \al_{o_j}(X_{\si}^{\pi}))
 \pl \]
for any sequence $(o_j)$ of orthogonal transformations such that $o_j(e_0)=e_0$, which converges weakly to $e_0^{\perp}$. For elements in $\hat{C}(H)$ the limit for $j\to \infty$ converges, and hence this remains true for the norm closure. Thus for an element $x\in \hat{C}(H)$ in the kernel of $q_H$ we find $q_{H'}(x)=0$ and hence \[ \langle  x,a(\pi)\rangle \lel \lim_{j\to \infty} \langle x,\Psi(\al_{o_j}(\tilde{X}_{\tilde{\si}}^{\si'}) \ten \al_{o_j}(X_{\si}^{\si'})\rangle  \lel 0 \pl .\]
Using linear combinations we deduce indeed that $\lan P_k(x),\Psi(\tilde{x}_{\si}\tilde{y}^{op}\ten x_{\si}y^{op})\ran=0$. \qd

\begin{cor} Let $m_{\al}$ be multipliers given by the $cb$-approximation property for $\Gamma_q(H)$.
\begin{enumerate}
\item[i)] Then $(id \ten m_{\al})_n$ extend to completely bounded maps on $C(H)$ with $\limsup_{\al}\|(id \ten m_{\al})_n\|_{cb}=1$, and $\lim_{\al} f_{\al}(k)=1$, where $f_{\al}$ are the associated scalar finitely supported functions. In particular, the maps $\tilde \varphi_n$ used in the proof of Thm. 5.1 above are completely bounded with $\limsup_n \|\tilde \varphi_n\|_{\rm cb} =1$.
\item[ii)] Let $L(H)=\overline{C(H)}^{so} \subset B(\mathcal L)$ and note that $L(H)$ is spanned by "extended Wick words" (i.e. images of extended Wick words through $\Phi$) such that $L^2_k(L(H))$ (i.e. the $\|\cdot\|_2$-closed linear span of the extended Wick words of degree $k$) is finitely generated over $B$. Then there exists a modified family $f_{\al}(N)^*:L(H)_*\to L(H)_*$ converging in the point norm topology.
\end{enumerate}
\end{cor}

\begin{proof} Since $(id \ten m_{\al})(T)=\sum_{k} f_{\al}(k)P_k(T)$, we see that $\|\Phi_H \circ (id \ten m_{\al})\|_{cb}\le 1+\eps_{\al}$ and also $\ker(q_H)\subset \ker(\Phi_H \circ (id \ten m_{\al}))$. But that means that there is a unique map $\tilde{m}_{\al}:\hat{C}(H)/\ker(q_H)\to B(\mathcal{K})$ such that
 \[ \|\tilde{m}_{\al}\|_{cb} \lel \|m_{\al}\|_{cb}\le 1 +\eps_{\al} \pl .\]
However,  $\hat{C}(H)/\ker(q_H)=C(H)$ completely isometrically, and hence $\tilde{m}_{\al}=(id \ten m_{\al})$ coincides with the densely defined map  $(id \ten m_{\al})W(\si,\xi,a,y)=f_{\al}(|\si_s|)W(\si,\xi,a,y)$. Let us now consider a finite dimensional subspace $H_0\subset H$. Since $L^2_k(L(H))$ is finitely generated over $B$, we deduce that the projection $P_d$ is normal on $L(H_0)$.  Hence the maps  $m_{\al}$ are also normal and restricted to the weakly dense subspace $C(H_0)$ we know that
  \[ \|m_{\al}\|_{cb}\kl (1+\eps_{\al}) \pl .\]
Since a weakly dense subspace is norming for $L(H_0)_*$ we deduce that $\|(m_{\al})_*:L(H_0)_*\to L(H_0)_*\|_{cb}\le (1+\eps_{\al})$. Hence the normal map $m_{\al}$ coincides with the normal map
$((m_{\al})_*)^*$ and satisfies the same cb-norm estimate.
Moreover, since we have normal conditional expectations $\E_{H_0}:L(H)\to L(H_0)$ so that $\cup_{H_i} \E_{H_i}(L(H_i))_*$ is norm dense in $L(H)_*$, we deduce that $(m_{\al})_*$ extends to a completely bounded map of $cb$-norm at most $(1+\eps_{\al})$ and hence $m_{\al}=((m_{\al})_*)^*$ is indeed a normal extension of the map $m_{\al}:C(H)\to C(H)$ with the same cb-norm estimate. This concludes the proof of ii).
\qd
The remainder of the subsection is devoted to proving some auxiliary results which will help us construct a standard form for the von Neumann algebra $\mathcal N$ which was used in the proof of Thm. 5.1. This standard form will be crucial in the proof of the main technical theorem.

\begin{lemma}\label{invariance}  There exists an action by *-automorphisms $\al:\mathcal O(H) \to {\rm Aut}(\N)$ such that
 \[ \al_o(\pi(x)\theta(y^{op})) \lel \pi(\al_o(x))\theta(y^{op}), o \in \mathcal{O}(H), x \in M, y^{op} \in P^{op} \pl .\]
Moreover, let $E_0$ be the orthogonal projection of $\mathcal L$ onto the closed linear span of the extended Wick words of degree zero.  For $T\in \N$ the condition
 \[ \al_o(T) \lel T, \quad \forall o \in \mathcal{O}(H) \]
implies that $[T,E_0]=0$.
\end{lemma}

\begin{proof} Let us recall that $\N$ acts on
 \begin{align*}
 \mathcal{H} &=  {\rm span}\{ \pi(x_{\si})(y \ten 1)\theta(z^{op})((1 \ten_{\mathcal A}1)\ten 1), x_{\si} \in M, y \in M, z \in P \} \\
 &  \subset ((L^2(\mathcal{M})\ten_{\A}L^2(P))\ten L_2(\Gamma_q(\ell^2\ten H)))^{\om} \pl .
 \end{align*}
Recall here that $H$ is infinite dimensional, and thanks to second quantization $u_o=(id\ten \al_{o})_n$ acts on $\mathcal{H}$  as a unitary. By normality, we deduce that $\al_o(x)=u_oxu_o^*$ extends to a *-automorphism of $\N$ and moreover,
$\al_o(\theta(y^{op}))=\theta(y^{op})$.  Let $o_{i} \in \mathcal O(H)$ be a family of orthogonal transformations of $H$ such that $o_i(h)$ goes to $0$ weakly in $H$.
Let $\xi=\pi(x_{\si})(y\ten_{\A}z\ten 1)$ and $\eta=  \pi(x'_{\si'})(y'\ten_{\A}z'\ten 1)$. Then we obtain
  \begin{align*}
   \lim_{i} (u_{o_i}(\xi),\eta)
   &= \lim_{i} \lim_{n\to\om} n^{-(m+m')/2} \sum_{(j_k)=\si,(j'_{k'})=\si'}
    ( \vec{\pi}_j(x)(y\ten_\A z),\vec{\pi}_{j'}(x')(y'\ten_\A z')\\
   &\quad   \quad \quad \quad
    \tau(s_{j_1}(o_i(h_1))\cdots s_{j_m}(o_i(h_m))s_{j'_{m'}}(h_{m'}')\cdots s_{j'_1}(h'_1)) \lel 0
   \end{align*}
Indeed, we expand the sum into the  summation over $\si''\in P_{1,2}(m+m')$ and execute the limit over $n$. Then we observe that the coefficients remain uniformly bounded. However, $o_i(h_k)$ is eventually orthogonal to every $h_{k'}'$ and then the moment formula for $q$-gaussian yields $0$ in the limit. We have therefore shown that $u_{o_i}$ converges weakly to $E_0$, the projection onto words of length $0$ in the second component. By taking  convex combinations we find a net such that
 \[SOT-\lim_{s} \sum_{i} \al^s_i u_{o_i} \lel E_0 \pl .\]
Thus for $T\in N$ with $\al_o(T)=T$ for all $o$, we deduce that $[u_o,T]=0$ and hence
 \[ E_0(T(\xi)) \lel \lim_{s} \sum_i \al^s_i u_{o_i}T(\xi))
 \lel T(\lim_{s} \sum_i \al^s_i u_{o_i}(\xi))
 \lel T(E_0(\xi)) \pl .\]
This means $E_0T=TE_0$ as desired. \qd

\begin{lemma}\label{normal} Let $B\vee P^{op}\subset B(L^2(\mathcal{M})\ten_{\A}L^2(P))$. Then the natural inclusion map 
\[\pi:B\vee P^{op}\to \N\] 
is normal.
\end{lemma}

\begin{proof} By density it suffices to consider $\xi_n=\pi(x_{\si}^n)(y\ten_{A}z)$ and $\eta_n=
\pi(\tilde{x}_{\tilde{\si}}^n)(\tilde{y}\ten_{A}\tilde{z})$. We may assume that $x_{\si}$ and $\tilde{x}_{\si}$ is a Wick word. Our goal is to analyze
 \[ \phi(T) \lel \lim_{n\to\om} \lan T\xi_n,\eta_n\ran \pl .\]
Let us first fix $\nen$. Then $\om_n(T)= \lan T\xi_n,\eta_n\ran $ is normal, and hence it suffices to assume $T=b\theta(p^{op})$. It turns out that we need $|\si|=|\tilde{\si}|=k$ and then
 \begin{align*}
 \om_n(T)
 &=   \frac{n\cdots (n-k+1)}{n^k}
 \sum_{\gamma\in S_k} q^{{\rm inv}(\gamma)}  \tau(\tilde{z}^*E_A(\tilde{y}^*\pi_{\gamma(k)}(\tilde{x}_k)\cdots \pi_{\gamma(1)}(\tilde{x}_1) b\pi_1(x_1)\cdots \pi_{k}(x_k)y)zp) \pl.
 \end{align*}
Thanks to Lemma \ref{12}, we may replace $L^2(\mathcal{M})$ by $L^2(D)\ten_B L^2(\mathcal{M})$ in the definition of $\H$.  For fixed $\gamma$ we may now define
\[ x_{\gamma} \lel \al_{1,...,k}(x)\ten_B y \ten_{\A}z \pl ,
    \tilde{x}_{\gamma} \lel \al_{\gamma(1),...,\gamma(k)}(\tilde{x})\ten_B\ten \tilde{y}\ten_{\A} \tilde{z} \pl .\]
Since $\om_n$ is normal we deduce that
 \[ \om_n(T) \lel \sum_{\gamma} q^{{\rm inv}(\gamma)}
  \frac{n\cdots (n-k+1)}{n^k} \lan T(x_{\gamma}),\tilde{x}_{\gamma}\ran \]
for all $T\in B\vee P^{op}$. Since the summation is finite and the scalar coefficients converge the limit exists for all $T\in B\vee P^{op}$ and result in a normal functional $\phi(T)$ given by the same sum but with coefficient $1$ instead of $\frac{n\cdots (n-k+1)}{n^k}$.\qd

\begin{prop}\label{L2space} Assume that for every finite dimensional Hilbert space $H$, $L^2_k(M(H))$ is finitely generated as a right $B$-module (note that in particular this is the case if $dim_B(D_k(S))<\infty$, for all $k$). Then
 \begin{enumerate}
 \item[i)] There exists a faithful normal conditional expectation $\E:\N \to B_P=\pi(B)\vee \theta(P^{op})$;
 \item[ii)] The action $\al$ is implemented by an sot-continuous family of unitary operators $(V_o)_{o \in \mathcal{O}(H)}$ on $L^2(\N)$;
 \item[iii)] $L^2(\N)=\overline{\bigoplus_{k \geq 0} W_k(M)L^2(B_P)}$ and $V_o(\pi(x_{\si})\xi)=\pi(\al_o(x_{\si}))\xi$ for $x_{\si} \in M, \xi \in L^2(B_P)$. Moreover, $\E|_{\pi(M)}=E_{B}$, where $E_{B}:\pi(M) \to \pi(B)$ is the conditional expectation.
\end{enumerate}
\end{prop}

\begin{proof} For a subspace $H'\subset H$ we use the notation
 \[ \mathcal{H}(H')=\{\pi(x_{\si})((y\ten_{\A}z)\ten 1)| y\in M,z\in P, x_{\si}=x_{\si}(x_1,...,x_m,h_1,...,h_m), h_i\in H'\} \]for the
subspace generated by $H'$-Wick words.  Let $\iota_{H'}:\mathcal{H}(H')\subset \mathcal{H}$ be the canonical  inclusion map and $F_{H'}(T)=\iota_{H'}^*T\iota_{H'}$ the induced completely positive map. Certainly, we have  $F_{H'}(\theta(y^{op}))=\theta(y^{op})$ and
 \[ F_{H'}(x_{\si}) \lel E_{H'}(x_{\si}) \pl .\]
Indeed, if a Wick word $x_{\si}$ contains a singleton $h_i\in (H')^{\perp}$, then $F_{H'}(x_{\si})=0$. Using $h_i\in H'\cup (H')^{\perp}$ we deduce the assertion by linearity. Thus $F_{H'}(\N(H))=\N(H_i)\subset B(\H(H'))$ defines a normal surjective conditional expectation $F_{H'}$. Let $e_{H'}$ be the support of $F_{H'}$. We observe that $\pi(M(H'))$ and $\theta(P^{op})$ belong to the multiplicative domain of $F_{H'}$. Let $\tilde{N}(H')\subset \N(H)$ be the von Neumann algebra generated by $\pi(M(H'))$ and $\theta(P^{op})$ inside $\N_P(H)$. According to remark \ref{Hinff} and Kaplansky's density theorem, we deduce that $F_{H'}$ induces the same weak$^*$ topology on the unit ball of $\tilde{N}(H')$. This means that the tautological embedding $\si_{H'H}:\N(H_i)\to \N(H)$ given by $\si_{H'H}(x_{\si})=x_{\si}$ and $\si_{H'H}(\theta(y^{op}))=\theta(y^{op})$ satisfies
$F_{H'}\si=id_{\N(H_i)}$ and $F_{H'}$ is an isomorphism when restricted to $\tilde{N}(H')$. We denote by $\E_{H'}=\si_{H_iH}F_{H'}:\N\to \N$ the resulting, not necessarily faithful, conditional expectation. Let $H_{i}$ be an increasing net of finite dimensional spaces whose union is dense. Since $\bigcup_{i} \mathcal{H}(H_{i})$ is norm dense, we deduce that $\hat{\E}_{H_{i}}(x)$ converges weakly to $x$ as $i$ goes to infinity along the net of finite dimensional subspaces. Recall that the multiplier maps $m_{\al}$ are normal and commute with every $\mathcal{E}_{H_i}$. Adding convex combination we may find a new completely contractive net, still denoted by $m_{\al}$,   converging in the strong, strong$^*$ operator topology. Thus we may assume that
 \begin{equation}\label{conv}
 \lim_i \lim_{\al} (\hat{E}_{H_i}(m_{\al}x)) \lel x
 \end{equation}
converges strongly for all $x\in \N$. In our next step we consider $H'=0$, i.e. the map $\iota:L^2(M)\ten_{\A}L^2(P)\to \H$, given by $\iota(y\ten_{\A} z)\lel (y\ten_{\A}z)\ten 1$. This yields a completely positive map $\Phi(T)=\iota^*T\iota$ such that $\Phi(\theta(y^{op}))=\theta(y^{op})$ and $\Phi(\pi(b))=\pi(b)$. On the other hand for a Wick word $x=W_{\si}$, we see that
 \[ \lan \pi(x)\iota(y\ten_{\A}z),\iota(y'\ten_{\A}z')\ran 
 \lel \lim_{n\to\om} n^{-m/2} \sum_{(i_j)=\si}
 \lan \vec{\pi}(x)(y\ten_{\A}z), y'\ten_{\A}z'\ran 
 \tau(s_{j_1}(h_1)\cdots s_{j_m}(h_m)) \lel 0 \pl .\]
By normality, we deduce that $\Phi(\N)=B \vee P^{op}\subset B(L^2(\mathcal{M})\ten_{\A}L^2(P))$. Let us denote by $B_P=\Phi(\N)$ the resulting von Neumann algebra and by $e_{B_P}$ the support of $\E=\Phi|_{\N_P}$. Since the Wick words of order $0$ are obviously invariant under $\al_o$ for all $o \in \mathcal O(H)$ and
 \[ \E\al_o(x) \lel \al_o(\E(x)) \lel \E(x) \]
we must have $\al_o(e_{B_P})=e_{B_P}$ for every $o\in \mathcal O(H)$. More precisely, $1-e_{B_P}$ is the projection of the ideal $I=\{x: \E(x^*x)=0\}$ and we certainly have $\al_o(I)=I$.  This implies $\al_o(1-e_{B_P})=1-\al_o(e_{B_P})$.
We deduce that for all $\al$ we have $\al_o(m_{\al}e_{B_P})=m_{\al}e_{B_P}$ and hence, thanks to Lemma \ref{invariance} we know  that $[E_0,m_{\al}(e_{B_P}))]=0$. Now, we fix $\al$ and consider $x_{i,\al}=\F_{H_i}(m_{\al}(e_{B_P}))=m_{\al}F_{H_i}(e_{B_P})$. This means
 \[ x_{i,\al} \lel \sum_{k\le k(\al)} x_k \pl ,\]
where $x_k=P_k(x)$. However, we have a finite basis $\xi_{k,s}$ of $L^2_k(M(H))$ over $B$ made of elements in $W_k(H_i)$ and hence for all $z=\pi(x_{\si'})\theta(y^{op})$ we find
 \[ P_k(z) \lel \sum_{s} \pi(\xi_{k,s}) E_B(\xi_s^*x_{\si'})\theta(y^{op}) \pl .\]
Since $P_k$ is normal we deduce that there are coefficients $a_s\in \pi(B) \vee \theta(P^{op})$ such that
 \[ x_k \lel \sum_{s} \pi(\xi_{k,s})a_{s,k}  \in \N(H_i) \pl .\]
Note here that we have rewritten $m_{\al}$ as normal map, because the maps $T_{k,s}(x)=\pi(\xi_{ks})
\si(\E(\xi_{ks}^*x))$ are normal, thanks to Lemma \ref{normal}. Note also that due to Lemma \ref{normal} $\si(B\vee P^{op})=\pi(B)\vee \theta(P^{op})\subset \N$. On the other hand the projection $P_{H_i}$ onto the range of $\iota_{H_i}$ contains the range of $\iota$ and hence
 \[ [E_0,\iota_{H_i}^*\hat{m}_{\al}(e_{B_P})\iota_{H_i}]
 \lel \iota_{H_i}^*[E_0,\hat{m}_{\al}(e_{B_P})]\iota_{H_i} \lel 0 \pl .\]
Thus we have $[E_0,x_{i,\al}]=0$. Let us consider $\eta=(y\ten_{\A}z)\ten 1$. We deduce that
  \[ x_{i,\al}(\eta) \lel \sum_{k\le k(\al)} \sum_s
  \pi(\xi_{k,s})a_{k,s}(\eta) \pl.
  \]
Moreover, we see that
  \[ E_B(\xi_{s,k}^*x_{i,\al}(\eta)) \lel E_B(\xi_{s,k}^*\xi_{s,k})a_{k,s}(\eta) \pl .\]
We may assume that $f_{k,s}=E_B(\xi_{s,k}^*\xi_{s,k})$ is a projection in $B$ and $a_{k,s}=f_{ks}a_{ks}$. Since the conditional expectation can be calculated using vectors in the Hilbert space, we deduce that
 \[ a_{k,s}(\eta) \lel E_B(\xi_{s,k}^*x_{i,\al}(\eta))
  \lel E_B(\xi_{s,k}^*E_0(x_{i,\al}(\eta)))
  \lel 0 \]
for all $k>0$. Thus only the coefficient for $k=0$ survives and hence $x_{i,\al}\in \si(B\vee P^{op})$. This remains true for  the limit along $\al$, i.e.  $x_i=F_{H_i}(e_{B_P})\in \si(B\vee P^{op})$. Since $\bigcup_i \iota_{H_i}$ is norm dense we find that
 \[ e_{B_P} \lel w^*-\lim_i  F_{H_i}(e_{B_P}) \in \si(B\vee P^{op}) \pl .\]
The restriction of the normal map $\si\circ \E$ to $\si(B\vee P^{op})$ is the identity. This implies
 \[ 1-e_{B_P} \lel \si\circ \E(1-e_{B_P})
 \lel \si \circ\E((e_{B_P}(1-e_{B_P})e_{B_P}) \lel 0 \pl .\]
Thus $e_{B_P}=1$ and $\E$ is indeed  a faithful normal expectation. Now it is easy to conclude the proof of the crucial assertion iii). Indeed, we may assume that $\pi(B)$ and $\theta(P^{op})$ both admit weakly dense separable sub $C^*$-algebras and hence fix a faithful normal state $\phi$ on $B_P$. Then $\psi=\phi\circ \E$ satisfies the Connes's commutativity relation for the modular group $\E(\si_t^{\psi}(x))=\si_t^{\phi}(\E(x))$.
We refer to \cite{HaagerupJungeXu} for the fact that we have a natural embedding of the Haagerup spaces $L^p(B_P)\to L^p(\N)$ given by
 \[ \iota_p (xd_{\phi}^{1/p})\lel x  d_{\psi}^{1/p} \]
for the densities $d_{\phi}\in L^1(B\vee P^{op})$, $d_{\psi}\in L^1(\N_P)$ associated with the states. Moreover, the support of $d_{\psi}$ is $1$. This implies that  $L^2(\N)=\N L^2(B_P)$. By approximation in the $C^*$-algebra generated by $\pi(M)$ and $\theta(P^{op})$ we see that span of elements of the form
 \[ \pi(x_{\si})\theta(y^{op})d_{\psi}^{1/2} \]
are  dense in $L^2(\N)$. However, we have
 \begin{align}
  &tr((\pi(x_{\si})\theta(y^{op})d_{\psi}^{1/2})^* \pi(x_{\nu})\theta(z^{op})d_{\psi}^{1/2})
  \lel tr(\theta(y^{op})^*\pi(x_{\si})^*
  \pi(x_{\nu})\theta(z^{op})d_{\psi})\\
  &=
  tr( \theta(y^{op})^*\theta(z^{op}) \pi(x_{\si})^*\pi(x_{\nu})d_{\psi})   \lel
    \psi(\theta(y^{op})^*\theta(z^{op}) \pi(x_{\si})^*\pi(x_{\nu}))   \nonumber \\
 &=
   \phi(\E(\theta(y^{op})^*\theta(z^{op}) \pi(x_{\si})^*\pi(x_{\nu})))  \lel
  \phi(\theta(y^{op})^*\theta(z^{op}) \E( \pi(x_{\si})^*\pi(x_{\nu})))  \nonumber \\
  &= \phi(\theta(y^{op})^*\theta(z^{op})E_B \pi(x_{\si})^*\pi(x_{\nu})))  \pl . \nonumber
  \end{align}
For the proof of the last equality, we may assume that $x_{\si}$ and $x_{\nu}$ are reduced Wick words. As in Lemma \ref{normal}, we see that
 \begin{align*}
  &\lan \pi(x_{\xi})(y\ten_{\A}z),\pi(x_{\nu})(\tilde{y}\ten_{\A}\tilde{z}\ran 
   \lel \lim_{n} n^{-(|\si|+|\nu|)/2}
  \sum_{(j_k)=\si,(\tilde{j}_{\tilde{k}})=\tilde{\si}} \\
 &\quad
   \tau(\tilde{z}^*E_{\A}(\tilde{y}^*\vec{\pi}_{\tilde{j}}(\tilde{x})^*\vec{\pi}_j(x)y)z)
    \pl \tau(s_{\tilde{j}_m}(\tilde{h}_m)\cdots s_{\tilde{j}_1}(\tilde{h}_1)(s_{j_1}(h_1)\cdots s_{j_m}(h_m)) \\
   &= \delta_{|\si|,|\nu|} \sum_{\gamma\in S_k} q^{{\rm inv}(\si)}
   n^{-|\si|} \sum_{(j_1,...,j_k)}
   \tau((\al_{j_{\gamma(1)},...,j_{\gamma(k)}}(\tilde{x}))^* \al_{j_1,...,j_k}(x) yE_{\A}(z\tilde{z}^*)\tilde{y}^*) \\
   &= \delta_{|\si|,|\nu|}
    \sum_{\gamma\in S_k} q^{{\rm inv}(\si)} \tau(b(x,\tilde{x},\gamma)yE_{\A}(z\tilde{z}^*)\tilde{y}^*)) \pl .
   \end{align*}
The limit $b(x,\tilde{x},\gamma)\in B$ only depends on $x$ and $\tilde{x}$ and the permutation $\gamma$, see Lemma \ref{12}. Placing the summation inside we find indeed $E_B(x_{\nu}^*x_{\si})$. Thus we have shown that $\E|_{\pi(M)}=E_B$. Together with (5.2), we deduce that the spaces $W_k(M)L_2(B_P)$ are mutually orthogonal. 
Finally, we have to discuss the action $\al:\mathcal O(H) \to Aut(\N)$. For an arbitrary *-automorphism $\al$ of $\N$, we may define the action on $L^2(\N)$ via
 \[ \al(xd_{\psi}^{1/2}) \lel \al(x) (d_{\psi} \circ \al^{-1})^{1/2} \pl .\]
It is easy to show that this action is independent of the choice of a normal faithful density $d$ associated with state $\psi$. Here   $d\circ \al^{-1}$ is the density of $\psi\circ \al^{-1}$. Thus we deduce from $\al_o(\theta(y^{op})=\theta(y^{op})$
and the fact that $\psi\circ \al_o=\psi$, that
 \[ \al_o(\pi(x_{\si})\theta(y^{op})d_{\psi}^{1/2})
 \lel \al_o(\pi(x_{\si}))\theta(y^{op})d_{\psi}^{1/2} \pl ,\]
as expected. \qd

\begin{rem}\label{apost} {\rm A posteriori, we deduce that under the assumptions above $F_{H'}$ is faithful for every subspace $H'\subset H$ because $\E=\E F_{H'}$.
}
\end{rem}


\section{The deformation bimodules are weakly contained in $L^2(M) \ten_B L^2(M)$ for sub-exponential dimensions of $D_k(S)$ over $B$}

\subsection{Norm estimates for decomposable maps}

Let $H$ be an $M$-$N$ bimodule over finite von Neumann algebras $M$ and $N$. We will introduce some norms which will enable us to show that the $M-N$ bimodules associated to certain maps $\Phi:M\to L^1(N)=N^{op}_*$ are weakly contained in $H$. To be more precise define
 \[ \|\Phi\|_H \lel \inf \{\sum_j\|\xi_j\|\|\eta_j\| : \tau(\Phi(x)y) \lel \sum_j \langle (x\ten y^{op})\xi_j, \eta_j \rangle \} \pl .\]
The infimum is taken over elements $\xi_j,\eta_j \in H$.

\begin{lemma}\label{ne11} Let $K$ be an $M$-$N$ bimodule such that for a total set of vectors $\xi\in K$ the map $\Phi_{\xi}:M \to L^1(N)$ defined by
 \[\tau(\Phi_{\xi}(x)(y)) = \lan (x \ten y^{op})\xi, \xi \ran = \lan x\xi y, \xi \ran\]
 satisfies $\|
 \Phi_{\xi}\|_H<\infty$. Then $K$ is weakly contained in $H$.
\end{lemma}

\begin{proof} Let us recall that $K \prec H$ if and only if we have the relation between the kernels
 \[ \ker(\pi_H) \subset \ker(\pi_K) \pl ,\]
where  $\pi_H: M\ten_{\rm bin}N^{op}\to B(H)$, respectively $\pi_K: M\ten_{\rm bin}N^{op}\to B(K)$ are the canonical representations. Let $z=\lim z_j$ be a limit of norm one
elementary tensors which converges to an element $z\in \ker(\pi_H)$ with respect to the $\max$ norm. Let $\xi \in K$ such that $\|\Phi_{\xi}\|_H < \infty$. This means we may assume that
   \[ \tau(\Phi_{\xi}(x)y) \lel \sum_l \al_l \lan \xi_l,x\eta_l y \ran \quad , \quad  \|\xi_l\| \|\eta_l\|\le 1 \]
and $\sum_l |\al_l|$ is finite. Using $\|z_j\|_{\rm bin}\le 1$ and uniform convergence, we may interchange limits and deduce
 \begin{align*}
  \lan z\xi,z\xi\ran &=
 \lim_j  \lan \xi,z_j^*z_j\xi\ran 
 \lel \sum_l \al_l \lim_j \lan \xi_l,z_j^*z_j\eta_l\ran \lel
 \sum_l \al_l \lan z\xi_l,z\eta_l\ran \lel 0 \pl.
 \end{align*}
Thus for any linear combination $\xi=\sum_k \xi_k$ of elements such that the $\Phi_{\xi_k}$'s have finite $H$ norm, we still have $\pi_K(z)\xi=0$. By density this holds for all $\xi\in K$.
\qd

As an illustration for the norm estimates let us prove the following result.
\begin{lemma}\label{basis1} Let $H_B=L_2(M)\ten_B L_2(M)$, and assume that $L^2_k(M)$ has dimension $d_k$ over $B$. Let $P_k:L^2(M) \to L^2_k(M)$ be the orthogonal projection. Then
 \[ \|P_k\|_{H_B}\le d_k  \pl.\]
\end{lemma}

\begin{proof} We recall that
 \begin{equation}\label{modform}
 \langle x\ten y^{op}(c\ten d),a\ten b\rangle
\lel   \tau(b^*E_B(a^*xc)dy)\lel \tau(E_B(a^*xc)E_B(dyb^*)) \pl .
 \end{equation}
Assuming that $\xi_j$ is a basis with $E_B(\xi_j^*\xi_i)=\delta_{ij}e_i$, $e_i$ a projection, we see that
 \[ \tau(yP_k(x))
 \lel \sum_j  \tau(y\xi_jE_B(\xi_j^*x))
 \lel \sum_j \langle x\ten y^{op}(1\ten 1),\xi_j\ten \xi_j^*\rangle \pl.\]
Since $\langle \xi_j\ten \xi_j^*,\xi_j\ten \xi_j^*\rangle=\tau(E(\xi_j^*\xi_j)E(\xi_j^*\xi_j))\le \tau(e_j)$, we deduce the assertion.
\qd

\subsection{Configurations}

Our main goal here is to analyze the operators
$\Phi_{\xi,\eta}:M \to L_1(M)$ given by $\Phi_{\xi,\eta}(x) \lel E_M(\xi x \eta)$, where $\xi, \eta$ are elements in $\Gamma_q(B,A\otimes (H\oplus H))$. We will start with monomials
  \[ \xi \lel s(x_1,h_1)\cdots s(x_m,h_m) \pl ,\pl \eta \lel s(x_{m'}',h'_{m'})\cdots s(x_1',h'_1) \]
where $h_i, h'_{i'}\in H\times \{0\}\cup \{0\}\times H$.  Although our goal is to obtain estimates for arbitrary $x$, we will first assume that $x=\zeta$ is a reduced Wick word from $M$ and only contains singletons from $H\times\{0\}$. By considering the moment formula we can reorganize the trace using configurations
 \[ \tau(\zeta'\xi \zeta \eta) \lel \sum_{\al \mbox{ \scriptsize configuration}} \tau(\zeta'  \Phi_\al(\zeta))      \]
whenever $\zeta'$ is another reduced Wick. Here a
configuration
$\al=(\si_{0\times H},\si_{H\times 0},I_{\xi,\zeta},I_{\zeta,\eta})$ is given by
 \begin{enumerate}
  \item[i)] A pair partition $\si_{0 \times H}$ of $\{1,...,m\}\stackrel{\cdot}{\cup}\{m',...,1\}$ so that all the pairs $\{l,r\}$ have indices in $0\times H$;
 \item[ii)] A pair partition $\si_{H \times 0}$ of $\{1,...,m\}\stackrel{\cdot}{\cup}\{m',...,1\}$ so that all the pairs $\{l,r\}$ have indices from $H \times 0$;
 \item[iii)] Subsets  $I_{\xi,\zeta}\subset \{1,...,m\}$, $I_{\zeta,\eta}\subset \{m',...,1\}$ disjoint from the support $\cup\si_{0 \times H} \cup \cup \si_{H \times 0}$ of the partitions above.
 \end{enumerate}
Indeed, using the moment formula for $\tau(\zeta'\xi \zeta \eta)$ we know that we have to take the sum over all pair partitions of length $m+m'+k+k'$, $k=|\zeta|$, $k'=|\zeta'|$. Every such pair partition has to respect the pairs of $0\times H$ and that defines our $\si_{0\times H}$. Some pairs can combine elements from $\xi$ and $\eta$ with coefficients in $H\times 0$. This defines $\si_{H\times 0}$. Some partitions connect $\xi$ and $\zeta$ and some $\zeta$ with $\eta$. The left hand sides of the pairs between $\xi$ and $\eta$ define the set $I_{\xi,\zeta}$ and the right hand sides of the pairs from $\zeta$, $\eta$ define $I_{\zeta,\eta}$. All the remaining pairs will connect $\zeta'$ and $\zeta$.
Since $\zeta$ and $\zeta'$ are themselves Wick words, there are no pairs connecting elements from $\zeta$ ($\zeta'$) with itself. We see that indeed, the sum over all partitions can be regrouped into first summing over all configuration (which only depend on $\xi$ and $\eta$), and then sum over all partitions supported by these configurations. Let us note that once a configuration $\al$ is known we can determine exactly how many crossings will be produced by pairs in $0\times H$. Indeed, we know that $|I_{\xi,\zeta}|+|I_{\zeta,\eta}|$ many singletons will be removed from $\zeta$. According to the position of the left legs in $\si_{0\times H}$ some extra crossing will be produced from the set $I_{\xi,\zeta}$. The same applies for $I_{\xi,\zeta}$. Here is an example
 \[ \begin{array}{ccccc|cccccc|ccccc}
     a_1&a_2&b_1&a_3&b_3& c_1&c_2&a_3&d_1&c_3&c_4&d_2&b_3&d_1&a_1&b_1
     \end{array} \]
Here $\si_{0\times H}$ are given by the positions of $b_1$ and $b_3$. The set $I_{\xi,\zeta}$ is given by the position of $a_3$ and $\si_{H\times 0}$ is given by the positions of $a_1$.  The $b$'s are responsible for $8+1+1+1$ crossings, $8$ crossing with $c$'s,
one crossing among themselves, one crossings coming from $a$'s and $b$'s, one crossing from the $b$ and $d$'s.  Thus  $k(\al)=
2\times (6-2)+1+2$.

In our next step we replace the monomials $\xi$ and $\eta$ by Wick words. This means we only have to sum over those configurations such that $\si_{0\times H}$ and $\si_{H\times 0}$ connect $\xi$ and $\eta$ and no pairs $\xi$ and $\eta$ with itself. In addition the reduction procedure produces scalars and new operator valued expression $\al_{j_1,..,j_l}(\beta)$ with $\beta\in D_k(S)$. We have proved the following simple combinatorial fact:

\begin{lemma}  Let $\xi$ and $\eta$ be Wick words obtained by reduction and $\zeta\in M$ be a Wick word of length $k=|\zeta|$. For a fixed configuration $\al$ there is a number
$k(\al)$ such that for all $-1\le q\le 1$
 \[  \Phi_{\al}(\zeta) \lel q^{k(\al)}\tilde{\zeta} \pl ,\]
where $\tilde{\zeta}$ is a linear combination of reduces words with smaller length $k-|I_{\xi,\zeta}|-|I_{\zeta,\eta}|$. Moreover, if $k\gl m+m'$ is the length of $\zeta$, and $L$ is the cardinality of $\si_{0\times H}$, then
 \[ k(\al)\gl  (k-m-m')L \pl .\]
\end{lemma}

We will give more precise information about $\tilde{\zeta}$ in the next paragraph.

\subsection{Generalized $Q$-gaussians}
As a tool we will use a slight generalization of the von Neumann algebra $\Gamma_q(B,A\ten H)$. This generalization is based on matrix models of the ordinary $q$-gaussian von Neumann algebras. This approach was invented by Speicher (\cite{SpeNCL, SpeGSM}) and has been applied in many situations, see e.g. \cite{Biane, JungeAW, JPPPR, JungeZeng, NouI, NouII}.
Let $Br:H\to \bigcap_p L^p(\Om,\Si,\mu)$ the standard brownian motion so that $Br(h)$ is a normal random variable and and $(Br(h),Br(h'))=(h,h')$. The $\si$-algebra is chosen minimal. This construction is well-known as the gaussian measure space construction. Given a selfadjoint matrix $\eps_{ij}$ with values $\{-1,1\}$ there are symmetries $v_j\in M_{2^n}(\cz)$ such that
 \[  v_iv_j\lel \eps_{ij}v_jv_i \pl. \]
Speicher's important idea is to choose the matrix $\eps_{ij}$ independently at random for all pairs. We will work with double indices  $\eps_{(j,t),(k,s)}$, which  are independent as functions of the pairs $\{(j,t),(k,s)\}$  whenever $t\neq s$ or $j\neq k$ and satisfy
 \[ P( \eps_{(jt),(ks)}=1)  \lel \frac{1-Q_{s,t}}{2} \pl \]
as along as $(j,t)\neq (ks)$ for a given matrix $Q_{s,t}$. This allows us to construct matrix models
\[ u(t,h) \lel (\frac{1}{\sqrt{n}}\sum_{j=1}^n v_{j,t}\ten g_j(h))_n \in \prod_{n,\om}(M_{2^n}(\cz) \ten L^{\infty}(\Om))_n \]
which  satisfy
 \[ \tau(u(t_1,h_1)\cdots u(t_m,h_m))
 \lel \sum_{\si\in P_2(m)}  \prod_{\{a,b\}\in \si, \{c,d\}\in \si, a<c<b<d} Q_{t_at_c}       \prod_{\{a,b\}\in \si} \lan h_a,h_b\ran \pl .\]
 In other words the constant term $q^{{\rm inv}(\si)}$ is replaced by the product of the crossing inversions weighted according to $Q$.  Indeed, by independence
  \[
  \prod_{\{a,b\}\in \si, \{c,d\}\in \si, a<c<b<d} Q_{t_at_c} \lel \ez \tau(v_{j_1,t_1}\cdots v_{j_m,t_m}) \]
for $(j_1,...,j_m)\le \si$. In particular for a fixed $t$ and $\|h\|=1$ the random variable $u(t,h)$ is just an ordinary $q$ gaussian. This central limit theorem is well-known and goes back to \cite{SpeNCL, SpeGSM}, see also \cite{JungeAW}, \cite{JungeZeng}.

We may easily generalize this to the $A$-valued situation by considering a sequence of symmetric independent copies $(\pi_j,B,A,D)$ and defining
 \[ u(t,h,a) \lel
 (\frac{1}{\sqrt{n}}\sum_{j=1}^n v_{j,t}\ten g_j(h)\ten \pi_j(a))_n \in \prod_{n,\om}(M_{2^n}(\cz) \ten L^{\infty}(\Om) \bar{\ten} D)_n \pl .\]
For a subset $1 \in S=S^* \subset A$, we denote the von Neumann algebra generated by the elements $u(t,h,a), t\in Q, h \in H, a \in S$ by $\Gamma_Q^0(B,S \ten H)$. Then define the von Neumann algebra $\Gamma_Q(B,S \ten H)$ by the same procedure as in Def. 3.4. A look at the moment formula allows us to state the following fact.

\begin{lemma}\label{Q} Let $T_0\subset T$ be a non-empty subset such that $Q_{st}=q$ for all $s,t\in T_0$. Then $\Gamma_q(B,S \ten H)$ embeds into $\Gamma_Q(B,S\ten H)$ in a trace preserving way. \end{lemma}

\begin{rem}{\rm As observed in \cite{JungeZeng} the reduction procedure still works in the generalized $Q$-gaussian setting.}
\end{rem}

Let us return to a configuration $\al$ as in 6.2. above. We replace the Wick words $W_q(\vec{a},\vec{h})$ and $W_q(\vec{a'},\vec{h'})$ by new Wick words $W_Q(\vec{a},\vec{h})$ $W_Q(\vec{a'},\vec{h'})$ as follows. For a configuration $\al$ with a partition $\si_{0\times H}$ of the indices labeled with $0\times H$, we define a new matrix
  \[  Q_{st}(\tilde{q})\lel \begin{cases} \tilde{q} &
  \mbox{if
  $h_s$ and $k_s$ are both in $ 0\times H$} \\
  q& \mbox{else.}
  \end{cases} \]
Note that the matrix only depends on the first component  $\si_{0\times H}$ of a configuration. For every pair $p=\{l,r\}\in \si_{0\times H}$ we introduce a label $e_p$ and replace $h_l$ and $h'_r$ by $h_l\ten e_p$ and $h'_r\ten e_p$ to avoid over-counting. We denote by $H_s$, $H'_t$ the modified vectors. Starting from
  \[ \xi_{Q(\tilde{q})} \lel s_{Q(\tilde{q})}(H_1,a_1)\cdots s_{Q(\tilde{q})}(H_m) \pl ,\pl \eta_{Q(\tilde{q})} \lel s_{Q(\tilde{q})}(H'_{m'},a'_{m'})\cdots s_{Q(\tilde{q})}(H'_1,a'_1) \]
we apply the same reduction procedure (eliminating all the pairs from the non-reduced words $X_{\si}(\vec{h},\vec{a})$)  for the $W_q$'s and obtain the reduced Wick words $W_{Q(\tilde{q})}(\vec{H},\vec{a})$, $W_{Q(\tilde{q})}(\vec{H'},\vec{a'})$.

\begin{lemma} Fix $\si_{0\times H}$. The function
 \[ F(\tilde{q})  \lel \sum_{\al,\al_1=\si_{0 \times H}} E_M(W_{Q(\tilde{q})}(\vec{H},\vec{a})
 \zeta W_{Q(\tilde{q})}(\vec{H'},\vec{a'}))
  \]
is a polynomial in $\tilde{q}$ with lowest degree at least  $(|\zeta|-m+m')L$ and largest degree at most $(|\zeta|+m+m')L$.
\end{lemma}

\begin{proof} Let $\al$ be a configuration  which contains $\si_{0\times H}$. Comparing the terms in the  moment formula for
 \[ \tau(\zeta' W_q(\vec{h},\vec{a})
 \zeta W_q(\vec{h'},\vec{a'})
 ) \quad \mbox{and} \quad
 \tau(\zeta' W_{Q(\tilde{q})}(\vec{h},\vec{a})
\zeta W_{Q(\tilde{q})}(\vec{h'},\vec{a'}))
  \]
we see that they differ by the factor $(\frac{\tilde{q}}{q})^{k(\al)}$ number of pairs. Note however, that $k(\al)$ only depends on $\al$. This implies the assertion.
\qd

\subsection{Weak containment} We need a simple fact about polynomials:

\begin{lemma}\label{interval} Let $[a,b]$ be an interval,  $\mathcal{P}_d(a,b)$ the set of polynomials of degree $d$,and $a<t_0<t_1<\cdots<t_d<b$ distinct points. Then the map  $\Phi:P_d(a,b)\to \cz^{d+1}$, $\phi(f)=f(t_j)$ is injective. Moreover, there exists a matrix $a_{i,j}$ such that for every polynomial
 \[ p(t) \lel \sum_{0\le k\le d} \al_kt^k \]
of degree $\le d$  we have
 \[ \al_k \lel \sum_{j} a_{k,j} f(t_j) \pl .\]
\end{lemma}

\begin{proof} For $0\le j\le d$ we define  the polynomial $p_j(t)=(\prod_{i\neq j} (t_j-t_i))^{-1}\prod_{i\neq j} (t-t_i)$ which has degree $d$. Then we see that $p_j(t_j)=1$ and $p_j(t_i)=0$ for $i\neq j$. In particular, the polynomials $(p_j)_{0\le j\le d}$ are linearly independent and hence $P_d(a,b)={\rm span}\{p_j|0\le j\le d\}$. This implies
 \[ p(t) \lel \sum_{0\le j\le d} p(t_j)p_j(t) \]
and in particular $\Phi$ is injective. Since moreover, the monomials are linearly independent in $C_{\infty}(a,b)$, we see that the linear map $\Psi(\al_0,...,\al_d)=\Phi(\sum_{k}\al_k t_j^k)$ is invertible and can be represented by the matrix $C_{j,k}=t_j^k$, the well known Vandermonde matrix.
Then $A=C^{-1}$ does the job.\qd

From now on we fix $\si=\si_{0\times H}$,
Wick words $\xi= W_q(\vec{H},\vec{a})$, $\eta=W_q(\vec{H'},\vec{a'})$ which are obtained after reduction from  possible longer terms $s_q(h_1,a_1)\cdots s_q(h_m,a_m)$ and $s_{q}(h'_{m'},a'_{m'})\cdots s_{q}(h'_1,a'_1)$. This allows us to define
 \[ F_{\si}(t) \lel  E_M(W_{Q(t)}(\vec{H},\vec{a})\zeta W_{Q(t)}(\vec{H},\vec{a})) \]
As in section 6.2. we assume that at least $L$ labels of $\xi$ and  $\eta$ are of the form $(0,h_i)$.

\begin{cor}\label{polyfor} Fix $m,m'$ and $L$. Then there exists a degree  $D=D(m,m',L)$ such that for $q\in [a,b]$ and $a\le t_1<\cdots t_{D}\le b<1$  there are coefficients $\gamma_l$ such that
 \[ E_M(\xi \zeta \eta)
  \lel  \sum_{\si}
  \sum_l (\frac{q}{t_l})^{(k-m-m')L}\gamma_l F_{\si}(t_l)  \pl \]
holds for $k=|\zeta|\gl 2(m+m')$. Moreover, for some possibly different coefficients $\tilde{\gamma_l}$
 \[ E_M(\xi \zeta \eta) \lel \sum_{\si} \sum_l \tilde{\gamma_l} F_{\si}(t_l)  \]
holds for $|\zeta|\le 2(m+m')$.
 \end{cor}

\begin{proof} We fix $\si$ and $k\gl m+m'$. Let  $[a,b]\subset (-1,1)$ be an interval and $a=q$.
 The $t_i$'s are all chosen bigger than $a$. We define the polynomial  $p_k(t)=t^{-(k-m-m')L}F(t)$ which has degree at most $(k+m'+m-(k-m-m'))L\le (2m+2m')L$ and hence
 \[ p_k(t) \lel \sum_{0\le j\le (2m+2m')L} a_jt^j \quad \mbox{and}\quad    a_j \lel \sum_{i} c_{ij} p_k(t_i) \pl \]
holds for mutually different points $a\le t_1,...,t_d \le b$ where $d\le (2m+2m')L+1$ are independent of $k$. Hence we get
 \begin{align*}
 & F_{\si}(q) \lel  q^{(k-m-m')L}p_k(q)
\lel q^{(k-m-m')L}  \sum_{j,i} c_{ij}q^j
 p_k(t_i) \\
 &= q^{(k-m-m')L} \sum_{j,i} c_{ij}q^j
 t_i^{-(k-m-m')L} F_{\si}(t_i)  \lel
  \sum_i (\sum_j c_{ij}q^j) (\frac{q}{t_i})^{(k-m-m')L} F_{\si}(t_i)
 \pl .
   \end{align*}
This defines the coefficients $\gamma_i$.
For $k\le 2(m+m')$ we work directly with the polynomial $F(t)$ of degree at most $2(m+m')L$.
\qd

Let $M=\Gamma_q(B,S\ten H)$, $\tilde M=\Gamma_q(B,S\ten(H \oplus H))$. Define the $M-M$ bimodule $\F_m \subset L^2(\tilde M)$ as the $\|\cdot\|_2$-closed linear span of the reduced Wick words $W_{\si}(x_1,\ldots, x_t,h_1, \ldots, h_N), N \geq 1$ such that $h_i \in H \times \{0\} \cup \{0\} \times H$ for all $i$ and at least $m$ of the vectors $h_i$ belong to $\{0\} \times H$. This bimodule will play a crucial role in our deformation-rigidity arguments in the next section.

\begin{theorem} Let $M=\Gamma_q(B,S \ten H)$ and let $C>0, d>0$ be two constants such that the dimension of $L^2_k(M)$ over $B$ is smaller than $C d^k$ for all $k$. Let $|q|<1$. Then there exists an $L_0\in \nz$ and a $B$-$M$ bimodule $\K$ such that $\mathcal{F}_l$ is weakly contained in $L_2(M)\ten_B \K$ for all $l\gl L_0$.
\end{theorem}

\begin{proof} Let us recall that
 \[  \langle \zeta\ten (\zeta')^{op}(a\ten_B b),\al\ten_B \beta\rangle
 \lel \tau(\beta^*E_B(\al^*\zeta a)b\zeta')
 \lel \tau(E_B(\al^*\zeta a)E_B(b\zeta'\beta^*)) \pl .\] Now we may assume that $(\xi_i)_{i\in I_k}$ is a basis of dimension $d_k$ over $B$ so that
  \[ P_k(\zeta) \lel \sum_{i\in I_k} \xi_i E_B(\xi_i^*\zeta) \pl  \quad \mbox{and} \quad E_B(\xi_i^*\xi) \le 1\pl. \]
This implies
  \begin{align*}
   & \tau(\zeta'\xi \zeta \eta)
    \lel \sum_{i\in I_k} \tau(\zeta' \xi \xi_i E_B(\xi_i^*\zeta)\eta)  \lel
       \sum_{i\in I_k} \tau(\xi_i E_B(\xi_i^*\zeta)\eta \zeta') \\
 &=
   \sum_{i\in I_k} \langle \zeta\ten (\zeta')^{op} (1 \ten_B \eta)  ,\xi_i \ten_B (\xi \xi_i)^*\rangle \pl .
   \end{align*}
Let $q_0<1$ so that $q/q_0<1$. Then we define the $B-M$ bimodule 
\[\K =\bigoplus_{q/q_0\le t <1} L^2(\Gamma_{Q(t)}(B,S\ten H))\] 
with the natural left and right actions. For fixed  $\xi$, $\eta$ we choose $a=\pm q$ and $|q|/q_0\le t_0<\cdots t_D <b $ for some $b<1$. This allows us to define $W_{Q(t_i)}(\vec{h},\vec{a})$ and $W_{Q(t_i)}(\vec{h'},\vec{a'})$ in $\K$. With the help of Corollary \ref{polyfor} we deduce that  the map
 $\Phi^+(\zeta)=\sum_{k\gl 2(m+m')} E_M(\xi P_k(\zeta)\eta)$  satisfies
 \begin{align*}
 &\|\Phi^+\|_{L_2(M)\ten_B\K} \kl \\
  &\sum_{\si} \sum_l |\gamma_l| \sum_{k\gl 2(m+m')}
   q_0^{(k-m-m')L}
  \sum_{i\in I_k}  \|1\ten_B W_{Q(t_l)}(\vec{H'},\vec{a'})\| \|\xi_i\ten_B(
  (W_{Q(t_l)}(\vec{H},\vec{a})
   \xi_i)^*\|   \pl .
  \end{align*}
Now we note that
 \[   \|1\ten_B W_{Q(t_l)}(\vec{H'},\vec{a'})
 \| \lel \|W_{Q(t_l)}(\vec{H'},\vec{a'})\|_{L_2(\Gamma_{Q(t_l)})}
 \kl c(t_l) \]
and
  \begin{align*}
&  \|\xi_i \ten_B (W_{Q(t_l)}(\vec{H},\vec{a})\xi_i)^*\|
 \lel  \tau(W_{Q(t_l)}(\vec{H},\vec{a})\xi_i E_B(\xi_i^*\xi_i) (W_{Q(t_l)}(\vec{H},\vec{a}) \xi_i)^*) \\
 & \kl \tau(\xi_i\xi_i^*W_{Q(t_l)}(\vec{H},\vec{a})^*W_{Q(t_l)}(\vec{H},\vec{a}))
 \kl \|W_{Q(t_l)}(\vec{H},\vec{a})
 \|_{\Gamma_{Q(t_l)}}^2  \kl c(t_l) \pl .
 \end{align*}
Thus it suffices to know that $\sum_k  q_0^{(k-m-m')L} Cd^k$ is finite. Note here that $m$ and $m'$ depend on the Wick word and that we may assume $l\gl L_0$.  Thus $q_0^{L_0}d<1$ and $b<1$  is enough to achieve summability . Using the second part of Corollary \eqref{polyfor} we also have summability for $k\le 2(m+m')$.  Lemma \ref{ne11} then yields the assertion. \qd

\begin{cor} Let $M=\Gamma_q(B,S\ten H)$ and assume that $H$ is finite dimensional and $dim_B(D_k(S)) \leq Cd^k$ for some constants $C, d>0$. Then there exists an $B-M$ bimodule $\K$ such that for $m\geq 1$ large enough we have $\F_m \prec L^2(M) \ten_B \K$. In particular, for $m$ large enough, $\F_m$ is weakly contained into $L^2(M) \ten_B L^2(M)$.
\end{cor}

\section{The proof of the main theorem and its applications} 
We first need some preliminaries. Throughout this section we use the notations $M=M(H)=\Gamma_q(B,S\ten H)$, $\tilde M=\Gamma_q(B,S\ten (H \oplus H))=M(H\oplus H)$. Let  
\[\mathcal M=(D \bar{\ten} \Gamma_q(\ell^2 \ten H))\vee M \subset (D \bar{\ten} \Gamma_q(\ell^2 \ten H))^{\om}.\]
As in Section 5, let 
\[\mathcal H \subset ((L^2(\mathcal M)\ten_{\mathcal A} L^2(P))\ten \mathcal F_q(\ell^2 \ten H))^{\om}\]
be the norm closed linear span of the sequences 
\[(n^{-\frac{m}{2}}\sum_{(j_1,\ldots,j_m)=\si} (\pi_{j_1}(x_1)\cdots \pi_{j_m}(x_m)y \ten_{\mathcal A} z) \ten s_{j_1}(h_1)\cdots s_{j_m}(h_m)),\]
for $m\geq 1, \si \in P_{1,2}(m), x_i \in BSB, h_i \in H, y\in M, z \in P$.
Take the representations 
\[\pi:M \to B(\mathcal H), \theta:P^{op}\to B(\mathcal H)\] 
introduced in Section 5 and define $\mathcal N=\pi(M) \vee \theta(P^{op})\subset B(\mathcal H)$. As seen is Section 5, we choose a normal faithful state $\phi$ on $B_P=\pi(B)\vee \theta(P^{op})\subset B(\mathcal H)$ and then define a normal faithful state $\psi$ on $\mathcal N$ by $\psi=\phi \circ E_{B_P}$, where $E_{B_P}:\mathcal N \to B_P$ is the normal faithful conditional expectation. Let $d_{\psi} \in L^1(\mathcal N)$ be the density of $\psi$ and $\xi_{0}=d_{\psi}^{\frac{1}{2}}$. Then $L^2(\mathcal N)$ is the norm closed span of the elements $\pi(x_{\si})\theta(y^{op})\xi_0$, for $x_{\si} \in M$ a Wick word and $ y \in P$. Let $W_k(M)$ be the linear span of the Wick words of degree $k$ in $M$ and $L^2(B_P)=L^2(B_P,\phi)$ be the standard form for $B_P \subset B(\mathcal H)$. Then $\mathcal N$ is standardly represented on
\[L^2(\mathcal N)=\overline{\bigoplus_{k\geq 0} W_k(M)L^2(B_P)}\]
by the formulas
\[\pi(x_{\si})\theta(y^{op})(\pi(x_{\nu})\theta(z^{op})\xi_0) = \pi(x_{\si}x_{\nu})\theta((zy)^{op})\xi_0, x_{\si}, x_{\nu}\in M, y,z \in P.\]
The conjugation $\mathcal J:L^2(\mathcal N)\to L^2(\mathcal N)$ associated to the standard representation of $\mathcal N$ is given by 
\[\mathcal J(\pi(x_{\si})\theta(y^{op})\xi_0)=\si_{-\frac{i}{2}}^{\psi}(\pi(x_{\si}^*)\theta(\bar{y}))\xi_0, x_{\si} \in M, y \in P,\]
where $\si_t^{\psi}$ is the modular group on $\mathcal N$ associated to $\psi$.
We will also consider $\tilde{\mathcal N}=\mathcal{N}(\tilde H)$ constructed in the same way as $\mathcal N$ by using $\tilde H=H \oplus H$ instead of $H$. Thus take 
\[\tilde{\mathcal H} \subset ((L^2(\mathcal M)\ten_{\mathcal A} L^2(P))\ten \mathcal F_q(\ell^2 \ten \tilde H))^{\om}\]
to be the norm closed linear span of the sequences
\[(n^{-\frac{m}{2}}\sum_{(j_1,\ldots,j_m)=\si}(\pi_{j_1}(x_1)\cdots \pi_{j_m}(x_m)y \ten_{\mathcal A} z)\ten s_{j_1}(\tilde h_1)\cdots s_{j_m}(\tilde h_m)),\]
for $m\geq 1, \si \in P_{1,2}(m), x_i \in BSB, \tilde h_i \in \tilde H$. Exactly as in Section 5, define the *-representations 
\[\pi:\tilde M \to B(\tilde{\mathcal H}), \theta:P^{op}\to B(\tilde{\mathcal H})\]
and then define $\tilde{\mathcal N}=\pi(\tilde M)\vee \theta(P^{op})$. Then $\tilde{\mathcal N}$ is standardly represented on
\[L^2(\tilde{\mathcal N})=\overline{\bigoplus_{k\geq 0} W_k(\tilde M)L^2(B_P)},\]
and the associated conjugation $\tilde{\mathcal J}:L^2(\tilde{\mathcal N})\to L^2(\tilde{\mathcal N})$ is given by the formula
\[\tilde{\mathcal J}(\pi(x_{\si})\theta(y^{op})\xi_0)=\si_{-\frac{i}{2}}^{\psi}(\pi(x_{\si}^*)\theta(\bar{y}))\xi_0, x_{\si} \in \tilde M, y \in P,\]
where $\si_t^{\psi}$ is the modular automorphisms group on $\tilde{\mathcal N}$ associated to $\psi$. For every angle $t$ define the unitary $V_{t}$ on $L^2(\tilde{\mathcal N})$ by
\[\pi(x_{\si}(x_1,\tilde h_1,\ldots,x_m,\tilde h_m))\theta(y^{op})\xi_0 \mapsto \pi(x_{\si}(x_1,o_{t}(\tilde h_1),\ldots,x_m,o_{t}(\tilde h_m)))\theta(y^{op})\xi_0. \]
Then the one parameter group of *-automorphisms Ad$(V_{t})$ of $B(L^2(\tilde{\mathcal N}))$ restricts to a group $\al_{t}$ of *-automorphisms of $\tilde{\mathcal N}$, acting according to the formula
\[\al_{t}(\pi(x_{\si}(x_1,\tilde h_1,\ldots,x_m,\tilde h_m)\theta(y^{op}))=\pi(x_{\si}(x_1,o_{t}(\tilde h_1),\ldots,x_m,o_{t}(\tilde h_m)))\theta(y^{op}).\] 
When further restricted to $\tilde M =\Gamma_q(B,S\ten \tilde H)$ this group of *-automorphisms coincides with the one introduced in Theorem 3.16 and we have the following identity
\[T_t(x)=E_M(\al_{s}(x)), x \in M, 0 \leq s < \frac{\pi}{2},\]
where $T_t$ is the heat semigroup introduced in Theorem 3.16 and $t=-$ln$(\cos(s))$. We finally introduce the bimodules needed in the deformation argument. To do this, recall that $\tilde M=M(H\oplus H)$ is the generalized $q$-gaussian algebra generated by $B$, $s_q(a,h,0)$ and $s_q(a,0,h)$, where $a \in S$ runs through the generating set and $h\in H$ are unit vectors. Let $F \subset H$ be an orthonormal basis.
Then we define an $M-M$ bimodule $\mathcal{F}_{=m}\subset L^2(\tilde M)$ by
 \[ \mathcal{F}_{=m}
 \lel \overline{{\rm span}}^{\|\cdot\|_2}\{ W_{\si}(k_1,...,k_N,a_1,....,a_{N'}) | k_i\in F\times \{0\}\cup \{0\}\times F, 
 \#\{i| k_i\in \{0\}\times F\}=m \} \pl . \]
Note that we use reduced Wick words. This means $N=|\si_s|$ and the vectors $(k_1,...,k_N)$ are the ones obtained after contracting the pairs. Here $\si\in P_{1,2}(N')$ and $a_1,..,a_{N'}$ are the original coefficients from $S$. One can see that $\mathcal{F}_{=m}$ is exactly the eigenspace of vectors $\xi \in L^2(\tilde M)$ such that  $E_{M(0 \oplus H)}(\al_{t}(\xi))=e^{-tm}\xi$ for all (some) $t>0$. Likewise we define the $M-M$ bimodule $\F_{=m}^P \subset L^2(\tilde{\mathcal N})$ as the $\|\cdot\|_2$-closed span of the elements
\[\pi(W_{\si}(x_1,h_1,\ldots,x_m,h_m))\theta(y^{op})\xi_0, x_i \in BSB, h_i \in F\times \{0\}\cup \{0\} \times F,\]
such that exactly $m$ of the vectors $h_i$ belong to $\{0\}\times F$. It's easy to see that $\mathcal F_{=m}^P$ can be described by
\[ \F_{=m}^P \lel \{\xi\in L_2(\tilde{\mathcal N})| E_{\mathcal N(0 \oplus H)}(\al_{t}(\xi))=e^{-tm}\xi, \forall t>0  \} \pl .\] 
Finally, we set 
 \[ \mathcal{F}_{m} \lel \bigoplus_{m'\gl m} \mathcal{F}_{=m'}\subset L^2(\tilde M) \quad, \quad 
 \mathcal{F}_{m}^P \lel \bigoplus_{m'\gl m} \mathcal{F}_{=m'}^P \subset L^2(\tilde{\mathcal N})\pl .\] 
Let's remark that we have the following transversality property, whose proof is virtually the same as that of Prop. 5.1. in \cite{Avsec}.
\begin{lemma} There exists a constant $C=C(m)>0$ such that for $0<t<2^{-m-1}$ we have
\[\|V_{t^{m+1}}(\xi)-\xi\| \leq C \|e^{\perp}V_t (\xi)\| \quad \mbox{for all} \quad \xi \in \bigoplus_{k \geq m+1}L^2_k(\N)\subset L^2(\tilde{\mathcal N}).\]
\end{lemma}

\begin{theorem} Let $M=\Gamma_q(B,S\ten H)$ associated to a sequence of symmetric independent copies $(\pi_j,B,A,D)$ and assume that the dimension of $D_k(S)$ over $B$ is finite for every $k$ and that $H$ is finite dimensional. Let $\mathcal A \subset M$ be a von Neumann subalgebra which is amenable relative to $B$ and denote $P=\mathcal N_M(\mathcal A)''$. Let $m \geq 1$ be fixed. Then at least one of the following statements holds:
\begin{enumerate}
\item The $M-M$ bimodule $\mathcal F_m$ is left $P$-amenable;
\item there exist $t, \delta >0$ such that $\inf_{a \in \mathcal{U}(\mathcal A)}\|T_t(a)\|_2 \geq \delta$,
\end{enumerate}

\end{theorem}
\begin{proof}
The approximately invariant states $\om_n \in \mathcal N_*$ constructed in Thm. 5.1 are implemented by unit vectors $\xi_n \in L^2(\mathcal N)\subset L^2(\tilde{\mathcal N})$. Using the Powers-Stormer inequalities we see that the vectors $\xi_n$ have the following properties
\begin{enumerate}
\item $\lan \pi(x)\xi_n,\xi_n \ran \to \tau(x), x \in M$;
\item $\|\pi(a)\theta(\bar{a})\xi_n -\xi_n\| \to 0, a \in \mathcal{U}(\mathcal A)$;
\item $\|\pi(u)\theta(\bar{u})\mathcal J \pi(u)\theta(\bar{u})\mathcal J\xi_n-\xi_n\|\to 0, u \in \mathcal N_M(\mathcal A)$.
\end{enumerate}
Let $e^{\perp}:L^2(\tilde{\mathcal N})\to \mathcal{F}_m^P$ be the orthogonal projection. We have the following alternative: \\
{\bf Case 1.} For every non-zero projection $p \in \mathcal Z(P)$ and for every $t>0$ we have
\[\limsup_n \|e^{\perp}V_t\pi(p)\xi_n\|>\frac{\|p\|_2}{8C}.\]
We will prove that in this case the $M-M$ bimodule $\F_m$ is left $P$-amenable.
\begin{lemma} Let $X$ be the the strong operator topology completion of $\mathcal F_m$ as a right $M$-module with respect to the $M$-valued inner product $\lan x,y\ran=E_M(x^*y), x,y \in \mathcal F_m$. Let $\mathcal L(X)$ be the von Neumann algebra of adjointable operators on $X$. Then there exists a normal $^*$-homomorphism $\Psi:\mathcal L(X)\to B(L^2(\F_m^P))$ such that $\Psi(\mathcal L(X))\subset B(L^2(\F_m^P))\cap (\N^{op})' \cap (\theta(P^{op}))'$.
\end{lemma}
\begin{proof} The condition \eqref{L2space}iii) implies that $\F_m^{P}=X\ten_{M}L^2(\N)$, where the left action on $\N$ is that of $\pi(M)$. Therefore the map $\Psi:\mathcal L(X) \to B(\F_m^P)$ given by
\[\mathcal L(X) \ni T \mapsto T \ten_{M} id \in B(\F_m^P)\] 
is a well-defined normal $*$-homomorphism. Let us consider a rank one operator
$\xi \ten \bar{\eta} \in \mathcal L(X)$ with $\xi,\eta$ Wick words in $\tilde M$. Then we calculate 
 \[\Psi(\xi \ten \bar{\eta})(\pi(x_{\si})\theta(y^{op})\xi_{0})
 \lel \pi(\xi)\pi(E_M(\eta^*x_{\si}))\theta(y^{op})\xi_{0} \pl .\] 
Let $e_{\N}$ be the orthogonal projection of $\tilde{\mathcal N}$ onto the closure of $\N\xi_0$, which exists thanks to the fact that $E_{B_P}^{\tilde{\mathcal N}}$ is faithful, see Remark \ref{apost}. Then we note that for $\tilde{x}_{\tilde{\si}}\in M$ we have 
 \begin{align*}
 &\lan \pi(E_M(\eta^*x_{\si})\theta(y^{op})\xi_{0},\tilde{x}_{\tilde{\si}}\tilde{y}^{op}\xi_{0}\ran
 \lel  \psi(\theta(\tilde{y}^{op})^*\theta(y^{op})E^{\tilde{\mathcal N}}_{B_P}(\tilde{x}_{\tilde{\si}}^*x_{\si})) \\
 &=   \psi(\theta(\tilde{y}^{op})^*\theta(y^{op})(E_{\pi(B)}^{\pi(M)} \circ E_{\pi(M)}^{\tilde{\mathcal N}}) (\tilde{x}_{\tilde{\si}}^*x_{\si})) \lel 
\lan \pi(E_M(\eta^*x_{\si}))\theta(y^{op})\xi_{0},\pi(\tilde{x}_{\tilde{\si}})\theta(\tilde{y}^{op})\xi_{0} \ran
 \pl .
 \end{align*} 
This shows that  
 \[\pi(E_M(\eta^*x_{\si}))\theta(y^{op})\xi_{0}\lel e_{\N}(\pi(\eta^*x_{\si})\theta(y^{op})\xi_{0}) \pl .\] 
Thus we deduce that for all the rank one operators $\xi \ten \bar{\eta} \in \mathcal L(X)$ 
 \[ \Psi(\xi \ten \bar{\eta}) \lel L_{\pi(\xi)} e_{\N} L_{\pi(\eta^*)} \] 
is a right $\N$-module map, hence belongs to $B(L^2(\F_m^P))\cap (\N^{op})'$. It's also trivial to check that $\Psi(\xi \ten \bar{\eta})$ commutes with the operators $L_{\theta(y^{op})}$, for all $y \in P$. Since $\Psi$ is normal and the linear span of the rank one operators is so-dense in $\mathcal L(X)$ we have $\Psi(\mathcal L(X)) \subset B(L^2(\F_m^P))\cap (\N^{op})' \cap (\theta(P^{op}))'$, as desired. \qd 
The lemma provides a normal *-homomorphism
\[\Psi:B(\mathcal{F}_m)\cap (M^{op})'\to B(\mathcal{F}_m^P)\cap (\theta(P^{op})\vee \tilde{\mathcal J}\pi(M)\tilde{\mathcal J}\vee \tilde{\mathcal J}\theta(P^{op})\tilde{\mathcal J})'\]
such that $\Psi(\lambda(x))=\pi(x)$ for $x \in M$, where $\lambda$ is the natural left action of $M$ on $L^2(\tilde M)$. From this point on, the proof proceeds verbatim as in \cite{PoVaI}, proof of Case 1 in Thm. 3.1.\\
{\bf Case 2.} There exist a non-zero central projection $p \in \mathcal Z(P)$ and $t>0$ such that
\[\limsup_n \|e^{\perp}V_t\pi(p)\xi_n\| \leq \frac{\|p\|_2}{8C}.\]
In this case we prove that there exist $s,\delta>0$ such $\|T_s(a)\|_2\geq \delta$ for all $a \in \mathcal U(\mathcal A)$. Write $\pi(p)\xi_n=\zeta_n +\eta_n$, where $\zeta_n \in \bigoplus_{k \leq m} L^2_k(\mathcal N)$, $\eta_n \in \bigoplus_{k \geq m+1} L^2_k(\mathcal N)$. Note that $\|\zeta_n\| \leq 1, \|\eta_n\| \leq 1$. Since $V_t$ converges uniformly on $(\bigoplus_{k \leq m}L^2_k(\mathcal N))_1$, there exists a $t_0>0$ such that for $0<s<t_0$ we have
\[\|V_s\xi -\xi\| \leq \mbox{min}\{\frac{\|p\|_2}{8},\frac{\|p\|_2}{8C}\} \quad for \quad \xi \in(\bigoplus_{k \leq m} L^2_k(\mathcal N))_1.\]
Fix $0<s<$min$\{t^{m+1},t_0,t_0^{m+1},2^{-(m+1)^2}\}$. For every $n \geq 1$ we have the following estimate:
\begin{align*}
& \|V_s\pi(p)\xi_n-\pi(p)\xi_n\| \leq \|V_s\zeta_n-\zeta_n\|+\|V_s\eta_n-\eta_n\|\leq \frac{\|p\|_2}{8}+\|V_s\eta_n-\eta_n\|\leq \\ 
& \frac{\|p\|_2}{8}+C\|e^{\perp}V_{\sqrt[m+1]{s}}\eta_n\|\leq \frac{\|p\|_2}{8}+C\|e^{\perp}V_{\sqrt[m+1]{s}}\pi(p)\xi_n\|+C\|e^{\perp}V_{\sqrt[m+1]{s}}\zeta_n\| \leq \\
& \frac{\|p\|_2}{8}+C\|e^{\perp}V_{\sqrt[m+1]{s}}\pi(p)\xi_n\|+C\|e^{\perp}(V_{\sqrt[m+1]{s}}\zeta_n-\zeta_n)\|+C\|e^{\perp}\zeta_n\| \leq \\
& \frac{\|p\|_2}{8}+C\|e^{\perp}V_{\sqrt[m+1]{s}}\pi(p)\xi_n\|+C\|V_{\sqrt[m+1]{s}}\zeta_n-\zeta_n\| \leq \frac{\|p\|_2}{4}+C\|e^{\perp}V_{t}\pi(p)\xi_n\|.
\end{align*}
Taking limsup with respect to $n$ we obtain
\[\limsup_n \|V_s\pi(p)\xi_n-\pi(p)\xi_n\|\leq \frac{3\|p\|_2}{8}.\]
From this point on, the proof proceeds verbatim as in \cite{PoVaI}, proof of Case 2 in Thm. 3.1.
\end{proof}

\noindent {\bf Proof of Theorem \ref{main1}.} For the first alternative, we use (2) in Theorem 7.2, Proposition 3.20 and Proposition 2.3. For the second alternative, we use (1) in Theorem 7.2, Corollary 6.10 and Remark 2.7.
 $\hfill\square$

\noindent {\bf Proof of Corollaries \ref{maincor1}, \ref{maincor2}, \ref{maincor3}.} Immediate from Theorem \ref{main1}.
 $\hfill\square$

\bibliographystyle{amsplain}

\bibliographystyle{amsplain}
\bibliography{thebibliography}
\end{document}